\documentclass[12pt,letter,usenames,dvipsnames,twoside]{article}
\usepackage{fullpage}
    \usepackage[british]{babel}
    \usepackage[utf8]{inputenc}
    \usepackage[T1]{fontenc}
    \usepackage{indentfirst}

    \usepackage{enumerate}
    \usepackage{multirow}

    \usepackage{fancyhdr}
    \usepackage{lastpage}
    \usepackage{layout}
    \usepackage{amsmath,amsthm,amsfonts,amssymb,mathrsfs,mathtools}
    \usepackage{latexsym}
    \usepackage{bm}

    \usepackage{graphicx}
    \usepackage{caption}
    \usepackage{subfigure}

    \usepackage{microtype}

    \usepackage{xcolor}
    \usepackage{hyperref}
    \hypersetup{colorlinks,%
		linkcolor=RoyalBlue,%
        citecolor=Violet,
        urlcolor=BrickRed,
        bookmarksopen=true 
        }

    \usepackage{bookmark}
    \usepackage[nottoc]{tocbibind}

    \usepackage{url}
    \usepackage{doi}
    \usepackage[
        style = philosophy-modern,
        square = true,
        url = false,
        isbn = false,
        doi = true,
        hyperref=true,
        natbib = true,
        parentracker=true,
        shorthandintro = true,
        backend=biber]{biblatex}
    \usepackage{cleveref}
    \crefname{equation}{Eq.\@}{Eqs.\@}

    \DeclareFieldFormat{doi}{%
    \newline
    \mkbibacro{DOI}\addcolon\space
    \ifhyperref
      {\href{https://doi.org/#1}{\nolinkurl{#1}}}
      {\nolinkurl{#1}}}

    \DeclareFieldFormat[article,inbook,incollection,inproceedings,patent,thesis,unpublished]{titlecase}{\MakeSentenceCase*{#1}}
    \usepackage{titlecaps}
    \Addlcwords{on and for a is but with of in as the etc to if its de l'Institut von}
    \DeclareFieldFormat[phdthesis,book]{titlecase}{\itshape\titlecap{#1}}
    \DeclareFieldFormat{jtu}{\itshape\titlecap{#1}}
    \newbibmacro*{journal}{%
      \iffieldundef{journaltitle}
        {}
        {\printtext[journaltitle]{%
           \printfield[jtu]{journaltitle}%
           \setunit{\subtitlepunct}%
           \printfield[jtu]{journalsubtitle}}}}
    \DeclareFieldFormat[inproceedings]{booktitle}{
        \printfield[jtu]{booktitle}%
           \setunit{\subtitlepunct}%
           \printfield[jtu]{booksubtitle}
        }

    \NewBibliographyString{artno}
    \DefineBibliographyStrings{english}{artno = {article}}
    \DeclareFieldFormat[article,periodical]{eid}{\bibstring{artno}\addabbrvspace #1}

    \DefineBibliographyStrings{english}{jourvol = {vol\adddot}}
    \DefineBibliographyStrings{english}{number = {issue}}

    \DeclareFieldFormat[article]{volume}{\bibstring{jourvol}\addnbspace #1}
    \DeclareFieldFormat[article]{number}{\bibstring{number}\addnbspace #1}
    \renewbibmacro*{volume+number+eid}{%
      \printfield{volume}%
      \setunit{\addcomma\space}
      \printfield{number}%
      \setunit{\addcomma\space}%
      \printfield{eid}}

    \usepackage{cases}
    \newtheorem{theorem}{Theorem}
    \newtheorem{example}[theorem]{Example}
    \newtheorem{proposition}[theorem]{Proposition}
    \newtheorem{lemma}[theorem]{Lemma}
    \newtheorem{corollary}[theorem]{Corollary}
    \newtheorem{definition}[theorem]{Definition}
    \theoremstyle{theorem}
    
    \theoremstyle{theorem}
    
    \theoremstyle{theorem}
    
    \theoremstyle{theorem}
    
    \theoremstyle{theorem}
    
    \crefname{hlm}{Lemma}{Lemmas}
    \theoremstyle{definition}
    
    \theoremstyle{remark}
    \newtheorem*{rmk}{Remark}
    \theoremstyle{remark}

    \usepackage{algorithm}
    \usepackage{algorithmic}
    \numberwithin{theorem}{section}

    \DeclareMathOperator{\E}{\mathbf{E}\mathopen{}}

    \newcommand\Esubbig [2]{\E_{#1}\mkern-1.5mu\bigl[#2\bigr]}
    
    \newcommand\Esubbigg [2]{\E_{#1}\mkern-1.5mu\biggl[#2\biggr]}
    
    \DeclarePairedDelimiter\abs{\lvert}{\rvert} 
    \DeclarePairedDelimiter\norm{\lVert}{\rVert} 

    \newcommand{\ue}{\mathrm{e}} 
    \newcommand{\NN}{\mathbb{N}}
    \newcommand{\ZN}{\mathbb{Z}}
    
    \newcommand{\RNS}{\mathbb{R}}
    \newcommand{\CN}{\mathbb{C}}

    \newcommand\all[1]{\forall #1\enspace} 
    \newcommand\st{\, ; \;} 
    \newcommand{\ind}{\boldsymbol{1}} 
    \DeclareMathOperator{\img}{img}  
    \newcommand{\ud}{\mathrm{d}} 
    \newcommand{\df}[1]{\,\mathrm{d}#1}

    \def\rdiam{{\mathcal R}_{\textup{diam}}} 
    \DeclareMathOperator{\D}{diam} 
    \DeclareMathOperator{\dist}{\mathsf{dist}} 
    \DeclareMathOperator{\vol}{\mathsf{wt}} 
    \def\cp{{\cal P}}
    \def\cl{{\cal L}}
    \def\cq{{\cal Q}}
    \def\cE{{\cal E}}
    \def\cH{{\cal H}}
    \newcommand{\R}{\mathscr R}
    \def\be{{\bf e}} 
    \def\dmx{d_{\mathrm{max}}}
    \def\dmi{d_{\mathrm{min}}}
    \def\dav{d_{\mathrm{avg}}}
    \def\kg{\mathscr{K}(G)}
    \def\trel{t_{\mathrm{rel}}} 
    \def\tiq{\tau_{\infty}(1/4)} 
    
\usepackage{yfonts} 
\usepackage{calligra} 

    \usepackage{metalogo} 

    \usepackage{datetime}
    \usepackage[super]{nth}
    \newcommand\tstamp{\thanks{Current version: \today.}}
    \usepackage{titling} 
    \pretitle{\begin{center}\LARGE\bf}
    \posttitle{\end{center}}
    \title{Sharp Bounds on Eigenvalues via Spectral Embedding Based on Signless Laplacians\tstamp}

    \preauthor{\vskip 1em\begin{center}
    \large \lineskip 0.5em%
    \begin{tabular}[t]{c}}
    \postauthor{\end{tabular}\par\end{center}}
    \usepackage{authblk} 
    \newcommand\corref[2]{\thanks{Email: \href{mailto:#1}{#1}. #2}}

    \author{Zhi-Feng Wei\corref{zfwei@iu.edu}{Supported by NSF grant DMS-1954086.}}

    \affil{Department of Mathematics, Indiana University,\\
    831 E.\@ 3rd St., Bloomington, IN 47405-7106, United States}

    \predate{}
    \date{}
    \postdate{}

    \providecommand{\keywords}[1]{\flushleft\textbf{\textit{Keywords:}} #1}  
    \providecommand{\amsc}[1]{\flushleft\textbf{\textit{MSC 2020:}} #1}  

\begin{document}
\pagenumbering{roman}
\setcounter{page}{1}
\maketitle\thispagestyle{empty}
\begin{abstract}
    Using spectral embedding based on the probabilistic 
    signless Laplacian, we obtain bounds on the spectrum of transition matrices on graphs. As a consequence, we bound return probabilities and the uniform mixing time of simple random walk on graphs. In addition, spectral embedding is used in this article to bound the spectrum of graph adjacency matrices. Our method is adapted from [\citeauthor{Lyons-2018}, \citeyear{Lyons-2018}].

    \amsc{05C81, 60J10, 05C50.}
\end{abstract}
\clearpage
\setcounter{tocdepth}{2}
\tableofcontents
\thispagestyle{empty}
\clearpage
\pagenumbering{arabic}
\setcounter{page}{1}

\section{Introduction}
        Spectral embedding is a popular tool in modern data clustering, as summarized in \autocite{Luxburg-2007}.
        Also, spectral embedding was exploited to study graphs. In particular, \autocite{Lyons-2018} introduced spectral embedding as a new tool in analyzing reversible Markov chains (random walks on graphs).
        For instance,
        \autocite[Theorem~4.9]{Lyons-2018} gave a sharp bound on return probabilities of lazy random walk. Here, the lazy random walk on a graph $G$ stays put at a vertex with probability $1/2$ and moves to a random uniform neighbor otherwise.
        \begin{theorem}[\cite{Lyons-2018}, Theorem~4.9]
         Let $G$ be a regular, simple, connected graph with $n$ vertices. For each vertex $x$ of $G$ and $t>0$,
            \begin{displaymath}
                0\leqslant p'_t(x,x)-\frac{1}{n}\leqslant \frac{13}{\sqrt{t}},
            \end{displaymath}
            where $p'_t(x,x)$ is the probability of returning to $x$ at step $t$ for the lazy random walk on $G$ starting from $x$.\qed
        \end{theorem}
        Now, a question arises naturally: if simple random walk, instead of lazy random walk, is considered, do we still have good bounds on return probabilities? Here, the simple random walk on a graph moves to a random neighbor uniformly.

        To get some feeling, we recall the method used in the proof of the above bound in \autocite{Lyons-2018}. Given a graph $G$, let $P$ be the transition matrix of the simple random walk on $G$; then $P'\coloneqq \frac{I+P}{2}$ is the transition matrix of the lazy random walk on $G$. We know that $P$ and $P'$ are symmetric operators on the space of functions on $V(G)$ square summable with respect to the degree sequence of $G$, so their spectra are both real. Denoting the spectra of $P$ and $P'$ as $\sigma(P)$ and $\sigma(P')$ respectively,
        we know that $\sigma(P')$ is closely related to $\sigma(P)$: $\lambda\in\sigma(P)$ if and only if $\frac{1+\lambda}{2}\in\sigma(P')$. Since $\sigma(P)$ is contained in the interval $[-1,1]$, $\sigma(P')$ is contained in $[0,1]$. Also, it is well known that the non-trivial (different from $1$) spectrum of $P'$ ``governs'' the ``convergence rate'' of lazy random walk: for instance, if the gap between the largest non-trivial eigenvalue of $P'$ and $1$ is large, then intuitively, the convergence will be faster. Thus, by the relation of $\sigma(P)$ and $\sigma(P')$, we only need to consider the spectrum of $P$ near $1$. Moreover, we notice that $\sigma(P)$ is closely related to the spectrum of the probabilistic Laplacian $\cl \coloneqq I-P$ corresponding to $P$. To prove the aforementioned bound in \autocite{Lyons-2018}, using the spectral embedding based on $\cl$, \autocite{Lyons-2018} first revealed upper bounds on the vertex spectral measure of $\cl$.
        It is known that return probabilities are determined by vertex spectral measures; therefore, bounds on return probabilities can be obtained from bounds on vertex spectral measures (see \autocite[Proposition~3.5]{Lyons-2018}).

        When it comes to simple random walk, the situation is different. Note that simple random walks on bipartite graphs have period two and simple random walks on non-bipartite graphs are aperiodic. For simplicity of presentation, we discuss non-bipartite case only in this introduction. To begin, we still have the intuition that the ``convergence rate'' of simple random walk is related to the non-trivial  spectrum of $P$: for instance, \autocite[Proposition~3]{Diaconis-1991} proved that the geometric convergence rate in total variation norm is determined by the maximum non-trivial eigenvalue \emph{in absolute value} when the random walk has finitely many states. Note that $\sigma(P)\subseteq[-1,1]$, but $\sigma(P)$ is not necessarily non-negative. Therefore, in order to deal with ``convergence'' of simple random walk, one also has to consider the negative spectrum of $P$. However, we usually study the spectral gap of the probabilistic Laplacian $\cl = I-P$, which is related to the spectrum of $P$ near $1$ but does not provide much information about the \emph{negative} spectrum of $P$.

        Our solution is to consider another operator: the probabilistic signless Laplacian operator $\cq\coloneqq I+P$. Obviously, the spectrum of $P$ is closely related to the spectrum of $\cq$ by a shift of 1 unit horizontally. The spectrum of $\cq$ is therefore real and non-negative. This brings us some convenience: we have more tools to deal with the operator $\cq$, its associated quadratic form for a start. 
        We will first consider the vertex spectral measure of $\cq$ using the spectral embedding based on $\cq$, so the negative spectrum of $P$ is bounded. Then we may proceed to
        get a bound on return probabilities of simple random walk. See \Cref{reg_return} for details.

    \subsection{Main Results}\label{ssec_results}
        To give an overview of our results, in this subsection, we constrain ourselves to the case of unweighted graphs. Some notation will be needed, which will be explained in more detail in subsequent sections. Consider a locally finite, simple, connected graph $G=(V,E)$. A vertex $x\in V$ has degree $d(x)$ in $G$. Write $\pi(x)\coloneqq \frac{d(x)}{2\abs{E}}$, which is $0$ when $G$ is an infinite graph. When $G$ is finite, we denote the eigenvalues of the transition matrix $P$ on $G$ as
            \begin{displaymath}
                -1 \leqslant  \lambda^P_{\mathrm{min}} = \lambda^P_1
                    \leqslant  \lambda^P_2 \leqslant \lambda^P_3 \leqslant \dotsb \leqslant \lambda^P_{n-1}
                    <\lambda^P_n = \lambda^P_{\mathrm{max}} = 1,
            \end{displaymath}
        where $n=\abs{V}$.
        When $G$ is finite, it is well known that $\lambda^P_1=-1$ if and only if $G$ is bipartite.
        However, for infinite graphs, it is more appropriate to consider their vertex spectral measures. For instance, denoting the probability of returning to $x$ at step $t$ as $p_t(x,x)$ for the simple random walk on $G$ starting at $x$, then we have
        \begin{displaymath}
            p_t(x,x)=\int_{[0,2]}(1-\lambda)^t\ud\mu_x(\lambda),
        \end{displaymath}
        where $\mu_x$ is the vertex spectral measure at $x$ of $\cl$. Denote the vertex spectral measure at $x$ of $\cq$ as $\mu_x^\cq$. It is shown in \Cref{spec_loc} that when $G$ is finite,
            \begin{displaymath}
                \sum_{x\in V}\mu_x^\cq(\delta)=\abs[\big]{\{j\st 1+\lambda_j^P\leqslant \delta\}}.
            \end{displaymath}
        This enables us to count eigenvalues of $P$ on the interval $[-1,-1+\delta]$. Therefore, we can get lower bounds on eigenvalues of $P$ from upper bounds of the vertex spectral measure of $\cq$.

        We first consider non-bipartite graphs. In fact, when the graph is non-bipartite, the simple random walk is aperiodic, so some troubles are avoided. Our result for simple random walk on regular graphs reads as follows.
        \begin{theorem}\label{reg_intro}
            For a regular, non-bipartite, simple, connected graph $G$, we have
            \begin{displaymath}
                \mu_x^\cq(\delta)\leqslant 10\sqrt{\delta},\qquad 0\leqslant \delta\leqslant 2,\,x\in V.
            \end{displaymath}
            For each $x\in V$, simple random walk on $G$ satisfies
            \begin{displaymath}\begin{split}
                0\leqslant  p_t(x,x)-\pi(x)&\leqslant \frac{18}{\sqrt{t}}  \qquad \text{for\quad}  t\equiv 0\bmod 2,\\
                \abs[\big]{p_t(x,x)-\pi(x)}&\leqslant \frac{9}{\sqrt{t}}  \qquad \text{for\quad}  t\equiv 1\bmod 2.
            \end{split}\end{displaymath}
            Furthermore, when $G$ is finite, for $1\leqslant k\leqslant n$, $\lambda^P_k\geqslant -1+\frac{k^2}{100n^2}$.
        \end{theorem}
        The above result is sharp by \Cref{sharp}. 
        \Cref{reg_intro} is interesting since the degree and size of the regular graph $G$ are not involved.
        In \Cref{reg_intro}, the first assertion about bounding vertex spectral measure follows from  \Cref{reg_return_spec} in the subsequent text; the second assertion about return probability bound follows from \Cref{reg_return}; and the last assertion is treated in \Cref{lbb_reg}.

         It would be helpful to briefly describe the mechanism of getting return probability bounds in this paper, in which \Cref{ibp,spec_bound} play essential roles. In fact, \Cref{ibp,spec_bound} reveal how the asymptotics of large-time return probabilities corresponds to the asympotitics of the spectral measures of $\cl$ and $\cq$ near $0$.
         For instance, when proving \Cref{reg_return}, we first obtain a bound on vertex spectral measures of $\cq$ in \Cref{reg_return_spec}, which in fact characterizes the spectrum of $P$ around $-1$; The first assertion of \autocite[Theorem 4.9]{Lyons-2018}, as a counterpart of \Cref{reg_return_spec}, characterizes the spectrum of $P$ around $1$; based on \Cref{reg_return_spec} and the first assertion of \autocite[Theorem 4.9]{Lyons-2018}, \Cref{ibp,spec_bound} will conveniently produce a bound on return probabilities. Here, \Cref{ibp,spec_bound} show how return probabilities of simple random walk are determined by the spectrum of $P$ around $1$ and $-1$. In fact, \Cref{spec_bound} is an extension of \autocite[Lemma 3.5]{Lyons-2018}: intuitively, \autocite[Lemma 3.5]{Lyons-2018} reveals how return probabilities of lazy random walk are determined by the spectrum of $P$ around $1$.

        The following proposition is for graphs satisfying volume growth conditions.
    \begin{proposition}\label{vgrow_intro} 
        Let $G$ be a non-bipartite, infinite, simple, connected graph. Suppose that for some vertex $x$ of $G$, there are constants $c > 0$ and $D \geqslant  1$ such that
        \begin{displaymath}
        \sum_{y\st\dist(x,y)\leqslant r}d(y)\geqslant c (r+1)^D
        \end{displaymath}
        for all
        $r \geqslant  0$, where $\dist$ is the distance on graph. Then for all $\delta\in(0,2)$,
        \begin{displaymath}\begin{split} 
            \mu^*_x(\delta) &\leqslant  C d(x) \delta^{D/(D+1)},\\
            \mu^\cq _x(\delta) &\leqslant  C d(x) \delta^{D/(D+1)},
        \end{split}\end{displaymath}
        where
        \begin{displaymath}
        C \coloneqq  \frac{(D+1)^2}{c^{1/(D+1)} D^{2D/(D+1)}}.
        \end{displaymath}
        Hence for all $t \geqslant  1$, simple random walk satisfies
        \begin{displaymath}\begin{split}
            p_t(x,x)
            &\leqslant
            2C' w(x) t^{-D/(D+1)}\qquad \text{for\quad}  t\equiv 0\bmod 2,\\
            p_t(x,x)
            &\leqslant
            C' w(x) t^{-D/(D+1)}
            \qquad \text{for\quad}  t\equiv 1\bmod 2,
        \end{split}\end{displaymath}
        where
        \begin{displaymath}
        C' \coloneqq  \frac{D+1}{c^{1/(D+1)} D^{(D-1)/(D+1)}} \Gamma\Big(\frac{D}{D+1}\Big).
        \end{displaymath}
    \end{proposition}
    \Cref{vgrow_intro} follows from \Cref{vgrow_pro} directly.

        Using the spectral embedding based on $\cq$, we also bound the vertex spectral measure of $\cq$ in another way and get bounds on the uniform mixing time.
        \begin{proposition}\label{mix_intro}
            For any non-bipartite, finite, simple, connected graph $G$, we have
                \begin{displaymath}
                    \mu_x^\cq(\delta) \leqslant \frac{d(x)\delta}{\kg},\qquad \delta\in[0,2),\,x\in V,
                \end{displaymath}
            where\/ $\kg$ is defined as
            \begin{displaymath}
                \min\biggl\{
                \frac{\sum_{(v,u)\in E}\abs{f(v)+f(u)}^2}{\max_{y\in V}\abs{f(y)}^2}
                \st
                \min_{y\in V}f(y)<0< \max_{y\in V}f(y)
                \biggr\}
            \end{displaymath}
            and satisfies $\kg\geqslant \frac{1}{\D(G)+1}$. Consequently, for\/ $1\leqslant k\leqslant n$,
            \begin{displaymath}
                \lambda^P_k
                \geqslant
                -1+\frac{k\kg}{\sum_{x\in V}d(x)}
                \geqslant
                -1+\frac{k}{\bigl(\D(G)+1\bigr)\sum_{x\in V}d(x)}.
            \end{displaymath}
        Furthermore, $\tiq\leqslant 8n^3$. If $G$ is also regular, then we have $\tiq\leqslant 24n^2$.
        \end{proposition}
        \Cref{mix_intro} follows from \Cref{mix_spec_prep,est}, and \Cref{Landau_imp,mixx}.

        Note that for a finite, simple, connected graph $G$, \autocite[Proposition~4.2]{Lyons-2018} implies
            \begin{displaymath}
                \lambda^P_{n-k}\leqslant 1-\frac{k}{\rdiam(G)\sum_{x\in V}d(x)},\qquad 0\leqslant k\leqslant n-1,
            \end{displaymath}
            where $\rdiam(G)$ is the resistance diameter of $G$.
            This bound combined with the lower bound on eigenvalues of $P$ in \Cref{mix_intro} improves \autocite{Landau-1981}, which asserted that each eigenvalue $\lambda$ of $P$ that is neither $1$ nor $-1$ satisfies
            \begin{displaymath}
                \abs{\lambda}\leqslant 1-\frac{1}{\bigl(\D(G)+1\bigr)n\dmx},
            \end{displaymath}
            where $\dmx$ is the maximum degree of $G$.

        The bound $\tiq= O(n^2)$ in \Cref{mix_intro} is sharp by the example of cycles: there is a constant $c>0$ such that for all odd number $n$, simple random walk on
        an $n$-cycle satisfies $\tiq\geqslant cn^2$ (see \autocite[Example~3.11]{Tetali-2005}). See \Cref{sec_mix} for more details.

        As a special class of regular graphs, vertex-transitive graphs are of interest, since they are intuitively ``homogeneous'' and especially well studied.
        \begin{theorem}\label{trans_intro}
            Let $G$ be a non-bipartite, simple, connected, vertex-transitive graph with degree $d<\infty$. For each $x\in V$, $c\in(0,1)$, and $\delta\in (0,2]$, we have
            \begin{displaymath}
                \mu^\cq_x(\delta)\leqslant
                \frac{1}{c^2N^\#\Bigl(\frac{\sqrt{1-c}}{\sqrt{d\delta}}\Bigr)},
            \end{displaymath}
            where $N^\#(r)$ denotes the number of vertices in a ball of radius $r$. In addition, if $G$ is finite,
            \begin{displaymath}
                \lambda^\cq_{\mathrm{min}}
                \geqslant
                \frac{2}{d}\Bigl(\sin\frac{\pi}{4\bigl(\D(G)+1\bigr)}\Bigr)^2.
            \end{displaymath}
        \end{theorem}
        \Cref{trans_intro} follows from \Cref{trans_spec,trans_mini} directly. As usual, a bound on return probabilities of simple random walk follows immediately from a bound on vertex spectral measures. See \Cref{trans_return} as an example.

        Apart from vertex spectral measure at one single vertex, we may consider average spectral measure for finite graphs. Given a finite graph $G$, the average spectral measure of $\cq$ is defined as $\mu^\cq\coloneqq \sum_{x\in V}\mu_x^\cq/n$.
        \begin{theorem}\label{avg_intro}
            For any non-bipartite, finite, simple, connected graph $G$ and
            $\delta\in(0,2)$, we have
            \begin{displaymath}
                \mu^\cq(\delta)<(4000\delta)^{1/3}.
            \end{displaymath}
            Consequently, $\lambda^P_k\geqslant -1+\frac{k^3}{4000n^3}$ for $1\leqslant k\leqslant n$.
            Furthermore,
            \begin{displaymath}\begin{split}
                0\leqslant
                \frac{\sum_{x\in V}p_t(x,x)-1}{n}
                &\leqslant
                \frac{30}{t^{1/3}}  \qquad \text{for\quad}  t\equiv 0\bmod 2,\\
                \frac{\abs[\big]{\sum_{x\in V}p_t(x,x)-1}}{n}
                &\leqslant
                \frac{15}{t^{1/3}}  \qquad \text{for\quad}  t\equiv 1\bmod 2.
            \end{split}\end{displaymath}
        \end{theorem}
       \Cref{avg_intro} is treated in \Cref{sec_avg}; see \Cref{spec_avg,avg_even,avg_odd}, and \Cref{spec_avg_ct} for more details.

       Non-bipartite graphs are considered in previous paragraphs. But the bipartite case is a bit different, because simple random walk on a bipartite graph has period two. Fortunately, by \autocite[Theorem~4.8]{Mohar-1989}, the vertex spectral measures of $\cl$ and $\cq$ coincide. Therefore, only the estimates for the vertex spectral measure of $\cl$ from \autocite{Lyons-2018} will be enough for us to get bounds on return probabilities of simple random walk. For example, we have the following result for bipartite graphs. 
        \begin{theorem}
            Consider simple random walk on a regular, bipartite, simple, connected graph $G$. Then for each $x\in V$,
                \begin{displaymath}
                    0\leqslant  p_t(x,x)-2\pi(x)\leqslant \frac{18}{\sqrt{t}} \qquad \text{for\quad} t\equiv 0\bmod 2.
                \end{displaymath}
        \end{theorem}
        This result is later proved as \Cref{reg_return_bi} in \Cref{sec_bip}.

        Our method works as well for analyzing the spectrum of the adjacency matrix $A$ of a finite graph $G$. Suppose the eigenvalues of $A$ are
        \begin{displaymath}
            -d_{\mathrm{max}} \leqslant \lambda^A_{\mathrm{min}} = \lambda^A_1
                \leqslant  \lambda^A_2 \leqslant \lambda^A_3 \leqslant \dotsb \leqslant \lambda^A_{n-1} < \lambda^A_n= \lambda^A_{\mathrm{max}} \leqslant  d_{\mathrm{max}},
        \end{displaymath}
        where $\dmx$ is the maximum degree of $G$.
        \begin{proposition}\label{Alon_ext_intro}
            Let $G$ be a non-bipartite, finite, simple, connected graph. For\/ $1\leqslant k\leqslant n$,
            \begin{displaymath}
                \dmx+\lambda_k^A
                \geqslant
                \frac{k\kg}{n}
                \geqslant
                \frac{k}{\bigl(\D(G)+1\bigr)n}.
            \end{displaymath}
        \end{proposition}
        \Cref{Alon_ext_intro} follows from \Cref{Alon_ext} directly. \Cref{Alon_ext_intro} improves \autocite[Theorem~1.1]{Alon-2000}, which obtained that $\dmx+\lambda_1^A  \geqslant \frac{1}{(\D(G)+1)n}$. To get \Cref{Alon_ext_intro}, in \Cref{sec_comb}, we first consider the spectral embedding based on the combinatorial signless Laplacian $\Theta\coloneqq D+A$, where $D$ is the diagonal degree matrix of $G$. Then \Cref{Alon_ext_intro} follows from the bound below on the vertex spectral measure of $\Theta$.
        \begin{proposition}
            Let $G$ be a non-bipartite, finite, simple, connected graph. Then for each
            $\delta\in [0,\lambda^\Theta_{\mathrm{max}})$ and $x\in V$, we have
            \begin{displaymath}
                \mu_x^\Theta(\delta)
                \leqslant
                \frac{\delta}{\kg}
                \leqslant
                \bigl(\D(G)+1\bigr)\delta.
            \end{displaymath}
        \end{proposition}
        This proposition follows from \Cref{comb_spec} directly. Suppose the eigenvalues of $\Theta$ are
        \begin{displaymath}
            0 \leqslant \lambda^\Theta_{\mathrm{min}} = \lambda^\Theta_1
                \leqslant  \lambda^\Theta_2 \leqslant \lambda^\Theta_3 \leqslant \dotsb \leqslant \lambda^\Theta_{n-1} < \lambda^\Theta_n.
        \end{displaymath}
        The above proposition has the following corollary.
        \begin{corollary}\label{com_s_lap_intro}
        Let $G$ be a non-bipartite, finite, simple, connected graph . For\/ $1\leqslant k\leqslant n$, we have
            \begin{displaymath}
                \lambda^\Theta_k
                \geqslant \frac{k\kg}{n}
                \geqslant\frac{k}{\bigl(\D(G)+1\bigr)n}.
            \end{displaymath}
        \end{corollary}
        \Cref{com_s_lap_intro} is proved as \Cref{com_s_lap} in \Cref{sec_comb}. In fact, it is known that a graph is bipartite if and only if $\lambda^\Theta_{\mathrm{min}}=0$; for a non-bipartite graph $G$, \autocite{Desai-1994} showed that $\lambda^\Theta_{\mathrm{min}}$ measures non-bipartiteness of $G$.
    \subsection{Related Works}
        In the field of spectral graph theory, the combinatorial signless Laplacian has already drawn wide attention: \autocite{Cvetkovic-2009,Cvetkovic-2010a,Cvetkovic-2010b} are surveys on the study of the combinatorial signless Laplacian. In fact, \autocite[Theorem~1.1]{Alon-2000} also used the combinatorial signless Laplacian implicitly by considering its associated quadratic form. We use the method of spectral embedding, so not only the minimum eigenvalue but all eigenvalues of the graph adjacency matrices are bounded from below in \Cref{Alon_ext_intro}.

        One highlight of this article is the introduction of the probabilistic signless Laplacian $\cq$, enabling us to deal with the negative spectrum of the transition matrix $P$ on a graph with the tool of spectral embedding, and therefore deal with return probabilities. Indeed, the quadratic form associated to $\cq$ was used in \autocite{Desai-1993,Diaconis-1991} to give lower bounds on the minimum eigenvalue of $P$. Signless Laplacian operator is also related to dual Cheeger inequalities; see \autocite{Trevisan-2012,Liu-2015} for more details.
        We consider the operator $\cq$ explicitly and exploit the spectral embedding based on it. Therefore, the entire spectrum of $P$ is treated, rather than only the minimum eigenvalue. For instance, \autocite{Landau-1981} is improved as we discussed after \Cref{mix_intro}.


        In the literature, there are many fewer results on the negative spectrum of $P$ than on the positive spectrum of $P$. Recall that in the study of the positive spectrum of $P$, the probabilistic Laplacian operator $\cl$ was usually used; in particular, the spectral gap of $\cl$ equals the gap between the largest nontrivial spectrum of $P$ and $1$. However, when one wants to study the negative spectrum of $P$, for instance by considering the signless probabilistic Laplacian operator $\cq$, it is harder: since we don't have many tools. For example, tools from electrical network theory are pretty useful in the study of $\cl$, but they are not readily available to deal with $\cq$. We have to make a ``detour'' and adapt the existing tools.

    \subsection{Structure of this Article}
        We review notation for graphs and introduce spectral embedding based on the signless Laplacian in
        \Cref{ssec_notation,ssec_se}. Some preliminaries are included in \Cref{ssec_prel}.
        Return probabilities of simple random walk on regular graphs are considered in \Cref{sec_reg}.
        Return probability bounds based on volume growth conditions are discussed in \Cref{sec_vol}. The case of transitive graphs is treated in \Cref{sec_trans}. We also consider average return probabilities for finite graphs in \Cref{sec_avg}. In \Cref{sec_mix}, we bound the uniform mixing time. Bipartite graphs are discussed in \Cref{sec_bip}. The tool of spectral embedding is exploited to study eigenvalues of graph adjacency matrices in \Cref{sec_comb}. The appendices contain some calculations and auxiliary results. 

    \section{Notation and Spectral Embedding}\label{sec_prep}
        \subsection{Graph Notation, Random Walk, and Laplacian Operators}\label{ssec_notation}
        Let $G=\bigl(V(G),E(G)\bigr)$ be a finite or infinite, undirected, simple, connected, weighted graph. For an edge of $G$, say $e=(x,y)\in E(G)$, let $w(e)=w(x,y)>0$ be its weight. We say $G$ is \emph{unweighted} if $w(e)=1$ for each edge $e\in E(G)$.
        We assume $G$ has weighted adjacency matrix $A(G)$, (unweighted) diameter $\D(G)$, and (weighted) resistance diameter $\rdiam(G)$. Also, minimum, maximum, and average degrees in $G$ are denoted by $\dmi(G)$, $\dmx(G)$, and $\dav(G)$, respectively.
        When $G$ is finite, we denote $\abs[\big]{V(G)}=n(G)$. If $G$ is understood, reference to $G$ may be omitted.

        For $x\in V(G)$, we use standard graph notation $N(x)\coloneqq \bigl\{y\in V(G)\st (x,y)\in E(G)\bigr\}$ to denote the collection of all neighbors of $x$. Throughout this article, we require that $G$ is \emph{locally finite}, i.e., $\sum_{y\in N(x)}w(x,y)<\infty$ for each $x\in V(G)$. We say that ${w(x)\coloneqq \sum_{y\in N(x)}w(x,y)}$ is the weight of $x\in V(G)$ in $G$, and the weight of a vertex subset $S\subseteq V(G)$ in $G$ is ${\vol(S;G)\coloneqq \sum_{x\in S}w(x)}$. If $G$ is unweighted, $w(x)$ equals the degree $d(x)$ of $x$. Again, when $G$ is understood, the reference to $G$ may be omitted. In particular, when $G$ is vertex-transitive, all vertices have the same weight, denoted by $w$. Write $\pi(x)= \frac{w(x)}{\vol(V(G))}$. For a vertex $x\in V(G)$ and $r\geqslant0$, let
        \begin{displaymath}
            B(x,r;G)=\bigl\{y\in V(G)\st \dist(x,y)\leqslant r\},
        \end{displaymath}
        where $\dist$ is the distance on $G$. Set
            \begin{displaymath}
                \vol(x,r;G)
                \coloneqq
                \vol\bigl(B(x,r;G);G\bigr).
            \end{displaymath}

        For the simple random walk on $G$, the transition probability from $x\in V(G)$ to $y\in V(G)$ is $\frac{w(x,y)}{w(x)}$. We use $p_t(x,y)$ to denote the probability that the simple random walk started at
        $x\in V(G)$ arrives at $y\in V(G)$ at step $t$.

        Recall that we write
        $\ell^2(V(G), w)$ for the (real or complex) Hilbert space of functions $f \colon V(G) \to \RNS \text{ or } \CN$ with inner product
        \begin{displaymath}
            \langle f,g \rangle_w \coloneqq  \sum_{x \in V(G)} w(x) f(x)\overline{g(x)}
        \end{displaymath}
        and squared norm $\norm{f}_w^2 \coloneqq  \langle f,f\rangle_w$.
        We reserve $\langle \cdot , \cdot \rangle$ and $\|\cdot\|$ for the standard inner product and norm on $\RNS^k\,(k \in \NN)$ and $\ell^2\bigl(V(G)\bigr)$. 
        For a vertex $x\in V(G)$, we use $\ind_x$ to denote the indicator vector of $x$:
                    \begin{numcases}
                        {\ind_x(y)\coloneqq }
                            1
                            \quad \nonumber &if\/ $y=x$,\\
                            0
                            \quad \nonumber &otherwise.
                    \end{numcases}
        We also write $\be_x \coloneqq  \frac{\ind_x}{\sqrt{w(x)}}$. Note that $\be_x\in \ell^2(V(G), w)$ is of unit norm: $\norm{\be_x}_w=1$.

        We have a series of useful operators on $\ell^2(V(G), w)$. The probability transition operator $P\colon \ell^2(V(G), w)\to \ell^2(V(G), w)$ is defined as
        \begin{displaymath}
            (Pf)(x)\coloneqq \sum_{y\in V(G)}\frac{w(x,y)}{w(x)}f(y).
        \end{displaymath}
        We define the probabilistic Laplacian $\cl = I-P$ and the probabilistic signless Laplacian $\cq=I+P$, where $I$ is the identity operator on $\ell^2(V(G), w)$.
        We know that $P$, $\cl$, and $\cq$ are bounded self-adjoint operators on Hilbert space $\ell^2(V(G), w)$. The spectrum of $P$ is contained in the interval $[-1,1]$.
        Obviously, whether $G$ is finite or infinite, the spectrum of $P$ and $\cq$ are related by a shift of 1 unit horizontally.

        Denote the resolution of identity for $\cl$ as $I_\cl$.
        As in \autocite{Lyons-2018}, the vertex spectral measure of $\cl$ at $x\in V(G)$ is defined by
            \begin{displaymath}
                \mu_x(\delta)
                \coloneqq \bigl\langle I_\cl\bigl([0,\delta]\bigr)\be_x,\be_x\bigr\rangle_w
                =\bigl\langle I_\cl\bigl([0,\delta]\bigr)\ind_x,\ind_x\bigr\rangle,
                \qquad 0\leqslant \delta\leqslant 2.
            \end{displaymath}
        For convenience, we will also use
            \begin{displaymath}
                \mu_x^*(\delta)
                \coloneqq \bigl\langle I_\cl\bigl((0,\delta]\bigr)\be_x,\be_x\bigr\rangle_w
                =\bigl\langle I_\cl\bigl((0,\delta]\bigr)\ind_x,\ind_x\bigr\rangle,
                \qquad 0\leqslant \delta\leqslant 2,\,x\in V(G).
            \end{displaymath}
        It is easy to see that $\mu_x(\delta)=\mu_x^*(\delta)+\pi(x)$ for $0\leqslant \delta\leqslant 2$
        (see \autocite[Section~3.1]{Lyons-2018}).

        When $G$ is finite, the spectra of $P$ and $\cq$ consist of eigenvalues (point spectrum) only:
        \begin{itemize}
          \item Denote the eigenvalues of the transition matrix $P$ on $G$ as
            \begin{displaymath}
                -1 \leqslant  \lambda^P_{\mathrm{min}} = \lambda^P_1
                    \leqslant  \lambda^P_2 \leqslant \lambda^P_3 \leqslant \dotsb \leqslant \lambda^P_{n-1}
                    <\lambda^P_n = \lambda^P_{\mathrm{max}} =   1.
            \end{displaymath}
          \item Denote the eigenvalues of the probabilistic signless Laplacian $\cq$ as
            \begin{displaymath}
                0\leqslant  \lambda^\cq_{\mathrm{min}} = \lambda^\cq_1
                \leqslant  \lambda^\cq_2 \leqslant \lambda^\cq_3 \leqslant \dotsb \leqslant \lambda^\cq_{n-1}
                 < \lambda^\cq_n = \lambda^\cq_{\mathrm{max}} =   2.
            \end{displaymath}
        \end{itemize}
    \subsection{Spectral Embedding Based on the Signless Laplacian}\label{ssec_se}
        The spectral embedding based on $\cl$ is introduced in \autocite[Section~3.4]{Lyons-2018} as a powerful tool in analyzing random walk on graphs. In this subsection, we will introduce the spectral embedding based on $\cq$ in parallel. Denote the resolution of identity for $\cq$ as $I_\cq$, with vertex spectral measure at $x$
            \begin{displaymath}
                \mu_x^\cq(\delta)\coloneqq \bigl\langle I_\cq(\delta)\be_x,\be_x\bigr\rangle_w,\qquad 0\leqslant \delta\leqslant 2,\,x\in V(G),
            \end{displaymath}
        where $I_\cq(\delta)\coloneqq I_\cq\bigl([0,\delta]\bigr)$.
        \begin{lemma}\label{q_form}
            Let $f\in \ell^2(V(G), w)$. We have
                \begin{displaymath}
                    \langle\cq f,f\rangle_w=\sum_{(v,u)\in E(G)}w(v,u)\abs[\big]{f(v)+f(u)}^2.
                \end{displaymath}
            For $\delta\in[0,2]$ and $f\in\img\bigl(I_\cq(\delta)\bigr)$,
            \begin{displaymath}
                \pushQED{\qed}
                \langle\cq f,f\rangle_w\leqslant \delta\norm{f}_w^2=\delta\sum_{v\in V}w(v)\abs[\big]{f(v)}^2.\qedhere
                \popQED
            \end{displaymath}
        \end{lemma}
        See the appendix for a proof.
        \begin{corollary}\label{0span}
            If $G$ is connected and non-bipartite, then\/ $\img\bigl(I_\cq(\{0\})\bigr)$ contains only the zero function. Therefore in this case, $I_\cq(\delta)=I_\cq\bigl([0,\delta]\bigr)=I_\cq\bigl((0,\delta]\bigr)$. It also follows that\/ $0$ is not an eigenvalue of $\cq$ if $G$ is non-bipartite.
        \end{corollary}
        \begin{proof}
            Assume $f\in\img\bigl(I_\cq(\{0\})\bigr)$. We see that
            \begin{displaymath}
                \sum_{(v,u)\in E(G)}w(v,u)\abs[\big]{f(v)+f(u)}^2
                =\langle\cq f,f\rangle_w
                \leqslant 0\norm{f}_w^2=0.
            \end{displaymath}
            Therefore, for $(v,u)\in E(G)$, we have $f(v)=-f(u)$.
            Since $G$ is non-bipartite, there is an odd cycle $C=v_1v_2\cdots v_sv_1$, where $s\equiv 1\bmod 2$. Thus, $f(v_1)=-f(v_1)$, implying that $f(v_1)=0$. By connectedness, $f(v)=0$ for all $v\in V(G)$. This corollary is proved.
        \end{proof}
        For a fixed $\delta\in[0,2]$, we define spectral embedding $F^\cq$ based on $\cq$ as
            \begin{displaymath}\begin{split}
                F^\cq\colon V(G)&\rightarrow \ell^2(V(G), w)\\
                x&\mapsto F^\cq_x\coloneqq \frac{I_\cq(\delta)\be_x}{\sqrt{w(x)}}=\frac{I_\cq(\delta)\ind_x}{w(x)}.
            \end{split}\end{displaymath}
        It is clear that for each $x\in V(G)$, $F^\cq_x\in\ell^2(V(G), w)$ is a real-valued function on $V(G)$.
        \begin{lemma}\label{spec_embed_norm}
            For each finite or infinite graph $G$ and $x\in V(G)$,
                \begin{displaymath}
                    \pushQED{\qed}
                    \norm[\big]{F^\cq_x}_w^2=F^\cq_x(x)=\frac{\mu_x^\cq(\delta)}{w(x)}.\qedhere
                    \popQED
                \end{displaymath}
        \end{lemma}
        \begin{lemma}\label{def_f}
            If $\mu_x^\cq(\delta)>0$, define $f\colon V(G)\rightarrow\CN$ as $f\coloneqq \frac{F^\cq_x}{\norm{F^\cq_x}_w}$. Then
            \begin{enumerate}[\upshape(1)]
              \item $\norm{f}_w=1$;
              \item $f(x)=\sqrt{\mu_x^\cq(\delta)/w(x)}$;
              \item $f\in\img\bigl(I_\cq(\delta)\bigr)$.\hfill\qedsymbol
            \end{enumerate}
        \end{lemma}
        We need only mimic the proofs of \autocite[Lemmas~3.11 and~3.12]{Lyons-2018} to prove {\Cref{spec_embed_norm,def_f}}.
    \subsection{Some Preliminaries}\label{ssec_prel}
        We will use \Cref{reg_diam}, a standard path fact, which was proved  in \autocite[Proposition 10.16(b)]{Levin-2017} and \autocite[Lemma 4.5]{Lyons-2018}. Bounds on the diameter of regular graphs go back to \autocite{Moon-1965}, but a different approach was used there.
            \begin{lemma}\label{reg_diam}
                Let $G$ be a finite, simple, connected graph. We have\/ $\D(G)\leqslant \frac{3n}{\dmi}-1$.\hfill\qedsymbol
            \end{lemma}

            The following two results will be useful when we are dealing with return probabilities in subsequent sections.
            \begin{lemma}\label{ibp}
                Let $t$ be a positive integer and $\eta$ be an increasing and right-continuous function on\/ $[0,2]$ with $\eta(0)=0$. Then we have
                \begin{displaymath}
                    t\int_{0}^{2}\eta(\lambda)(1-\lambda)^{t-1}\df{\lambda}
                    =
                    (-1)^{t+1}\eta(2)+\int_{(0,2]}(1-\lambda)^t\df{\eta(\lambda)}.
                \end{displaymath}
            \end{lemma}
            \begin{proof}
                Using integration by parts, we have
                \begin{displaymath}\begin{split}
                    t\int_{0}^{2}\eta(\lambda)(1-\lambda)^{t-1}\df{\lambda}
                        &=-\int_{0}^{2}\eta(\lambda)\df{(1-\lambda)^t}\\
                        &=-\left.\eta(\lambda)(1-\lambda)^t\right|^2_0+\int_{(0,2]}(1-\lambda)^t\df{\eta(\lambda)}\\
                        &=(-1)^{t+1}\eta(2)+\int_{(0,2]}(1-\lambda)^t\df{\eta(\lambda)}.\qedhere
                \end{split}\end{displaymath}
            \end{proof}
            \begin{lemma}\label{spec_bound}
                Consider simple random walk on a non-bipartite, simple, connected, weighted graph $G$.
                \begin{enumerate}[\upshape(1)]
                  \item
                    Let $\varphi$ be an increasing and right-continuous function on\/ $[0,2]$. Assume further that $\varphi$ satisfies the following conditions:
                        \begin{displaymath}\begin{split}
                            &\varphi(0)=0=\mu_x^*(0),\qquad \varphi(2)=1-\pi(x)=\mu_x^*(2),\\
                            &\mu_x^*(\lambda)\leqslant \varphi(\lambda) \qquad \text{for\quad} \lambda\in[0,1),\\
                            &\mu_x^*(\lambda)\geqslant \varphi(\lambda) \qquad \text{for\quad} \lambda\in[1,2].
                        \end{split}\end{displaymath}
                    Then for $t\equiv 0\bmod 2$,
                        \begin{displaymath}
                            \pi(x)\leqslant  p_t(x,x)\leqslant \pi(x)+\int_{(0,2]}(1-\lambda)^t\df{\varphi(\lambda)}.
                        \end{displaymath}
                  \item Let $\psi_1$ and $\psi_2$ be increasing and right-continuous functions on\/ $[0,2]$. Assume further that $\psi_1$ and $\psi_2$ satisfy the following conditions:
                        \begin{displaymath}\begin{split}
                            &\psi_1(0)=\psi_2(0)=0=\mu_x^*(0),\\
                            &\psi_1(2)=\psi_2(2)=1-\pi(x)=\mu_x^*(2),\\
                            &\mu_x^*(\lambda)\leqslant \psi_2(\lambda) \qquad \text{for\quad} \lambda\in[0,2],\\
                            &\mu_x^*(\lambda)\geqslant \psi_1(\lambda) \qquad \text{for\quad} \lambda\in[0,2].
                        \end{split}\end{displaymath}
                      Then for $t\equiv 1\bmod 2$, we have
                        \begin{displaymath}
                            \pi(x)+\int_{(0,2]}(1-\lambda)^t\,\ud\psi_1(\lambda)
                            \leqslant
                            p_t(x,x)
                            \leqslant \pi(x)+\int_{(0,2]}(1-\lambda)^t\,\ud\psi_2(\lambda).
                        \end{displaymath}
                \end{enumerate}
            \end{lemma}
            \begin{proof}
                a) This part is essentially a mimic of the proof of \autocite[Lemma~3.5]{Lyons-2018}. We consider a non-negative integer $t$ in this part.
                Since $P=I-\cl$, we have
                    \begin{displaymath}
                        p_t(x,x)=\bigl\langle(I-\cl)^t\ind_x,\ind_x\bigr\rangle.
                    \end{displaymath}
                Symbolic calculus gives
                    \begin{displaymath}
                        (I-\cl)^t=\int_{[0,2]}(1-\lambda)^t\,I_\cl(\ud\lambda).
                    \end{displaymath}
                Therefore, by the definition of the vertex spectral measure of $\cl$, we have
                    \begin{displaymath}
                        p_t(x,x)=\int_{[0,2]}(1-\lambda)^t\df{\bigl\langle I_\cl(\lambda)\ind_x,\ind_x\bigr\rangle}
                        =\int_{[0,2]}(1-\lambda)^t\df{\mu_x(\lambda)}.
                    \end{displaymath}
                It follows that
                    \begin{equation}\label{rep}
                        p_t(x,x)=\pi(x)+\int_{[0,2]}(1-\lambda)^t\df{\mu^*_x(\lambda)}.
                    \end{equation}
                Furthermore, using integration by parts, we have
                    \begin{displaymath}\begin{split}
                        p_t(x,x)&=\pi(x)+\int_{[0,2]}(1-\lambda)^t\df{\mu^*_x(\lambda)}\\
                        &=\pi(x)+\left.(1-\lambda)^t\mu^*_x(\lambda)\right|^2_0-\int_{0}^{2}\,\mu^*_x(\lambda)\df{(1-\lambda)^t}\\
                        &=\pi(x)+(-1)^t\mu^*_x(2)+t\int_{0}^{2}\mu^*_x(\lambda)(1-\lambda)^{t-1}\df{\lambda}.
                    \end{split}\end{displaymath}
            \par
                b) When $t\equiv 0\bmod 2$, because $(1-\lambda)^t$ is always non-negative, by \Cref{rep}, $p_t(x,x)\geqslant \pi(x)$. On the other hand, by the result in part a),
                    \begin{displaymath}\begin{split}
                        p_t(x,x)&=\pi(x)+\bigl(1-\pi(x)\bigr)+t\int_{0}^{2}\mu^*_x(\lambda)(1-\lambda)^{t-1}\df{\lambda}\\
                        &\leqslant 1+t\int_{0}^{2}\varphi(\lambda)(1-\lambda)^{t-1}\df{\lambda}\\
                        &= 1-\bigl(1-\pi(x)\bigr)+\int_{(0,2]}(1-\lambda)^t\,\ud\varphi(\lambda)\\
                        &=\pi(x)+\int_{(0,2]}(1-\lambda)^t\,\ud\varphi(\lambda),
                    \end{split}\end{displaymath}
                where the second equality follows from \Cref{ibp}. The first assertion is proved.
            \par
                c) When $t\equiv 1\bmod 2$, by the result in part a), we have
                    \begin{displaymath}
                        p_t(x,x)=2\pi(x)-1+t\int_{0}^{2}\mu^*_x(\lambda)(1-\lambda)^{t-1}\df{\lambda}.
                    \end{displaymath}
                Therefore, by \Cref{ibp},
                    \begin{displaymath}\begin{split}
                        p_t(x,x)&\geqslant 2\pi(x)-1+\bigl(1-\pi(x)\bigr)+t\int_{0}^{2}\psi_1(\lambda)(1-\lambda)^{t-1}\df{\lambda}\\
                        &=\pi(x)+\int_{(0,2]}(1-\lambda)^t\,\ud\psi_1(\lambda).
                    \end{split}\end{displaymath}
                Similarly,
                    \begin{displaymath}
                        p_t(x,x)
                        \leqslant
                        \pi(x)+\int_{(0,2]}(1-\lambda)^t\,\ud\psi_2(\lambda).\qedhere
                    \end{displaymath}
            \end{proof} 

\clearpage
\section{Return Probability on Regular Graphs}\label{sec_reg}
    In this section, we are mainly interested in regular graphs.
    \subsection{Estimate of Spectral Measure and Convergence Rate}\label{reg}
        \begin{theorem}\label{reg_return_spec}
            Let $G$ be a non-bipartite, connected, regular, unweighted, simple graph. For each $x\in V(G)$, we have
            \begin{displaymath}
                \mu_x^\cq(\delta)\leqslant 10\sqrt{\delta},\qquad 0\leqslant \delta\leqslant 2.
            \end{displaymath}
        \end{theorem}
        \Cref{reg_return_spec} is parallel to the first assertion of \autocite[Theorem 4.9]{Lyons-2018}. Rather than prove \Cref{reg_return_spec} directly, we will show the following more general result. Note that when $G$ is regular, $\tfrac{d(x)}{\dmi}=1$ for each $x\in V(G)$.
        \begin{proposition}\label{irreg_return_spec}
            Let $G$ be a non-bipartite, connected, unweighted, simple graph. For each $x\in V(G)$, we have
            \begin{displaymath}
                \mu_x^\cq(\delta)\leqslant \frac{10 d(x)\sqrt{\delta}}{\dmi},\qquad 0\leqslant \delta\leqslant 2.
            \end{displaymath}
        \end{proposition}
        \begin{proof}
            Fixing a vertex $x\in V(G)$, we define $f$ as in \Cref{def_f}. Recall that we denote $\dmi=\min\bigl\{d(v)\st v\in V(G)\bigr\}$.
            \par
            a) If $f(y)\geqslant 0$ for all $y\in V(G)$, then
                \begin{displaymath}\begin{split}
                    \frac{10 d(x)\sqrt{\delta}}{\dmi}&\geqslant \delta\geqslant \langle\cq f,f\rangle_w=\sum_{(v,u)\in E(G)}w(v,u)\abs[\big]{f(v)+f(u)}^2\\
                    &\geqslant
                    \sum_{y\in N(x)}w(x,y)\abs[\big]{f(x)+f(y)}^2
                    \geqslant
                    w(x)\abs[\big]{f(x)}^2\\
                    &=\mu_x^\cq(\delta).
                \end{split}\end{displaymath}
            \par
            b) We may now assume without loss of generality that $\bigl\{v\in V(G)\st f(v)<0\bigr\}$ is non-empty. Let $S\coloneqq \bigl\{s\in V(G)\st f(s)>0\bigr\}$ and $T\coloneqq \bigl\{t\in V(G)\st f(t)<0\bigr\}$. By our assumptions, both $S$ and $T$ are non-empty. Recall that for each edge $e\in E(G)$, we have $w(e)=1$ because $G$ is assumed to be unweighted. However, we do the following construction of an auxiliary graph $G'$ for general weighted graphs because this will also be useful subsequently in the proof of \Cref{vgrow_spec}:
                \begin{enumerate}[\hspace{3em}(1)]
                  \item The vertex set $V'$ of $G'$ includes $V(G)$. Also, if $(u,v)=e\in E(G)$ and $u,v\in S$, we introduce two vertices $u^{(e)}$ and $v^{(e)}$ in $V'$. Similarly, if $(u,v)=e\in E(G)$ and $u,v\in T$, we introduce two vertices $u^{(e)}$ and $v^{(e)}$ in $V'$.
                  \item Construct the edge set $E'$ of $G'$ and their weights: If $e=(u,v)\in E(G)$ with $u$ and $v$ both in $S$, we introduce three edges $(u,u^{(e)})$, $(u^{(e)},v^{(e)})$, and $(v^{(e)},v)$ in $E'$; if $e=(u,v)\in E(G)$ with $u$ and $v$ both in $T$, we introduce three edges $(u,u^{(e)})$, $(u^{(e)},v^{(e)})$, and $(v^{(e)},v)$ in $E'$. Suppose $e=(u,v)$ has weight $w(u,v)$ in $G$. For edges introduced above, we assign $w(u,v)$ as their weights $w'$ in $G'$, i.e.,
                        \begin{displaymath}
                          w'(u,u^{(e)}) = w'(u^{(e)},v^{(e)}) = w'(v^{(e)},v) = w(u,v).
                        \end{displaymath}
                      If $e=(u,v)\in E(G)$ is not of the aforementioned forms, introduce one edge $(u,v)$ in $E'$ and set its weight $w'(u,w)$ in $G'$ as $w(u,v)$.
                \end{enumerate}
            It is obvious that $G'$ is connected. By the above construction, for $v\in V(G)\subseteq V'$, $v$ has the same weight in $G$ and in $G'$. For $v\in V'$, denote its weight in $G'$ as $w'(v)$. Since $G$ is unweighted in the current setup, $G'$ is also unweighted. So for a vertex $v'$ of $G'$, $w'(v')$ equals the degree of $v'$ in $G'$.
            \par
            c) Define a function $g\colon V'\rightarrow \RNS$:
                    \begin{numcases}
                        {g(v)\coloneqq }
                            \abs[\big]{f(v)},
                            \quad \nonumber &$v\in S\cup T$,\\
                            0, \quad \nonumber &\text{otherwise}.
                    \end{numcases}
                It follows that
                    \begin{displaymath}\begin{split}
                        g(x)=f(x)&=\sqrt{\mu_x^\cq(\delta)/w(x)},\\
                        \sum_{v\in V'}\abs[\big]{g(v)}^2w'(v)&=\sum_{v\in V(G)}\abs[\big]{f(v)}^2w(v)=1,\\
                        \sum_{(v_1,v_2)\in E'}w'(v_1,v_2)\abs[\big]{g(v_1)-g(v_2)}^2&\leqslant \sum_{(v_1,v_2)\in E(G)}w(v_1,v_2)\abs[\big]{f(v_1)+f(v_2)}^2.
                    \end{split}\end{displaymath}
                We claim that $g(v')=0$ for some $v'\in V'$. Otherwise, if $V'\setminus{(S\cup T)}=\varnothing$, all edges in $G$ would be between $S$ and $T$, contradicting the assumption that $G$ is non-bipartite.

                Set $B\coloneqq \bigl\{y\in V'\st \abs[\big]{g(y)-g(x)}\leqslant \frac{g(x)}{2}\bigr\}$. It follows that $B\subseteq S\cup T\subsetneq V'$. Since $G'$ is connected, there exists a shortest path $\cp$ in $G'$ from $x$ to $V'\setminus B$.
            \par
            d) If $\abs{\cp}=1$ and at least half of the neighbors of $x$ are outside of $B$, then we have
                    \begin{displaymath}\begin{split}
                        \delta
                        &\geqslant
                        \langle\cq f,f\rangle_w=
                        \sum_{(v_1,v_2)\in E(G)}\abs[\big]{f(v_1)+f(v_2)}^2\\
                        &\geqslant \sum_{(v_1,v_2)\in E'}\abs[\big]{g(v_1)-g(v_2)}^2\\
                        &\geqslant\frac{d(x)}{2}\cdot\frac{1}{4}\abs[\big]{g(x)}^2
                        =\frac{d(x)}{2}\cdot\frac{\mu_x^\cq(\delta)}{4w(x)}\\
                        &=\frac{\mu_x^\cq(\delta)}{8}.
                    \end{split}\end{displaymath}
                \par
                Therefore, when $\delta\leqslant 1$, $\mu_x^\cq(\delta)\leqslant 8\delta\leqslant 8\sqrt{\delta}\leqslant \tfrac{10 d(x)\sqrt{\delta}}{\dmi}$; when $\delta>1$,  $\mu_x^\cq(\delta)\leqslant1<\delta<\tfrac{10 d(x)\sqrt{\delta}}{\dmi}$. The result holds in this case.
            \par
            e) If $\abs{\cp}=1$ and at least half of the neighbors of $x$ are inside of $B$, then the neighbors of $x$ inside of $B$ are in $S\cup T$ and are therefore of degree at least $\dmi$ in $G'$. Thus,
                    \begin{displaymath}
                        \vol(B;G')\geqslant \frac{d(x)}{2}\cdot \dmi
                        \geqslant \frac{\dmi^2}{6}\abs{\cp}.
                    \end{displaymath}
                If $\abs{\cp}\geqslant 2$, we claim that $\vol(B;G')\geqslant \frac{\dmi^2}{6}\abs{\cp}$ still holds. To justify this claim, we assume that $\cp=u_0u_1\cdots u_ru_{r+1}$, where $r=\abs{P}-1$. Consider
                \begin{displaymath}
                    \widetilde{B}\coloneqq \bigl\{y\in V'\st \dist(x,y;G')\leqslant r \bigr\}\subseteq B\subseteq S\cup T.
                \end{displaymath}
                Since $\cp$ is a shortest path, $u_0$, $u_1$, $\ldots$, $u_{r-2}$, and $u_{r-1}$ are all in $S\cup T$, and are therefore of the same degree in both $G$ and $G'$, which is at least $\dmi$. Setting $K\coloneqq \{u_0,u_3,\ldots, u_{3\lfloor (r-1)/3\rfloor}\}$, we have
                    \begin{displaymath}
                        \abs{K}=\lceil r/3\rceil.
                    \end{displaymath}
                Counting the number of vertices in $\widetilde{B}$, we get
                    \begin{displaymath}
                        \abs{\widetilde{B}}\geqslant (r+1)+\abs{K}(\dmi-2),
                    \end{displaymath}
                where $r+1$ counts the vertices $\{u_0,u_1,\ldots, u_r\}$, and $\abs{K}(\dmi-2)$ counts the neighbors of $K$ that are in $\widetilde{B}$ but not on $\cp$. Hence, we get that
                    \begin{displaymath}\begin{split}
                        \vol(B;G')
                        &\geqslant
                         \vol(\widetilde{B};G')
                         \geqslant
                         \dmi\abs{\widetilde{B}}\\
                         &\geqslant
                         \dmi\cdot\bigl((r+1)+\abs{K}(\dmi-2)\bigr)\\
                         &\geqslant
                         \dmi\cdot \bigl((r+1)+\frac{r}{3}(\dmi-2)\bigr)\\
                        &\geqslant \frac{\dmi^2r}{3}
                        =\frac{\dmi^2}{3} \bigl(\abs{\cp}-1\bigr)\\
                        &\geqslant \frac{\dmi^2}{6}\abs{\cp}.
                    \end{split}\end{displaymath}
                So our claim holds.
                \par
                Therefore, we may assume that $\vol(B;G')\geqslant \frac{\dmi^2}{6}\abs{\cp}$.
            \par
            f) Note that $\sum_{v\in V'}\abs[\big]{g(v)}^2w'(v)=1$. It is easy to get
                    \begin{displaymath}
                        \vol(B;G')\leqslant \frac{1}{\abs[\big]{\frac{1}{2}g(x)}^2}
                        =\frac{4w(x)}{\mu_x^\cq(\delta)}
                        =\frac{4d(x)}{\mu_x^\cq(\delta)}.
                    \end{displaymath}
                Therefore,
                    \begin{displaymath}
                        \frac{4d(x)}{\mu_x^\cq(\delta)}\geqslant \vol(B,G')\geqslant \frac{\dmi^2}{6}\abs{\cp}.
                    \end{displaymath}
                Thus,
                    \begin{displaymath}
                        \abs{\cp}\leqslant \frac{24 d(x)}{\dmi^2\mu_x^\cq(\delta)}.
                    \end{displaymath}
                Hence,
                    \begin{displaymath}\begin{split}
                        \delta&\geqslant \langle\cq f,f\rangle_w
                        =\sum_{(v_1,v_2)\in E(G)}\abs[\big]{f(v_1)+f(v_2)}^2\\
                        &\geqslant \sum_{(v_1,v_2)\in E'}\abs[\big]{g(v_1)-g(v_2)}^2
                        \geqslant \sum_{i=0}^{\abs{\cp}-1}\abs[\big]{g(u_i)-g(u_{i+1})}^2\\
                        &\geqslant \frac{1}{\abs{\cp}}\biggl(\sum_{i=0}^{\abs{\cp}-1}\abs[\big]{g(u_i)-g(u_{i+1})}\biggr)^2
                        \geqslant  \frac{1}{\abs{\cp}}\abs[\big]{g(u_0)-g(u_{\abs{\cp}})}^2\\
                        &\geqslant \frac{1}{\abs{\cp}}\frac{\abs[\big]{g(x)}^2}{4}
                        =\frac{\mu_x^\cq(\delta)}{4\abs{\cp}w(x)}\\
                        &=\frac{\mu_x^\cq(\delta)}{4d(x)\abs{\cp}}.
                    \end{split}\end{displaymath}
                Proceeding further, we have
                    \begin{displaymath}\begin{split}
                        \delta\geqslant \frac{\mu_x^\cq(\delta)}{4d(x)\abs{\cp}}\geqslant
                        \frac{\mu_x^\cq(\delta)}{4d(x)\frac{24d(x)}{\dmi^2\mu_x^\cq(\delta)}}
                        \geqslant \Bigl(\frac{\dmi\mu_x^\cq(\delta)}{10d(x)}\Bigr)^2.
                    \end{split}\end{displaymath}
                Therefore,
            \begin{displaymath}
                \mu_x^\cq(\delta)\leqslant \frac{10 d(x)\sqrt{\delta}}{\dmi},\qquad 0\leqslant \delta\leqslant 2.\qedhere
            \end{displaymath}
        \end{proof}

            The upper bound in \Cref{reg_return_spec} could be easily used to get lower bounds on the eigenvalues of $P$. To this end, we need \Cref{spec_loc}:
        \begin{lemma}\label{spec_loc}
            Let $G$ be a finite, connected, weighted graph. We have
            \begin{displaymath}
                \sum_{x\in V}\mu_x^\cq(\delta)=\abs[\big]{\{j\st \lambda_j^\cq\leqslant \delta\}}.
            \end{displaymath}
        \end{lemma}
        \begin{proof}
            Recall that the eigenvalues of $\cq$ are
                \begin{displaymath}
                    0\leqslant  \lambda^\cq_{\mathrm{min}} = \lambda^\cq_1
                    \leqslant  \lambda^\cq_2 \leqslant \lambda^\cq_3 \leqslant \dotsb \leqslant \lambda^\cq_{n-1}
                     < \lambda^\cq_n = \lambda^\cq_{\mathrm{max}} =   2.
                \end{displaymath}
            Let $h_1$, $h_2$, $\ldots$, $h_n$ be an orthonormal basis of $\ell^2(V,w)$ such that
            \begin{displaymath}
                \cq h_j=\lambda^\cq_j h_j,\qquad 1\leqslant j\leqslant n.
            \end{displaymath}
            It follows that
                \begin{displaymath}
                    \sum_{x\in V}\abs{\langle h_j,\be_x\rangle_w}^2
                    =\norm{h_j}_w^2=1.
                \end{displaymath}
            Therefore,
            \begin{displaymath}\begin{split}
                \sum_{x\in V}\mu_x^\cq(\delta)
                &=
                \sum_{x\in V}\bigl\langle I_\cq(\delta)\be_x,\be_x\bigr\rangle_w
                =
                \sum_{x\in V}\norm{I_\cq(\delta)\be_x}^2_w \\
                &=
                \sum_{x\in V}\norm[\Big]{\sum_{j\st \lambda_j^\cq\leqslant \delta}\langle\be_x,h_j\rangle_w h_j}^2
                =
                \sum_{x\in V}\sum_{j\st \lambda_j^\cq\leqslant \delta}\abs[\big]{\langle\be_x,h_j\rangle_w}^2\\
                &=
                \sum_{j\st \lambda_j^\cq\leqslant \delta}\sum_{x\in V}\abs[\big]{\langle\be_x,h_j\rangle_w}^2
                =
                \sum_{j\st \lambda_j^\cq\leqslant \delta}\norm{h_j}^2_w
                =
                \sum_{j\st \lambda_j^\cq\leqslant \delta}1\\
                &=
                \abs[\big]{\{j\st \lambda_j^\cq\leqslant \delta\}}.\qedhere
            \end{split}\end{displaymath}
        \end{proof}
        \begin{corollary}\label{lbb_reg}
             Let $G$ be a regular, non-bipartite, finite, simple, connected, unweighted graph. For\/ $1\leqslant k\leqslant n$, we have $\lambda^\cq_k\geqslant \frac{k^2}{100n^2}$. Therefore, $\lambda^P_k\geqslant -1+\frac{k^2}{100n^2}$.
        \end{corollary}
        \Cref{lbb_reg} is similar to the second assertion of \autocite[Theorem 4.9]{Lyons-2018}.
        \begin{proof}
            By \Cref{spec_loc} and \Cref{reg_return_spec},
            \begin{displaymath}
                \abs[\big]{\{j\st \lambda_j^\cq\leqslant \delta\}}=\sum_{x\in V}\mu_x^\cq(\delta)=10n\sqrt{\delta}.
            \end{displaymath}
            Therefore, if $10n\sqrt{\delta}<k$, $\abs[\big]{\{j\st \lambda_j^\cq\leqslant \delta\}}<k$. In other words, $\lambda^\cq_k\geqslant \frac{k^2}{100n^2}$.
        \end{proof}

        Our main interest is to get bounds on return probabilities of simple random walk on regular graphs.
        \begin{theorem}\label{reg_return}
            Let $G$ be a regular, non-bipartite, simple, connected, unweighted graph. For each $x\in V$, simple random walk on $G$ satisfies
            \begin{displaymath}\begin{split}
                0\leqslant  p_t(x,x)-\pi(x)&\leqslant \frac{18}{\sqrt{t}}  \qquad \text{for\quad}  t\equiv 0\bmod 2,\\
                \abs[\big]{p_t(x,x)-\pi(x)}&\leqslant \frac{9}{\sqrt{t}}  \qquad \text{for\quad}  t\equiv 1\bmod 2.
            \end{split}\end{displaymath}
        \end{theorem}

        \begin{proof}
            \Cref{reg_return} is parallel to the last assertion of \autocite[Theorem 4.9]{Lyons-2018}.
            \par
            a) Set
                    \begin{numcases}
                        {\varphi(\lambda)\coloneqq }
                            10\sqrt{\lambda}
                            \quad \nonumber &if\/ $\lambda\geqslant0\text{ and }10\sqrt{\lambda}\leqslant \mu_x^*(1)$,\\
                            \mu_x^*(1) \quad \nonumber &\text{for intermediate values of $\lambda$},\\
                            \mu_x^*(2)-10\sqrt{2-\lambda}
                            \quad \nonumber &if\/ $\lambda\leqslant2\text{ and }\mu_x^*(2)-10\sqrt{2-\lambda}\geqslant \mu_x^*(1)$.
                    \end{numcases}
                We claim that the function $\varphi$ defined above satisfies the conditions in \Cref{spec_bound}(1). In fact, it is known that $\mu^*_x(\lambda)<10\sqrt{\lambda}$ from \autocite[Theorem~4.9]{Lyons-2018}. On the other hand, by \Cref{reg_return_spec},
                \begin{displaymath}\begin{split}
                    \mu^*_x(\lambda)&=\bigl\langle I_\cl\bigl((0,\lambda]\bigr)\be_x,\be_x\bigr\rangle_w \\
                    &=\bigl\langle I_\cl\bigl((0,2]\bigr)\be_x,\be_x\bigr\rangle_w
                        -\bigl\langle I_\cl\bigl((\lambda,2]\bigr)\be_x,\be_x\bigr\rangle_w \\
                    &= 1-\bigl\langle I_\cl({0})\be_x,\be_x\bigr\rangle_w
                        -\bigl\langle I_\cl\bigl((\lambda,2]\bigr)\be_x,\be_x\bigr\rangle_w \\
                    &\geqslant  1-\bigl\langle I_\cl({0})\be_x,\be_x\bigr\rangle_w
                        -\bigl\langle I_\cl\bigl([\lambda,2]\bigr)\be_x,\be_x\bigr\rangle_w \\
                    &= 1-\pi(x)
                        -\bigl\langle I_\cq(2-\lambda)\be_x,\be_x\bigr\rangle_w \\
                    &\geqslant 1-\pi(x)-10\sqrt{2-\lambda}\\
                    &=\mu_x^*(2)-10\sqrt{2-\lambda}.
                \end{split}\end{displaymath}
                The claim is proved.

                Therefore, for $t\equiv 0\bmod 2$, by \Cref{spec_bound},
                    \begin{displaymath}\begin{split}
                        \pi(x)\leqslant  p_t(x,x)&\leqslant \pi(x)+\int_{0}^{2}(1-\lambda)^t\varphi'(\lambda)\df{\lambda}\\
                        &\leqslant \pi(x)+2\int_{0}^{1}(1-\lambda)^t\frac{5}{\sqrt{\lambda}}\df{\lambda}\\
                        &= \pi(x)+10\int_{0}^{1}(1-\lambda)^t\frac{1}{\sqrt{\lambda}}\df{\lambda}\\
                        &\leqslant \pi(x)+\frac{18}{\sqrt{t}},
                    \end{split}\end{displaymath}
                where the last inequality follows from \Cref{calc_aux}.
                The first assertion is proved.
            \par
            b) Set
                    \begin{displaymath}
                        \psi_1(\lambda)\coloneqq \bigl(\mu_x^*(2)-10\sqrt{2-\lambda}\,\bigr)\vee 0,\qquad \lambda\in [0,2],
                    \end{displaymath}
                and
                    \begin{displaymath}
                        \psi_2(\lambda)\coloneqq  (10\sqrt{\lambda}\,)\wedge \mu_x^*(2),\qquad \lambda\in [0,2].
                    \end{displaymath}
                Then they satisfy the conditions of \Cref{spec_bound}(2). Consequently, we have
                    \begin{displaymath}
                        \pi(x)+\int_{0}^{2}(1-\lambda)^t\psi_1'(\lambda)\df{\lambda}\leqslant  p_t(x,x)\leqslant \pi(x)+\int_{0}^{2}(1-\lambda)^t\psi_2'(\lambda)\df{\lambda}.
                    \end{displaymath}
                By some elementary calculation, we get that
                    \begin{displaymath}
                        \pi(x)-\int_{0}^{1}(1-\lambda)^t\frac{5}{\sqrt{\lambda}}\df{\lambda}\leqslant  p_t(x,x)\leqslant \pi(x)+\int_{0}^{1}(1-\lambda)^t\frac{5}{\sqrt{\lambda}}\df{\lambda}.
                    \end{displaymath}
                The second assertion follows immediately from \Cref{calc_aux}.
        \end{proof}
        For a graph that is not necessarily regular, we have the following result:
        \begin{proposition}\label{irreg_return}
            Let $G$ be a non-bipartite, simple, connected, unweighted graph. For each $x\in V$, simple random walk on $G$ satisfies
            \begin{displaymath}\begin{split}
                0\leqslant  p_t(x,x)-\pi(x)&\leqslant \frac{18d(x)}{\dmi\sqrt{t}}  \qquad \text{for\quad}  t\equiv 0\bmod 2,\\
                \abs[\big]{p_t(x,x)-\pi(x)}&\leqslant \frac{9d(x)}{\dmi\sqrt{t}}  \qquad \text{for\quad}  t\equiv 1\bmod 2.
            \end{split}\end{displaymath}
        \end{proposition}
        \begin{proof}
            To prove this proposition, we may use an argument similar to the one in the proof of \Cref{reg_return}: instead of using \Cref{reg_return_spec}, we will employ \Cref{irreg_return_spec}. When $G$ is not necessarily regular, checking the proof of \autocite[Theorem 4.9]{Lyons-2018} carefully, we find
            \begin{displaymath}
                \mu_x^*(\delta)\leqslant \frac{10 d(x)\sqrt{\delta}}{\dmi},\qquad 0\leqslant \delta\leqslant 2.
            \end{displaymath}
            Therefore, we set
                    \begin{numcases}
                        {\varphi(\lambda)\coloneqq }
                            \tfrac{10d(x)\sqrt{\lambda}}{\dmi}
                            \quad \nonumber &if\/ $\lambda\geqslant0\text{ and }\tfrac{10d(x)\sqrt{\lambda}}{\dmi}\leqslant \mu_x^*(1)$,\\
                            \mu_x^*(1) \quad \nonumber &\text{for intermediate values of $\lambda$},\\
                            \mu_x^*(2)-\tfrac{10d(x)\sqrt{2-\lambda}}{\dmi}
                            \quad \nonumber &if\/ $\lambda\leqslant2\text{ and }\mu_x^*(2)-\tfrac{10d(x)\sqrt{2-\lambda}}{\dmi}\geqslant \mu_x^*(1)$,
                    \end{numcases}
                \begin{displaymath}
                        \psi_1(\lambda)\coloneqq \bigl(\mu_x^*(2)-\tfrac{10d(x)\sqrt{2-\lambda}}{\dmi}\,\bigr)\vee 0,\qquad \lambda\in [0,2],
                    \end{displaymath}
                and
                    \begin{displaymath}
                        \psi_2(\lambda)\coloneqq  \bigl(\tfrac{10d(x)\sqrt{\lambda}}{\dmi}\,\bigr)\wedge \mu_x^*(2),\qquad \lambda\in [0,2].
                    \end{displaymath}
                The argument in the proof of \Cref{reg_return} will proceed with suitable modification; and the conclusion of \Cref{irreg_return} follows easily. Details are omitted.
        \end{proof}
    \subsection{Sharpness, Spectral Radius, and Non-diagonal Convergence}
        The order of $\frac{1}{\sqrt{t}}$ in \Cref{reg_return} is sharp. We show this by the following \Cref{sharp}.
        \begin{example}\label{sharp}
            Consider an unweighted graph $G$ with $V(G)=\ZN$: $(i,j)\in E(G)$ if and only if\/ $0<\abs{i-j}\leqslant 2$. Obviously, $G$ is non-bipartite, connected, and\/ $4$-regular. For the simple random walk on $G$, \autocite[Theorem 1.1]{Davis-1995} implies
                \begin{displaymath}
                    \lim_{t\to\infty}\sqrt{5t}p_t(0,0)=\frac{1}{\sqrt{\pi}}.
                \end{displaymath}
            Therefore,
                \begin{displaymath}
                    p_t(0,0)\sim \frac{1}{\sqrt{5\pi t}}\qquad\text{as }t\to\infty.
                \end{displaymath}
            Hence, the sharpness is demonstrated.\hfill\qedsymbol
        \end{example}
        In \Cref{reg}, we focused on the negative spectrum of $P$ by considering $\cq=I+P$. As a comparison, in \autocite[Theorem~4.9]{Lyons-2018}, $\cl=I-P$ is exploited and so essentially the positive spectrum of $P$ is focused on. Recalling that the spectrum $\sigma(P)$ of $P$ is contained in $[-1,1]$, we set
        \begin{displaymath}\begin{split}
            \gamma_-&\coloneqq 1+\inf\sigma(P),\\
            \gamma_+&\coloneqq 1-\sup\bigl(\sigma(P)\setminus\{1\}\bigr).
        \end{split}\end{displaymath}
        It is obvious that $\gamma_-$ and $\gamma_+$ are both non-negative.
        But which is larger between $\gamma_-$ and $\gamma_+$? When $G$ is finite, it depends; but when $G$ is infinite, we have the following simple fact due to \autocite{Mao-2013}:
        \begin{proposition}\label{gap}
            Let $G=(V,E)$ be a connected, weighted, infinite, locally finite graph, with\/ $\vol(V)=\infty$. Then we have $\gamma_+\leqslant \gamma_-$. In other words, the spectral radius of $P$ is achieved by the positive spectrum.
        \end{proposition}
        \begin{proof}
            To begin, since $\vol(V)=\infty$, we see that the constant function is not in $\ell^2(V, w)$ and $1$ is not an eigenvalue of $P$.
            \par
            For any fixed small number $\varepsilon>0$, there exists a real function $f\in \ell^2(V, w)$, with $\norm{f}_w=1$, satisfying
            \begin{displaymath}
                -1+\gamma_-+\varepsilon
                \geqslant
                \langle Pf,f\rangle_w
                =\sum_{x\in V}w(x)(Pf)(x)f(x)
                =\sum_{x\in V}w(x)\Bigl(\sum_{y\in N(x)}p(x,y)f(y)\Bigr)f(x).
            \end{displaymath}
            Hence, we have
            \begin{displaymath}\begin{split}
                1-\gamma_--\varepsilon &\leqslant \abs{\langle Pf,f\rangle_w}
                =
                \abs[\bigg]{\sum_{x\in V}w(x)\Bigl(\sum_{y\in N(x)}p(x,y)f(y)\Bigr)f(x)}\\
                &\leqslant
                \sum_{x\in V}w(x)\Bigl(\sum_{y\in N(x)}p(x,y)\abs{f(y)}\Bigr)\abs{f(x)}\\
                &= \langle P\abs{f},\abs{f}\rangle_w.
            \end{split}\end{displaymath}
            Since $1$ is not an eigenvalue of $P$, we further have
            \begin{displaymath}
                \langle P\abs{f},\abs{f}\rangle_w\leqslant 1-\gamma_+.
            \end{displaymath}
            Hence, $1-(\gamma_-+\varepsilon)\leqslant 1-\gamma_+$. It follows that
            \begin{displaymath}
                \gamma_+\leqslant \gamma_-+\varepsilon.
            \end{displaymath}
            Because $\varepsilon>0$ is arbitrary, we conclude $\gamma_+\leqslant \gamma_-$.
        \end{proof}
        \begin{rmk}
            \Cref{gap} is stated in the language of graphs; it is nothing but a ``translation" of \autocite[Theorem~1.1]{Mao-2013}.
        \end{rmk}
        Using \Cref{reg_return}, we can also get the following result on non-diagonal convergence.
        \begin{theorem}\label{ndc}
            Let $G$ be a regular, non-bipartite, simple, connected, unweighted graph. For $x,y\in V$ and $t\geqslant 2$, simple random walk on $G$ satisfies
            \begin{displaymath}\begin{split}
                &\abs[\big]{p_{t}(x,y)-\pi(y)}\leqslant \frac{18}{\sqrt{t}} \qquad \text{for\quad}  t\equiv 0\bmod 2,\\\\
                &\abs[\big]{p_{t}(x,y)-\pi(y)}\leqslant \frac{18}{\sqrt[4]{t^2-1}} \qquad \text{for\quad}  t\equiv 1\bmod 2.
            \end{split}\end{displaymath}
        \end{theorem}
        \begin{proof}
           Because $G$ is regular, for $s\geqslant 1$, we have
           \begin{displaymath}\begin{split}
                \abs[\big]{p_{s}(x,y)-\pi(y)}
                &=\abs[\big]{\bigl\langle\be_x,P^{s}\be_y\bigr\rangle_w-\pi(y)}
                =\abs[\Big]{\int_{(0,2]}(1-\lambda)^{s}\df{\bigl\langle \be_x,I_\cl(\lambda)\be_y\bigr\rangle_w}}\\
                &=\abs[\Big]{\int_{(0,2]}(1-\lambda)^{s}\df{\bigl\langle \cl(\lambda)\be_x,I_\cl(\lambda)\be_y\bigr\rangle_w}}
                \leqslant \int_{(0,2]}\abs{1-\lambda}^{s}
                \df{\abs[\big]{\bigl\langle I_\cl(\lambda)\be_x,I_\cl(\lambda)\be_y\bigr\rangle_w}}.
           \end{split}\end{displaymath}
           Therefore, for $t\equiv 0\bmod 2$,
           \begin{displaymath}\begin{split}
                \abs[\big]{p_{t}&(x,y)-\pi(y)}
                \leqslant \int_{(0,2]}\abs{1-\lambda}^{t}
                \df{\abs[\big]{\bigl\langle I_\cl(\lambda)\be_x,I_\cl(\lambda)\be_y\bigr\rangle_w}}\\
                &\leqslant \Bigl(\int_{(0,2]}(1-\lambda)^{t}
                \df{\bigl\langle I_\cl(\lambda)\be_x,I_\cl(\lambda)\be_x\bigr\rangle_w}\Bigr)^{1/2}
                \Bigl(\int_{(0,2]}(1-\lambda)^{t}
                \df{\bigl\langle I_\cl(\lambda)\be_y,I_\cl(\lambda)\be_y\bigr\rangle_w}\Bigr)^{1/2}\\
                &= \Bigl(\int_{(0,2]}(1-\lambda)^{t}
                \df{\bigl\langle I_\cl(\lambda)\be_x,\be_x\bigr\rangle_w}\Bigr)^{1/2}
                \Bigl(\int_{(0,2]}(1-\lambda)^{t}
                \df{\bigl\langle I_\cl(\lambda)\be_y,\be_y\bigr\rangle_w}\Bigr)^{1/2}\\
                &= \Bigl(\int_{(0,2]}(1-\lambda)^{t}
                \df{\mu_x^*(\lambda)}\Bigr)^{1/2}
                \Bigl(\int_{(0,2]}(1-\lambda)^{t}
                \df{\mu_y^*(\lambda)}\Bigr)^{1/2}\\
                &=
                \bigl(p_{t}(x,x)-\pi(x)\bigr)^{1/2}
                \bigl(p_{t}(y,y)-\pi(y)\bigr)^{1/2}
                \leqslant
                \frac{18}{\sqrt{t}},
            \end{split}\end{displaymath}
            where we are using \Cref{rep} to get the last equality and \Cref{reg_return} to get the last inequality.
            Hence, for $x,y\in V$ and $t\equiv 0\bmod 2$, we have $\abs[\big]{p_{t}(x,y)-\pi(y)}\leqslant \frac{18}{\sqrt{t}}$.

            Similarly, for $t\equiv 1\bmod 2$ and $t\geqslant 3$,
            \begin{displaymath}\begin{split}
                \abs[\big]{p_{t}&(x,y)-\pi(y)}
                \leqslant 
                \int_{(0,2]}\abs{1-\lambda}^{(t-1)/2}\abs{1-\lambda}^{(t+1)/2}
                \df{\abs[\big]{\bigl\langle I_\cl(\lambda)\be_x,I_\cl(\lambda)\be_y\bigr\rangle_w}}\\
                &\leqslant \Bigl(\int_{(0,2]}(1-\lambda)^{t-1}
                \df{\bigl\langle I_\cl(\lambda)\be_x,I_\cl(\lambda)\be_x\bigr\rangle_w}\Bigr)^{1/2}
                \Bigl(\int_{(0,2]}(1-\lambda)^{t+1}
                \df{\bigl\langle I_\cl(\lambda)\be_y,I_\cl(\lambda)\be_y\bigr\rangle_w}\Bigr)^{1/2}\\
                &=
                \bigl(p_{t-1}(x,x)-\pi(x)\bigr)^{1/2}
                \bigl(p_{t+1}(y,y)-\pi(y)\bigr)^{1/2}
                \leqslant
                \frac{18}{\sqrt[4]{t^2-1}}.
            \end{split}\end{displaymath}
            The proof is complete.
        \end{proof}
        \begin{rmk}
            We will give several results on return probability bound throughout this article. \Cref{ndc} is a sample of deducing non-diagonal convergence from return probability bounds.
        \end{rmk}

    \section{Volume Growth Conditions}\label{sec_vol}
         For lazy random walk, \autocite[Corollaries 4.10 and 4.11]{Lyons-2018} presented return probability bounds depending on volume growth conditions. We have parallel results for simple random walk. Let us begin with the following proposition, which is comparable to the first assertion of 
         \autocite[Proposition 4.7]{Lyons-2018}.
         \begin{proposition}\label{vgrow_spec}
            Let $G$ be a non-bipartite, finite or infinite, simple, connected, weighted graph with weight at least\/ $1$ for each edge. For each vertex $x\in V(G)$, $\delta\in(0,2)$, $\alpha\in(0,1)$, and $r\geqslant 0$,
            $\vol(x,r)>\frac{w(x)}{(1-\alpha)^2\mu_x^\cq(\delta)}$ implies $\mu_x^\cq(\delta)\leqslant \frac{\delta w(x)}{\alpha^2}r$.
         \end{proposition}
         \begin{proof}
            We may assume $\mu_x^\cq(\delta)>0$. Fixing a vertex $x\in V(G)$, we define $f$ as in \Cref{def_f}.
            \par
            a) If $f(y)\geqslant 0$ for all $y\in V(G)$, then
                \begin{displaymath}
                    \delta
                    \geqslant
                    \langle\cq f,f\rangle_w
                    =
                    \sum_{(v,u)\in E(G)}w(v,u)\abs[\big]{f(v)+f(u)}^2
                    \geqslant
                    \abs[\big]{f(x)}^2
                    =\frac{\mu_x^\cq(\delta)}{w(x)}.
                \end{displaymath}
                Note that $\vol(x,r)>\frac{w(x)}{(1-\alpha)^2\mu_x^\cq(\delta)}$ implies $r\geqslant 1$. Therefore,
                \begin{displaymath}
                    \mu_x^\cq(\delta)
                    \leqslant
                    \delta w(x)
                    \leqslant
                    \frac{\delta w(x)}{\alpha^2}r.
                \end{displaymath}
            \par
            b) From now on, we may assume that $T\coloneqq \bigl\{t\in V(G)\st f(t)<0\bigr\}$ is non-empty. We construct $G'$ and $g$ as in the proof of \Cref{irreg_return_spec} and use notations there. Set
                \begin{displaymath}
                    B^{(\alpha)}
                    \coloneqq
                    \bigl\{
                    y\in V' \st \abs{g(y)-g(x)}\leqslant \alpha g(x)
                    \bigr\}.
                \end{displaymath}
            It follows that $B^{(\alpha)}\subseteq S\cup T\subsetneq V'$. Since $G'$ is connected, there exists a shortest path $\cp=u_0u_1\cdots u_{\abs{\cp}}$ in $G'$ from $x$ to $V'\setminus B^{(\alpha)}$ with $u_0=x$. Hence,
                    \begin{displaymath}\begin{split}
                        \delta&\geqslant
                        \langle\cq f,f\rangle_w
                        =\sum_{(v_1,v_2)\in E(G)}w(v_1,v_2)\abs[\big]{f(v_1)+f(v_2)}^2\\
                        &\geqslant \sum_{(v_1,v_2)\in E'}w'(v_1,v_2)\abs[\big]{g(v_1)-g(v_2)}^2
                        \geqslant \sum_{i=0}^{\abs{\cp}-1}\abs[\big]{g(u_i)-g(u_{i+1})}^2\\
                        &\geqslant \frac{1}{\abs{\cp}}\biggl(\sum_{i=0}^{\abs{\cp}-1}\abs[\big]{g(u_i)-g(u_{i+1})}\biggr)^2
                        \geqslant  \frac{1}{\abs{\cp}}\abs[\big]{g(u_0)-g(u_{\abs{\cp}})}^2\\
                        &> \frac{\alpha^2\abs[\big]{g(x)}^2}{\abs{\cp}}
                        =\frac{\alpha^2\mu_x^\cq(\delta)}{w(x)\abs{\cp}}.
                    \end{split}\end{displaymath}
            \par
            c) We claim that
                \begin{equation}\label{aux_equal}
                    \bigl\{
                    y\in V' \st \dist(y,x;G')\leqslant \abs{\cp}-1
                    \bigr\}
                    =
                    \bigl\{
                    y\in V(G) \st \dist(y,x;G)\leqslant \abs{\cp}-1
                    \bigr\}.
                \end{equation}
            In fact,
                \begin{displaymath}
                    \bigl\{
                    y\in V' \st \dist(y,x;G')\leqslant \abs{\cp}-1
                    \bigr\}
                    \subseteq
                    B^{(\alpha)}\subseteq S\cup T\subseteq V(G).
                \end{displaymath}
            By the construction of $G'$, $\dist(y_1,y_2;G')\geqslant \dist(y_1,y_2;G)$ for $y_1,y_2\in V$. Therefore,
                \begin{displaymath}
                    \bigl\{
                    y\in V' \st \dist(y,x;G')\leqslant \abs{\cp}-1
                    \bigr\}
                    \subseteq
                    \bigl\{
                    y\in V(G) \st \dist(y,x;G)\leqslant \abs{\cp}-1
                    \bigr\}.
                \end{displaymath}
            Suppose the above inclusion is strict. Then there will be a vertex $\widetilde{v}\in V(G)$ that is not in the left-hand side of \Cref{aux_equal}, and a path $\widetilde{\cp}=\widetilde{u}_0\widetilde{u}_1\cdots\widetilde{u}_{\abs{\widetilde{\cp}}}$ in $G$ with $\abs{\widetilde{\cp}}\leqslant \abs{\cp}-1$, $\widetilde{u}_0=x$, and $\widetilde{u}_{\abs{\widetilde{\cp}}}=\widetilde{v}$. Set
            \begin{displaymath}
                L\coloneqq
                \bigl\{
                    0\leqslant i\leqslant \abs{\widetilde{\cp}}-1
                    \st
                    \{\widetilde{u}_i,\widetilde{u}_{i+1}\}\subseteq S \text{ or } \{\widetilde{u}_i,\widetilde{u}_{i+1}\}\subseteq T
                \bigr\}.
            \end{displaymath}
            Then $L$ must be non-empty, otherwise $\widetilde{v}\in V(G)$ would be in the left-hand side of \Cref{aux_equal}. Pick the smallest number $i^*$ in $L$ and write $e=(\widetilde{u}_{i^*},\widetilde{u}_{i^*+1})$. By the construction of $G'$ and $g$, we have $g(\widetilde{u}_{i^*}^{(e)})=0$.
            Thus, $\widetilde{u}_0\widetilde{u}_1\cdots\widetilde{u}_{i^*}\widetilde{u}_{i^*}^{(e)}$ is a path in $G'$ linking $x$ and $V'\setminus B^{(\alpha)}$, whose length is $i^*+1\leqslant \abs{\widetilde{\cp}}<\abs{\cp}$. This is a contradiction. Therefore, \Cref{aux_equal} holds.
            \par
            d) If $\vol(x,r;G)>\frac{w(x)}{(1-\alpha)^2\mu_x^\cq(\delta)}$, we must have $r\geqslant \abs{\cp}$. In fact, we notice that
            \begin{displaymath}
                \vol(B^{(\alpha)};G')(1-\alpha)^2g^2(x)
                \leqslant
                \sum_{y\in B^{(\alpha)}}\abs[\big]{g(y)}^2w'(y)
                \leqslant
                \sum_{y\in V'}\abs[\big]{g(y)}^2w'(y)
                =1.
            \end{displaymath}
            Therefore,
            \begin{displaymath}
                \vol(x,\abs{\cp}-1;G')
                \leqslant
                \vol(B^{(\alpha)};G')
                \leqslant
                \frac{1}{(1-\alpha)^2g^2(x)}
                =
                \frac{w(x)}{(1-\alpha)^2\mu_x^\cq(\delta)}.
            \end{displaymath}
            Hence, \Cref{aux_equal} gives
            \begin{displaymath}
                \vol(x,\abs{\cp}-1;G)
                =
                \vol(x,\abs{\cp}-1;G')
                \leqslant
                \frac{w(x)}{(1-\alpha)^2\mu_x^\cq(\delta)}.
            \end{displaymath}
            Consequently, $\vol(x,r)>\frac{w(x)}{(1-\alpha)^2\mu_x^\cq(\delta)}$ implies $r\geqslant  \abs{\cp}$.
            \par
            e) By the results in b) and d), we have
                \begin{displaymath}
                    \delta
                    \geqslant
                    \frac{\alpha^2\mu_x^\cq(\delta)}{w(x)\abs{\cp}}
                    \geqslant
                    \frac{\alpha^2\mu_x^\cq(\delta)}{w(x)r}.
                \end{displaymath}
            Hence, we arrive at the conclusion that $\mu_x^\cq(\delta)\leqslant \frac{\delta w(x)}{\alpha^2}r$.
         \end{proof}
    \begin{corollary}\label{prp472}
        Let $G$ be a non-bipartite, finite or infinite, simple, connected, weighted graph with weight at least\/ $1$ for each edge. For each vertex $x\in V$ and $\delta<\frac{1}{r\vol(x,r)}$, we have $\mu_x^\cq(\delta)\leqslant \frac{4w(x)}{\vol(x,r)}$.\qed
    \end{corollary}
    \Cref{prp472} is comparable to the second assertion of \autocite[Proposition 4.7]{Lyons-2018}. To get a proof, one need only exploit \Cref{vgrow_spec} and mimic the proof for the second assertion of \autocite[Proposition 4.7]{Lyons-2018}.

    Now we present return probability bounds based on volume growth conditions.
    \begin{corollary}\label{vgrow_pro}
        Let $G$ be a non-bipartite, infinite, simple, connected, weighted graph with weight at least\/ $1$ for each edge. Suppose that $c > 0$ and $D \geqslant  1$ are constants such that\/ $\vol(x, r) \geqslant  c (r+1)^D$ for all
        $r \geqslant  0$. Then for all $\delta\in(0,2)$,
        \begin{displaymath}\begin{split}
            \mu^*_x(\delta) &\leqslant  C w(x) \delta^{D/(D+1)},\\
            \mu^\cq _x(\delta) &\leqslant  C w(x) \delta^{D/(D+1)},
        \end{split}\end{displaymath}
        where
        \begin{displaymath}
        C \coloneqq  \frac{(D+1)^2}{c^{1/(D+1)} D^{2D/(D+1)}}.
        \end{displaymath}
        For $t \geqslant  1$, simple random walk satisfies
        \begin{displaymath}\begin{split}
            p_t(x,x)
            &\leqslant
            2C' w(x) t^{-D/(D+1)}\qquad \text{for\quad}  t\equiv 0\bmod 2,\\
            p_t(x,x)
            &\leqslant
            C' w(x) t^{-D/(D+1)}
            \qquad \text{for\quad}  t\equiv 1\bmod 2,
        \end{split}\end{displaymath}
        where
        \begin{displaymath}
        C' \coloneqq  \frac{D+1}{c^{1/(D+1)} D^{(D-1)/(D+1)}} \Gamma\Big(\frac{D}{D+1}\Big).
        \end{displaymath}
    \end{corollary}
    \begin{proof}
        The bound on $\mu^*_x$ is proved in \autocite[Corollary 4.10]{Lyons-2018}. A similar argument can be used to prove the bound on $\mu^\cq_x$: in lieu of \autocite[Eq.\@ (4.4)]{Lyons-2018}, we use \Cref{prp472}. To prove the bound on return probabilities, we set
        \begin{numcases}
            {\varphi(\lambda)\coloneqq }
            C w(x) \lambda^{D/(D+1)}
            \; \nonumber &if\/ $\lambda \in [0,1)\text{ and }C w(x) \lambda^{D/(D+1)}\leqslant \mu_x^*(1)$,\\
            \mu_x^*(1) \; \nonumber &\text{for intermediate values of $\lambda$},\\
            \mu_x^*(2) - C w(x) (2-\lambda)^{D/(D+1)}
            \; \nonumber &if\/ $\lambda\in [1,2)\text{ and }\mu_x^*(2)-C w(x) (2-\lambda)^{D/(D+1)}\geqslant \mu_x^*(1)$,
        \end{numcases}
        and
        \begin{numcases}
            {\psi_2(\lambda)\coloneqq }
            C w(x) \lambda^{D/(D+1)}\wedge \mu_x^*(2)
            \; \nonumber &if\/ $\lambda \in [0,1)$,\\
            \mu_x^*(2)
            \; \nonumber &if\/ $\lambda\in [1,2]$.
        \end{numcases}
        It is easy to see that $\varphi$ and $\psi_2$ satisfy the conditions in \Cref{spec_bound}. Therefore, we have
        \begin{displaymath}\begin{split}
            p_t(x,x)&\leqslant
            \frac{2 C w(x)D}{D+1}\int_{0}^{1}\lambda^{-1/(D+1)}(1-\lambda)^t\df{\lambda}
            \qquad \text{for\quad}  t\equiv 0\bmod 2,\\
            p_t(x,x)&\leqslant
            \frac{C w(x)D}{D+1}\int_{0}^{1}\lambda^{-1/(D+1)}(1-\lambda)^t\df{\lambda}
            \qquad \text{for\quad}  t\equiv 1\bmod 2.
        \end{split}\end{displaymath}
        But we have
        \begin{displaymath}\begin{split}
            \int_{0}^{1}\lambda^{-1/(D+1)}(1-\lambda)^t\df{\lambda}
            &\leqslant
            \int_{0}^{1}\lambda^{-1/(D+1)}\ue^{-\lambda t}\df{\lambda}
            \leqslant
            \int_{0}^{\infty}\lambda^{-1/(D+1)}\ue^{-\lambda t}\df{\lambda}\\
            &=
            \int_{0}^{\infty}(s/t)^{-1/(D+1)}\ue^{-s}\df{(s/t)}
            =
            t^{-D/(D+1)}\int_{0}^{\infty}s^{-1/(D+1)}\ue^{-s}\df{s}\\
            &=
            t^{-D/(D+1)}\Gamma\bigl(\frac{D}{D+1}\bigr),
        \end{split}\end{displaymath}
        where we introduce a change of variable $\lambda t=s$ to get the first equality. Hence, the bound on return probabilities is proved.
    \end{proof}

    \Cref{vgrow_pro} is comparable to \autocite[Corollary~4.10]{Lyons-2018}. Using a similar method, one may get an analogue of \autocite[Corollary~4.11]{Lyons-2018}. Details are omitted.


\section{Mixing Time Bound}\label{sec_mix}
    We are concerned in this section with mixing time bounds, which are based on the bounds on vertex spectral measures in \Cref{mix_spec_prep}. As a preparation, we introduce the following graph parameter.
    \begin{definition}\label{def_kudos}
        Let $G$ be a finite and weighted graph.
        We define
        \begin{displaymath}
            \kg\coloneqq
            \min\biggl\{
            \frac{\sum_{(v,u)\in E}w(v,u)\abs{f(v)+f(u)}^2}{\max_{y\in V}\abs{f(y)}^2}
            \st
            \min_{y\in V}f(y)<0< \max_{y\in V}f(y)
            \biggr\}.
        \end{displaymath}
    \end{definition}
    For a bipartite graph $G$, it is easy to see that $\kg=0$. For non-bipartite graphs, we have the following lemma.
    \begin{lemma}\label{est}
        Let $G$ be a connected and weighted graph, with weight at least\/ $1$ for each edge.
        \begin{enumerate}[\upshape(1)]
          \item Assume that $f$ is a function on $V$ and $\cp=z_0z_1\cdots z_{\abs{\cp}}$ is an edge-simple path. Then
                \begin{displaymath}
                    \sum_{(v,u)\in E}w(v,u)\abs[\big]{f(v)+f(u)}^2\geqslant \frac{1}{\abs{\cp}}\bigl(f(z_0)-(-1)^{\abs{\cp}}f(z_{\abs{\cp}})\bigr)^2.
                \end{displaymath}
          \item If $G$ is also non-bipartite, finite, and simple, then $\kg\geqslant \frac{1}{\D(G)+1}$.
        \end{enumerate}
    \end{lemma}
    \begin{proof}
        a) Note that
            \begin{displaymath}
                \sum_{(v,u)\in E}w(v,u)\abs[\big]{f(v)+f(u)}^2
                \geqslant \sum_{i=0}^{\abs{\cp}-1}\abs[\big]{f(z_i)+f(z_{i+1})}^2
            \end{displaymath}
            By the Cauchy--Schwarz inequality, we have
            \begin{displaymath}\begin{split}
                \sum_{i=0}^{\abs{\cp}-1}\abs[\big]{f(z_i)+f(z_{i+1})}^2
                &=
                \sum_{i=0}^{\abs{\cp}-1}\abs[\big]{(-1)^if(z_i)-(-1)^{i+1}f(z_{i+1})}^2\\
                &\geqslant \frac{1}{\abs{\cp}}\biggl(\sum_{i=0}^{\abs{\cp}-1}\bigl((-1)^if(z_i)-(-1)^{i+1}f(z_{i+1})\bigr)\biggr)^2\\
                &= \frac{1}{\abs{\cp}}\bigl(f(z_0)-(-1)^{\abs{\cp}}f(z_{\abs{\cp}})\bigr)^2.
            \end{split}\end{displaymath}
            The first assertion is proved.
        \par
        b) Now we deal with the second assertion. Take $f$ satisfying the constraints in \Cref{def_kudos} such that
            \begin{displaymath}
                \kg
                =
                \frac{\sum_{(v,u)\in E}w(v,u)\abs{f(v)+f(u)}^2}{\max_{y\in V}\abs{f(y)}^2}.
            \end{displaymath}
            Assume that $\abs{f}$ attains its maximum at $x$, then
            \begin{displaymath}
                \kg
                =
                \frac{\sum_{(v,u)\in E}w(v,u)\abs{f(v)+f(u)}^2}{\abs{f(x)}^2}.
            \end{displaymath}
            We may assume $f(x)> 0$ without loss of generality.
            Set
            \begin{displaymath}
                S\coloneqq \bigl\{s\in V\st f(s)\geqslant 0\bigr\},\qquad T\coloneqq \bigl\{t\in V\st f(t)<0\bigr\}.
            \end{displaymath}
            By the assumptions, $S$ and $T$ are both non-empty. Because $G$ is non-bipartite, there exists an edge $(s_1,s_2)\in E$ with $s_1,s_2\in S$, or an edge $(t_1,t_2)\in E$ with $t_1,t_2\in T$. \vadjust{\kern2pt}%
            \par
        c) If there is an edge $(s_1,s_2)\in E$ with $s_1,s_2\in S$, let $\widehat{\cp}_0$ be a shortest path from \vadjust{\kern2pt}%
            $x$ to $\{s_1,s_2\}$.
            Without loss of generality, we may assume $\widehat{\cp}_0$ is from $x$ to $s_1$. If $\abs{\widehat{\cp}_0}$ is odd, \vadjust{\kern2pt}%
            we set $\widehat{\cp}\coloneqq \widehat{\cp}_0$; if $\abs{\widehat{\cp}_0}$ is even, we set $\widehat{\cp}\coloneqq \widehat{\cp}_0.(s_1,s_2)$, \vadjust{\kern2pt}%
            the concatenation of $\widehat{\cp}_0$ and the edge $(s_1,s_2)$.
            Hence, $\widehat{\cp}$ is a path of odd length in any case. Assume
            \begin{displaymath}
                \widehat{\cp}=u_0u_1\cdots u_k,
            \end{displaymath}
            with $u_0=x$. Then $k$ is odd and $k\leqslant \D(G)+1$. Hence, by the first assertion,
            \begin{displaymath}\begin{split}
                \sum_{(v,u)\in E}w(v,u)\abs[\big]{f(v)+f(u)}^2
                \geqslant \frac{1}{k}\bigl(f(u_0)+f(u_k)\bigr)^2
                \geqslant \frac{1}{k}\abs[\big]{f(x)}^2
                \geqslant \frac{\abs[\big]{f(x)}^2}{\D(G)+1}.
            \end{split}\end{displaymath}
            Therefore, in this case,
            \begin{displaymath}
                \kg\geqslant \frac{1}{\D(G)+1}.
            \end{displaymath}
            \par
        d) If there is an edge $(t_1,t_2)\in E$ with $t_1,t_2\in T$, \vadjust{\kern2pt}%
            let $\overline{\cp}_0$ be a shortest path from $x$ to $\{t_1,t_2\}$.
            Without loss of generality, we may assume $\overline{\cp}_0$ is from $x$ to $t_1$. If $\abs{\overline{\cp}_0}$ is even, \vadjust{\kern2pt}%
            we set $\overline{\cp}\coloneqq \overline{\cp}_0$; if $\abs{\overline{\cp}_0}$ is odd, we set $\overline{\cp}\coloneqq \overline{\cp}_0.(t_1,t_2)$. Hence, $\overline{\cp}$ is a path of even length in any case. Assume
            \begin{displaymath}
                \overline{\cp}=v_0v_1\cdots v_\ell,
            \end{displaymath}
            with $v_0=x$. Then $\ell$ is even and $\ell\leqslant \D(G)+1$. Hence, by the first assertion,
            \begin{displaymath}\begin{split}
                \sum_{(v,u)\in E}w(v,u)\abs[\big]{f(v)+f(u)}^2
                \geqslant \frac{1}{\ell}\bigl(f(v_0)-f(v_{\ell})\bigr)^2
                > \frac{1}{\ell}\abs[\big]{f(x)}^2
                \geqslant \frac{\abs[\big]{f(x)}^2}{\D(G)+1}.
            \end{split}\end{displaymath}
            Therefore, in this case,
            \begin{displaymath}
                \kg\geqslant \frac{1}{\D(G)+1}.\qedhere
            \end{displaymath}
    \end{proof}
    Now, we present our bounds on vertex spectral measures.
    \begin{lemma}\label{mix_spec_prep}
        Let $G$ be a finite, simple, connected, weighted graph, with weight at least\/ $1$ for each edge.
        \begin{enumerate}[\upshape (1)]
          \item For $\delta\in[0,2)$ and $x\in V$, $\mu_x^*(\delta) + \pi(x)\leqslant \rdiam(G)w(x)\delta$.
          \item If $G$ is also non-bipartite, for $\delta\in[0,2)$ and
          $x\in V$, $\mu_x^\cq(\delta) \leqslant \frac{w(x)\delta}{\kg}$. Moreover, by \Cref{est}, we have
          $\mu_x^\cq(\delta)\leqslant  \bigl(\D(G)+1\bigr)w(x)\delta$ for $\delta\in[0,2)$.
        \end{enumerate}
    \end{lemma}
    \begin{proof}
      The first assertion follows immediately from \autocite[Proposition~4.2]{Lyons-2018}. We now deal with the second assertion. Fix a vertex $x\in V$ and define $f$ as in \Cref{def_f}.
        \par
        \par
        When $\delta\in[0,\lambda_{\mathrm{min}}^\cq)$, $\mu_x^\cq(\delta)=0$ by definition. So the inequality holds automatically.
        \par
        When $\lambda^\cq_{\mathrm{min}}\leqslant \delta<\lambda^\cq_{\mathrm{max}}=2$, we know that $f$ is orthogonal to the eigenspace of $\cq$ corresponding to $\lambda^\cq_{\mathrm{max}}=2$, which is spanned by $(1,1,\ldots, 1)$, the constant vector. Therefore, by \Cref{def_kudos}, we have
            \begin{displaymath}
                \delta \geqslant \langle f,\cq f\rangle_w
                =
                \sum_{(v,u)\in E}w(v,u)\abs{f(v)+f(u)}^2
                \geqslant f(x)^2\kg
                =\frac{\mu_x^\cq(\delta)\kg}{w(x)}.\qedhere
            \end{displaymath}
    \end{proof}
    By \Cref{est}(2), \Cref{mix_spec_prep}(2), and \Cref{spec_loc}, we have the following corollary.
    \begin{corollary}\label{Landau_imp}
        For any non-bipartite, finite, simple, connected, weighted graph $G$ with weight at least\/ $1$ for each edge, we have
        \begin{displaymath}
            \lambda^\cq_k\geqslant \frac{k\kg}{\vol(V)}\geqslant \frac{k}{\bigl(\D(G)+1\bigr)\vol(V)}
        \end{displaymath}
        for\/ $1\leqslant k\leqslant n$. Therefore,
        \begin{displaymath}
            \pushQED{\qed}
            \lambda^P_k
            \geqslant
            -1+\frac{k\kg}{\vol(V)}
            \geqslant
             -1+\frac{k}{\bigl(\D(G)+1\bigr)\vol(V)}.\qedhere
            \popQED
        \end{displaymath}
    \end{corollary}
    As mentioned in \Cref{ssec_results}, \Cref{Landau_imp}, combined with \autocite[Proposition~4.2]{Lyons-2018}, improves the result in \autocite{Landau-1981}.

    To get a mixing time bound, we first give a bound on return probabilities using \Cref{mix_spec_prep}.
    \begin{corollary}\label{mix_even_prep}
        Let $G$  be a non-bipartite, finite, simple, connected, unweighted graph, and $t\equiv 0\bmod 2$. Then we have
            \begin{displaymath}
                0\leqslant
                \frac{p_{t}(x,x)-\pi(x)}{\pi(x)}
                \leqslant
                \frac{2\bigl(\D(G)+1\bigr)\vol(V)}{t}.
            \end{displaymath}
    \end{corollary}
    \begin{proof}
        Set
            \begin{numcases}
                {\widetilde{\varphi}(\lambda)\coloneqq }
                \rdiam(G)w(x)\lambda
                \; \nonumber &if\/ $\lambda\in [0,1)\text{ and }\rdiam(G)w(x)\lambda\leqslant \mu_x^*(1)$,\\
                \mu_x^*(1)\; \nonumber &\text{for intermediate values of $\lambda$},\\
                \mu_x^*(2)-\bigl(\D(G)+1\bigr)w(x)(2-\lambda)
                \; \nonumber
                            &\hspace{-5pt}$\begin{array}{l}
                            \text{if }\lambda\in[1,2]\text{ and }\\
                            \mu_x^*(2)-\bigl(\D(G)+1\bigr)w(x)(2-\lambda)\geqslant \mu_x^*(1).
                            \end{array}$
            \end{numcases}
        By \Cref{mix_spec_prep}, using a similar argument as in part a) of the proof of \Cref{reg_return}, the function $\widetilde{\varphi}$ defined above satisfies the conditions in \Cref{spec_bound}(1). Therefore, for $t\equiv 0\bmod 2$,
            \begin{displaymath}\begin{split}
                0\leqslant  p_t(x,x)-\pi(x)&\leqslant \int_{0}^{2}(1-\lambda)^t\widetilde{\varphi}'(\lambda)\df{\lambda}\\
                &\leqslant \rdiam(G)w(x)\int_{0}^{1}(1-\lambda)^t\df{\lambda}+\bigl(\D(G)+1\bigr)w(x)\int_{0}^{1}(1-\lambda)^t\df{\lambda}\\
                &\leqslant 2\bigl(\D(G)+1\bigr)w(x)\int_{0}^{1}(1-\lambda)^t\df{\lambda}.
            \end{split}\end{displaymath}
        Hence,
            \begin{displaymath}
                \frac{p_{t}(x,x)-\pi(x)}{\pi(x)}
                \leqslant 2\bigl(\D(G)+1\bigr)\frac{w(x)}{\pi(x)}\int_{0}^{1}(1-\lambda)^t\df{\lambda}.
            \end{displaymath}
        But we know $\frac{w(x)}{\pi(x)}=\vol(V)$ and
        \begin{displaymath}
            \int_{0}^{1}(1-\lambda)^t\df{\lambda}
            \leqslant \int_{0}^{1}\exp\{-\lambda t\}\df{\lambda}
            \leqslant\int_{0}^{\infty}\exp\{-\lambda t\}\df{\lambda}=\frac{1}{t}.
        \end{displaymath}
        Therefore,
        \begin{displaymath}
            \frac{p_{t}(x,x)-\pi(x)}{\pi(x)}\leqslant \frac{2\bigl(\D(G)+1\bigr)\vol(V)}{t}.\qedhere
        \end{displaymath}
    \end{proof}
    We are now almost in position to give the following mixing time bound.
    \begin{corollary}\label{mixx}
        For a non-bipartite, finite, simple, connected, unweighted graph $G$, the uniform mixing time of the simple random walk on $G$ satisfies
            \begin{displaymath}
                \tiq\leqslant 8n^3.
            \end{displaymath}
        If we further assume that $G$ is regular, then $\tiq\leqslant 24n^2$.
    \end{corollary}
    \Cref{mixx} is parallel to \autocite[Corollaries 4.3 and 4.6]{Lyons-2018}. In the proof of \Cref{mixx}, the following lemma will be useful.
    \begin{lemma}\label{mix_mono_prep}
        For a reversible Markov chain with stationary distribution $\pi$ and $t\equiv 0\bmod 2$, we have
            \begin{displaymath}
                \abs[\bigg]{\frac{p_{t}(x,y)-\pi(y)}{\pi(y)}}
                \leqslant
                \sqrt{\frac{p_{t}(x,x)-\pi(x)}{\pi(x)}}\sqrt{\frac{p_{t}(y,y)-\pi(y)}{\pi(y)}}.
            \end{displaymath}
            In addition, $\max_{x,y}{\abs[\big]{\frac{p_{t}(x,y)-\pi(y)}{\pi(y)}}}$ is decreasing in $t$.
    \end{lemma}
    \begin{proof}
        One may refer to the proof of \autocite[Proposition~A.1]{Lyons-2018}. The monotonicity is mentioned in \autocite{Lyons-2021}.
    \end{proof}
    \begin{proof}[Proof of \Cref{mixx}]
        a) Combining \Cref{mix_mono_prep} and \Cref{mix_even_prep}, for $t\equiv 0\bmod 2$, we have
            \begin{displaymath}
                \max_{x,y}{\abs[\Big]{\frac{p_{t}(x,y)-\pi(y)}{\pi(y)}}}\leqslant \max_x\frac{p_{t}(x,x)-\pi(x)}{\pi(x)}\leqslant
                \frac{2\bigl(\D(G)+1\bigr)\vol(V)}{t}.
            \end{displaymath}
            In addition, we notice that $\vol(V)\leqslant n(n-1)$. Because
            \begin{displaymath}
                \frac{2\bigl(\D(G)+1\bigr)\vol(V)}{8n^3}\leqslant \frac{2n\cdot n(n-1)}{8n^3}\leqslant \frac{1}{4},
            \end{displaymath}
            the second assertion of \Cref{mix_mono_prep} ensures
            \begin{displaymath}
                \tiq\leqslant 8n^3.
            \end{displaymath}
            \par
        b) To prove the second assertion, we assume $G$ is $d$-regular. By \Cref{reg_diam}, we have
            \begin{displaymath}
                \D(G)+1\leqslant \frac{3n}{d}.
            \end{displaymath}
            In addition, $\vol(V)=nd$.
            Therefore, we have
            \begin{displaymath}
                \frac{2\bigl(\D(G)+1\bigr)\vol(V)}{24n^2}\leqslant \frac{1}{4}.
            \end{displaymath}
            Consequently, the second assertion of \Cref{mix_mono_prep} gives
            \begin{displaymath}
                \tiq
                \leqslant 24n^2.\qedhere
            \end{displaymath}
    \end{proof}

\clearpage
\section{Transitive Case}\label{sec_trans}
    We consider vertex-transitive graphs in this section. Recall that a graph $G$ is vertex-transitive if for every two vertices $x$ and $y$ of $G$, there is an automorphism $\varphi\colon G\to G$ such that $\varphi(x)=y$. When a graph $G$ is vertex-transitive, all vertices of $G$ have the same weight, denoted by $w$.
    \subsection{Estimate of Spectral Measure}
        Given a vertex-transitive graph $G$, set
            \begin{displaymath}
                N^\#(r)\coloneqq \abs[\big]{\{v\in V\st \dist(v,x)\leqslant r\}}.
            \end{displaymath}
        This value does not depend on the choice of $x$ because $G$ is vertex-transitive. Furthermore, as in
        \autocite[Lemma~6.4]{Lyons-2018}, for every two vertices $x, y \in V$, it is easy to see that $\norm{F^\cq_x}_w=\norm{F^\cq_y}_w$.
        \begin{lemma}
            When $G$ is vertex-transitive, for each $x\in V$,
                \begin{displaymath}
                    \cq F_x^\cq=\sum_{v\in N(x)} \frac{w(x,v)}{w}(F_x^\cq+F_v^\cq).
                \end{displaymath}
        \end{lemma}
        \begin{proof}
            By the definition of $F_x^\cq$,
                \begin{displaymath}
                    \cq F_x^\cq=\frac{\cq I_\cq(\delta)\ind_x}{w}=\frac{I_\cq(\delta)\cq \ind_x}{w}.
                \end{displaymath}
                However, by direct calculation,
                    \begin{displaymath}
                        \cq \ind_x=\sum_{v\in N(x)}\frac{w(x,v)}{w}(\ind_x+\ind_v).
                    \end{displaymath}
                Thus, we have
                \begin{displaymath}\begin{split}
                    \cq F_x^\cq&=\frac{1}{w}I_\cq(\delta)\sum_{v\in N(x)}\frac{w(x,v)}{w}(\ind_x+\ind_v)\\
                    &=
                    \sum_{v\in N(x)}\frac{w(x,v)}{w}\Bigl(\frac{I_\cq(\delta)\ind_x}{w}+\frac{I_\cq(\delta)\ind_v}{w}\Bigr)\\
                    &=\sum_{v\in N(x)}\frac{w(x,v)}{w}(F_x^\cq+F_v^\cq).\qedhere
                \end{split}\end{displaymath}
        \end{proof}
        The following \Cref{trans_spec} is similar to the first assertion of \autocite[Theorem 6.1]{Lyons-2018}.
        \begin{theorem}\label{trans_spec}
            Let $G$ be a vertex-transitive, non-bipartite, simple, connected, weighted graph with weight at least\/ $1$ for each edge. For each $x\in V$, $c\in(0,1)$, and $\delta\in\bigl(0,\frac{2}{w}\bigr]$, we have
            \begin{displaymath}
                \mu^\cq_x(\delta)\leqslant
                \frac{1}{c^2N^\#\Bigl(\frac{\arcsin\sqrt{(1-c)/2}}{\arcsin\sqrt{w\delta/2}}\,\Bigr)}.
            \end{displaymath}
            Furthermore, for $\delta\in (0,2]$,
            \begin{displaymath}
                \mu^\cq_x(\delta)
                \leqslant
                \frac{1}{c^2N^\#\Bigl(\frac{\sqrt{1-c}}{\sqrt{w\delta}}\Bigr)}.
            \end{displaymath}
        \end{theorem}
        \begin{proof}
            Fix a vertex $x\in V$.
            \par
            a) Consider the spectral embedding $\{F^\cq_v\st v\in V\}$ based on $\cq$. Let
                \begin{displaymath}
                    \rho\coloneqq \norm{F^\cq_x}_w,\qquad \beta_0\coloneqq \max_{v\in N(x)}\norm{F^\cq_x+F^\cq_v}_w.
                \end{displaymath}
                We have
                \begin{displaymath}\begin{split}
                    \delta&\geqslant \frac{1}{\norm{F^\cq_x}_w^2}\bigl\langle \cq F^\cq_x,F^\cq_x\bigr\rangle_w
                    = \frac{1}{\norm{F^\cq_x}_w^2}\biggl\langle \sum_{v\in N(x)}\frac{w(x,v)}{w}(F^\cq_x+F^\cq_v),F^\cq_x \biggr\rangle_w \\
                    &= \frac{1}{\norm{F^\cq_x}_w^2}\sum_{v\in N(x)}\frac{w(x,v)}{w}\bigl(\norm{F^\cq_x}_w^2+
                    \bigl\langle F^\cq_v,F^\cq_x\bigr\rangle_w\bigr)\\
                    &= \frac{1}{\norm{F^\cq_x}_w^2}\sum_{v\in N(x)}\frac{w(x,v)}{w}\frac{\norm{F^\cq_x+F^\cq_v}_w^2}{2}\\
                    &\geqslant \frac{1}{2w \rho^2}\max_{v\in N(x)}\norm{F^\cq_x+F^\cq_v}^2_w=\frac{1}{2w \rho^2}\beta_0^2.
                \end{split}\end{displaymath}
                In other words, $\frac{\beta_0}{2\rho}\leqslant \sqrt{w \delta/2}$. Since we assumed that $\delta\in (0,\frac{2}{w}]$, by the monotonicity of $\arcsin$ function on the interval $[0,1]$, we have
                \begin{displaymath}
                    2\arcsin{\frac{\beta_0}{2\rho}}\leqslant 2\arcsin{\sqrt{w \delta/2}}.
                \end{displaymath}
            \par
            b) Set $B_0\coloneqq  B\Bigl(x,\frac{\arcsin{\sqrt{(1-c)/2}}}{\arcsin{\sqrt{(w \delta)/2}}}\Bigr)$. \vadjust{\kern2pt}%
            For $v\in B_0$, let $\cp=u_0u_1\cdots u_{r-1}u_r$ be a shortest path joining $x$ and $v$, with $u_0=x$ and $u_r=v$. Define
                \begin{displaymath}
                    \widetilde{F}_i\coloneqq (-1)^iF^\cq_{u_i},\qquad i=0,1,2,\ldots,r.
                \end{displaymath}
                Then $r\leqslant \frac{\arcsin{\sqrt{(1-c)/2}}}{\arcsin{\sqrt{(w \delta)/2}}}$, and
                    \begin{displaymath}
                        \norm[\big]{\widetilde{F}_i-\widetilde{F}_{i+1}}_w
                        =\norm[\big]{F^\cq_{u_i}+F^\cq_{u_{i+1}}}_w, \qquad i=0,1,2,\ldots,r-1.
                    \end{displaymath}
                Write $\theta(h_1,h_2)$ for the angle between $h_1,h_2\in \ell^2(V,w)$. Then using the sphere metric, we get
                \begin{equation}\label{trans_cal}\begin{split}
                    \theta\bigl(\widetilde{F}_0,\widetilde{F}_r\bigr)
                    &\leqslant \sum_{i=0}^{r-1}\theta\bigl(\widetilde{F}_i,\widetilde{F}_{i+1}\bigr)\\
                    &=\sum_{i=0}^{r-1}2\arcsin{\frac{\norm[\big]{\widetilde{F}_i-\widetilde{F}_{i+1}}_w}{2\rho}}
                    =\sum_{i=0}^{r-1}2\arcsin{\frac{\norm[\big]{F^\cq_{u_i}+F^\cq_{u_{i+1}}}_w}{2\rho}}\\
                    &\leqslant \sum_{i=0}^{r-1}2\arcsin{\frac{\beta_0}{2\rho}}
                    \leqslant  2r\arcsin{\sqrt{w \delta/2}}\\
                    &\leqslant 2\frac{\arcsin{\sqrt{(1-c)/2}}}{\arcsin{\sqrt{w \delta/2}}}\arcsin{\sqrt{w \delta/2}}
                    =2\arcsin{\sqrt{(1-c)/2}}\\
                    &=\arccos c.
                \end{split}\end{equation}
                Thus,
                \begin{displaymath}
                    \abs[\big]{\cos{\theta\bigl(F^\cq_x,F^\cq_v\bigr)}}=\cos{\theta\bigl(\widetilde{F}_0,\widetilde{F}_r\bigr)}\geqslant  c.
                \end{displaymath}
                In summary, if $v\in B_0=B\Bigl(x,\frac{\arcsin{\sqrt{(1-c)/2}}}{\arcsin{\sqrt{(w \delta)/2}}}\Bigr)$,
                \begin{displaymath}
                    \abs[\big]{\cos{\theta\bigl(F^\cq_x,F^\cq_v\bigr)}}\geqslant  c.
                \end{displaymath}
            \par
            c) By the above discussion,
                \begin{displaymath}\begin{split}
                    \rho^2&
                    =\norm{F^\cq_x}_w^2
                    =\sum_{v\in V}w \abs[\big]{F^\cq_x(v)}^2\\
                    &=\sum_{v\in V}w \abs[\big]{\bigl\langle F^\cq_x,F^\cq_v\bigr\rangle_w}^2
                    \geqslant w \sum_{v\in B_0}\abs[\big]{\bigl\langle F^\cq_x,F^\cq_v\bigr\rangle_w}^2\\
                    &= w \sum_{v\in B_0}\abs[\big]{\cos{\theta\bigl(F^\cq_x,F^\cq_v\bigr)}}^2\norm{F^\cq_x}_w^2\norm{F^\cq_v}_w^2 \\
                    &\geqslant w \rho^4c^2 N^\#\Bigl(\frac{\arcsin\sqrt{(1-c)/2}}{\arcsin\sqrt{w\delta/2}}\,\Bigr).
                \end{split}\end{displaymath}
                Thus, we have
                    \begin{displaymath}
                        w\rho^2
                        \leqslant \frac{1}{c^2N^\#\Bigl(\frac{\arcsin\sqrt{(1-c)/2}}{\arcsin\sqrt{w\delta/2}}\Bigr)}.
                    \end{displaymath}
                Now, by \Cref{spec_embed_norm},
                    \begin{displaymath}
                        \mu_x^\cq(\delta)=w\norm{F^\cq_x}_w^2
                        =w\rho^2
                        \leqslant \frac{1}{c^2N^\#\Bigl(\frac{\arcsin\sqrt{(1-c)/2}}{\arcsin\sqrt{w\delta/2}}\Bigr)}.
                    \end{displaymath}
                The first assertion is proved.
            \par
            d) Now we turn to the second assertion.

                If $(1-c)<w \delta$, then
                \begin{displaymath}
                    N^\#\Bigl(\frac{\sqrt{1-c}}{\sqrt{w\delta}}\,\Bigr)=1.
                \end{displaymath}
                Therefore,
                \begin{displaymath}
                \frac{1}{c^2N^\#\Bigl(\frac{\sqrt{1-c}}{\sqrt{w\delta}}\Bigr)}>1.
                \end{displaymath}
                But $\mu_x^\cq(\delta)\leqslant 1$. So the second assertion holds in this case.

                If $(1-c)\geqslant  w \delta$, then $\sqrt{w \delta/2}<1$. Thus, $\arcsin{\sqrt{w \delta/2}}$ is defined.             Note that $\frac{x}{\arcsin{x}}$ is decreasing for $x\in (0,1)$. Therefore,
                \begin{displaymath}
                    \frac{\sqrt{(1-c)/2}}{\arcsin{\sqrt{(1-c)/2}}}
                    \leqslant
                    \frac{\sqrt{w \delta/2}}{\arcsin{\sqrt{w \delta/2}}}.
                \end{displaymath}
                Consequently,
                \begin{displaymath}
                    \frac{\sqrt{1-c}}{\sqrt{w \delta}}
                    \leqslant
                    \frac{\arcsin{\sqrt{(1-c)/2}}}{\arcsin{\sqrt{w \delta/2}}}.
                \end{displaymath}
                Hence, the second assertion follows from the first one immediately.
        \end{proof}
        By the bounds on vertex spectral measures in \Cref{trans_spec} and \autocite[Theorem~6.1]{Lyons-2018}, using the method as in \Cref{reg_return}, we may get a similar result as \autocite[Corollary~6.6]{Lyons-2018}.
        \begin{corollary}\label{trans_return}
            Let $G$ be a finite or infinite, vertex-transitive, $d$-regular, non-bipartite, simple, connected, unweighted graph with at least polynomial growth rate $N^\#(r)\geqslant  Cr^D$, where $C > 0$ and $D \geqslant  1$ are constants and\/ $0 \leqslant  r \leqslant  \D(G)$.
            Then for each $x\in V$ and $t > 0$,
            \begin{displaymath}\begin{split}
                0\leqslant  p_t(x,x)-\pi(x)
                &\leqslant 2\widetilde{C}t^{-D/2}  \qquad \text{for\quad}  t\equiv 0\bmod 2,\\
                \abs[\big]{p_t(x,x)-\pi(x)}
                &\leqslant \widetilde{C}t^{-D/2}  \qquad \text{for\quad}  t\equiv 1\bmod 2,
            \end{split}\end{displaymath}
        where $\widetilde{C}=\frac{(D+4)^{D/2+2}d^{D/2}}{32CD^{D/2-1}}\Gamma\bigl(\frac{D}{2}\bigr)$.\qed
        \end{corollary}
        The proof of this corollary is omitted.
        An analogue of \autocite[Corollary~6.7]{Lyons-2018} could also be written down easily.
    \subsection{Minimum Eigenvalue}
        For a connected, finite graph $G$ and its signless probabilistic Laplacian $\cq=I+P$, we know that $0$ is an eigenvalue of $\cq$ if and only if $G$ is bipartite from \Cref{0span}.
        \begin{theorem}\label{trans_mini}
            Let $G$ be a vertex-transitive, non-bipartite, finite, simple, connected, weighted graph with weight at least\/ $1$ for each edge. Then
            \begin{displaymath}
                \lambda^\cq_{\mathrm{min}}
                \geqslant
                \frac{2}{w}\Bigl(\sin\frac{\pi}{4\bigl(\D(G)+1\bigr)}\Bigr)^2.
            \end{displaymath}
        \end{theorem}
        \Cref{trans_mini} is a partial analogue of the second assertion of \autocite[Theorem 6.1]{Lyons-2018}.
        \begin{proof}
            We may assume $\frac{\lambda^\cq_{\mathrm{min}}}{2}<1$: otherwise, the inequality is trivial.
            Consider the spectral embedding based on $\cq$ with $\delta=\lambda^\cq_{\mathrm{min}}$. Fix $x\in V$ and let
            \begin{displaymath}
                S\coloneqq \bigl\{s\in V\st f(s)\geqslant 0\bigr\},\qquad T\coloneqq \bigl\{t\in V\st f(t)<0\bigr\}.
            \end{displaymath}
            It is easy to see that both $S$ and $T$ are non-empty. Because $G$ is assumed to be non-bipartite, there exists an edge $(s_1,s_2)\in E$ with $s_1,s_2\in S$, or an edge $(t_1,t_2)\in E$ with $t_1,t_2\in T$. \vadjust{\kern2pt}%
            \par
            a) If there is an edge $(s_1,s_2)\in E$ with $s_1,s_2\in S$, \vadjust{\kern2pt}%
            let $\widehat{\cp}_0$ be a shortest path from $x$ to $\{s_1,s_2\}$. Without loss of generality, we may assume $\widehat{\cp}_0$ is from $x$ to $s_1$. \vadjust{\kern2pt}%
            If $\abs{\widehat{\cp}_0}$ is odd, we set $\widehat{\cp}\coloneqq \widehat{\cp}_0$; if $\abs{\widehat{\cp}_0}$ is even, we set $\widehat{\cp}\coloneqq \widehat{\cp}_0.(s_1,s_2)$, the concatenation of $\widehat{\cp}_0$ and the edge $(s_1,s_2)$. Hence, $\widehat{\cp}$ is a path of odd length in any case. Assume
            \begin{displaymath}
                \widehat{\cp}=u_0u_1\cdots u_k,
            \end{displaymath}
            with $u_0=x$. We define
                \begin{displaymath}
                    \widehat{F}_i\coloneqq (-1)^iF^\cq_{u_i},\qquad i=0,1,2,\ldots,k.
                \end{displaymath}
            Then $\widehat{F}_0= F^\cq_x$ and
            \begin{displaymath}\begin{split}
                \bigl\langle\widehat{F}_0,\widehat{F}_k\bigr\rangle_w
                &= (-1)^k \bigl\langle F^\cq_x,F^\cq_{u_k}\bigr\rangle_w
                =-\bigl\langle F^\cq_x,F^\cq_{u_k}\bigr\rangle_w
                =-\langle I_\cq(\delta)\ind_x/w,I_\cq(\delta)\ind_{u_k}/w\bigr\rangle_w \\
                &=-\langle I_\cq(\delta)\ind_x/w,\ind_{u_k}/w\bigr\rangle_w
                =-\langle I_\cq(\delta)\ind_x/w,\ind_{u_k}\bigr\rangle
                =-\langle F^\cq_x,\ind_{u_k}\bigr\rangle\\
                &=-F^\cq_x(u_k)\leqslant 0.
            \end{split}\end{displaymath}
            Hence,
            \begin{equation}\label{dunjiao_1}
                \frac{\pi}{2}\leqslant \theta\bigl(\widehat{F}_0,\widehat{F}_k\bigr).
            \end{equation}
            In addition, using the sphere metric as in \Cref{trans_cal}, we have
                \begin{equation}\label{angle_est_1}
                    \theta\bigl(\widehat{F}_0,\widehat{F}_k\bigr)
                    \leqslant 2k\arcsin{\sqrt{w {\lambda}^\cq_{\mathrm{min}}/2}}
                    \leqslant 2\bigl(\D(G)+1\bigr)\arcsin{\sqrt{w {\lambda}^\cq_{\mathrm{min}}/2}}.
                \end{equation}
            Combining \Cref{dunjiao_1,angle_est_1},
                \begin{displaymath}
                    \frac{\pi}{2}\leqslant 2\bigl(\D(G)+1\bigr)\arcsin{\sqrt{w {\lambda}^\cq_{\mathrm{min}}/2}}.
                \end{displaymath}
            So in this case,
            \begin{displaymath}
                \lambda^\cq_{\mathrm{min}}
                \geqslant
                \frac{2}{w}\Bigl(\sin\frac{\pi}{4\bigl(\D(G)+1\bigr)}\Bigr)^2.
            \end{displaymath}
            \par
            b) If there is an edge $(t_1,t_2)\in E$ with $t_1,t_2\in T$, \vadjust{\kern2pt}%
            let $\overline{\cp}_0$ be a shortest path from $x$ to $\{t_1,t_2\}$. Without loss of generality, we may assume $\overline{\cp}_0$ is from $x$ to $t_1$. \vadjust{\kern2pt}%
            If $\abs{\overline{\cp}_0}$ is even, we set $\overline{\cp}\coloneqq \overline{\cp}_0$; if $\abs{\overline{\cp}_0}$ is odd, we set $\overline{\cp}\coloneqq \overline{\cp}_0.(t_1,t_2)$. Thus, $\overline{\cp}$ is a path of even length in any case. Assume
            \begin{displaymath}
                \overline{\cp}=v_0v_1\cdots v_\ell,
            \end{displaymath}
            with $v_0=x$, and define
                \begin{displaymath}
                    \overline{F}_i\coloneqq (-1)^iF^\cq_{v_i},\qquad i=0,1,2,\ldots,\ell.
                \end{displaymath}
            Then $\overline{F}_0= F^\cq_x$ and
            \begin{displaymath}
                \bigl\langle\overline{F}_0,\overline{F}_\ell\bigr\rangle_w
                = (-1)^\ell \bigl\langle F^\cq_x,F^\cq_{v_\ell}\bigr\rangle_w
                =F^\cq_x(v_\ell)< 0.
            \end{displaymath}
            Hence, we have
            \begin{equation}\label{dunjiao_2}
                \frac{\pi}{2}<\theta\bigl(\overline{F}_0,\overline{F}_\ell\bigr).
            \end{equation}
            In addition, using the sphere metric as in \Cref{trans_cal} again, we have
                \begin{equation}\label{angle_est_2}
                    \theta\bigl(\overline{F}_0,\overline{F}_r\bigr)
                    \leqslant 2\ell\arcsin{\sqrt{w {\lambda}^\cq_{\mathrm{min}}/2}}
                    \leqslant 2\bigl(\D(G)+1\bigr)\arcsin{\sqrt{w {\lambda}^\cq_{\mathrm{min}}/2}}.
                \end{equation}
            Combining \Cref{dunjiao_2,angle_est_2}, we get
                \begin{displaymath}
                    \frac{\pi}{2}< 2\bigl(\D(G)+1\bigr)\arcsin{\sqrt{w {\lambda}^\cq_{\mathrm{min}}/2}}.
                \end{displaymath}
            So in this case,
            \begin{displaymath}
                \lambda^\cq_{\mathrm{min}}
                >
                \frac{2}{w}\Bigl(\sin\frac{\pi}{4\bigl(\D(G)+1\bigr)}\Bigr)^2.\qedhere
            \end{displaymath}
        \end{proof} 


\section{Average Return Probability}\label{sec_avg}
    In this section, we deal with the average spectral measure of $\cq$ and average return probabilities for simple random walk on generic non-bipartite, simple, finite, connected graphs. Our method is inspired by \autocite[Section~5]{Lyons-2018}.
    \subsection{Estimate of Average Spectral Measure}
        Define average spectral measure of $\cq$ as
        \begin{displaymath}
            \mu^\cq(\delta)\coloneqq \frac{1}{n}\sum_{x\in V}\mu^\cq_x(\delta).
        \end{displaymath}
        We also write $\mu^\cq_S(\delta)\coloneqq \sum_{x\in S}\mu^\cq_x(\delta)$ for $S \subseteq V$.
        \begin{theorem}\label{spec_avg}
          For any non-bipartite, finite, simple, connected, unweighted graph $G$ and ${\delta\in(0,2)}$, we have
            \begin{displaymath}
                \mu^\cq(\delta)<(4000\delta)^{1/3}.
            \end{displaymath}
        \end{theorem}
        \Cref{spec_avg} is comparable to \autocite[Theorem 5.1]{Lyons-2018}.
        \begin{corollary}\label{spec_avg_ct}
            For any non-bipartite, finite, simple, connected, unweighted graph $G$, we have\/ $\lambda^\cq_k\geqslant \frac{k^3}{4000n^3}$ for\/ $1\leqslant k\leqslant n$. Therefore, $\lambda^P_k\geqslant -1+\frac{k^3}{4000n^3}$.
        \end{corollary}
        \begin{proof}
            By the same argument as in the proof of \Cref{lbb_reg}, this corollary follows easily from \Cref{spec_loc} and \Cref{spec_avg}.
        \end{proof}
        To prove \Cref{spec_avg}, we need some preparation.
        \par
        Let $m\coloneqq \bigl\lfloor\frac{n}{2}\mu^\cq(\delta)\bigr\rfloor+1$. For each $x\in V$, set
            \begin{displaymath}
                \R(x)\coloneqq \Bigl\{
                y\in V\st \norm{F_x^\cq-F_y^\cq}_w \leqslant \frac{\norm{F_x^\cq}_w}{4} \text{ or }\norm{F_x^\cq+F_y^\cq}_w \leqslant \frac{\norm{F_x^\cq}_w}{4}
                \Bigr\}.
            \end{displaymath}
        \par
        We will use \Cref{alg:setselection} to get some useful sets.
            \begin{algorithm}[htbp]
                \begin{algorithmic}
                \STATE Let $S_0\leftarrow V$.
                \FOR {$i=1 \to m$}
                \STATE Choose a vertex $x_i$ in $S_{i-1}$ that maximizes $\mu^\cq_{x_i}(\delta)$.
                \STATE Let $S_{i}\leftarrow S_{i-1}\setminus \R(x_i)$.
                \ENDFOR
                \RETURN $\R(x_1),\R(x_2),\ldots,\R(x_m)$.
                \end{algorithmic}
            \caption{Set-Selection.}
            \label{alg:setselection}
            \end{algorithm}
        \par
        For each $x\in V$, set
            \begin{displaymath}
            \widetilde{N}(x)\coloneqq
                \Bigl\{ y\in N(x)\st \norm{F_y^\cq+F_x^\cq}_w \leqslant \frac{\norm{F_x^\cq}_w}{9}\Bigr\}.
            \end{displaymath}
        Let $\widetilde{T}(x)$ be the star formed by $x$ and $\widetilde{N}(x)$. We pick $T(x)$ as a maximal (with respect to inclusion) connected bipartite graph including $\widetilde{T}(x)$ such that
            \begin{displaymath}
                \all{y\in V\bigl(T(x)\bigr)}\quad
                \norm[\big]{F^\cq_x-(-1)^{\dist(x,y;T(x))}F_y^\cq}_w\leqslant \frac{\norm{F_x^\cq}_w}{9}.
            \end{displaymath}
        \begin{lemma}\label{lm51}
            Let $G$ be a non-bipartite, finite, simple, connected, unweighted graph and $\delta\geqslant \lambda^\cq_\mathrm{min}$.
            \begin{enumerate}
              \item[\upshape(0)] For $i=1,2,\ldots, m$, we have $\mu^\cq_{\R(x_i)}(\delta)\leqslant \frac{4}{3}$.
              \item[\upshape(1)] For $i=1,2,\ldots, m$, we have $\mu^\cq_{x_i}(\delta)\geqslant \frac{\mu^\cq(\delta)}{3}$. In addition, \Cref{alg:setselection} is well designed: each $S_{i-1}$ is non-empty and $x_i$ could be chosen for $i=1,2,\ldots, m$. Therefore, \Cref{alg:setselection} is not stopping before $i=m$.
              \item[\upshape(2)] For\/ $1\leqslant i<j\leqslant m$, $V(T(x_i))\cap V(T(x_j))=\varnothing$.
            \end{enumerate}
        \end{lemma}
        \Cref{lm51} plays a similar role to that of Lemma 5.2 in \autocite{Lyons-2018}.
        \begin{proof}
            (0) Since $\delta\geqslant \lambda^\cq_\mathrm{min}$, we must have $\norm{F_{x_i}^\cq}_w>0$. By the proof of \autocite[Lemma~3.13]{Lyons-2018}, for any two non-zero vectors in a Hilbert space $\cH$,
                \begin{equation}\label{aux}
                    \Re\Bigl\langle\frac{f}{\norm{f}_\cH},\frac{g}{\norm{g}_\cH}\Bigr\rangle_\cH
                    \geqslant
                    1-\frac{2\norm{f-g}^2_\cH}{\norm{f}_\cH^2},
                \end{equation}
                where $\Re z$ is the real part of a complex number $z$. Also, for $y\in \R(x_i)$, by the construction of $\R(x_i)$, we have an integer $\sigma_i(y)$ such that
                \begin{displaymath}
                    \norm[\big]{F^\cq_{x_i}-(-1)^{\sigma_i(y)}F_y^\cq}_w\leqslant \frac{\norm{F_{x_i}^\cq}_w}{4}.
                \end{displaymath}
                This implies $\norm{F^\cq_y}_w>0$ for $y\in\R(x_i)$.
                \par
                Therefore,
                \begin{displaymath}\begin{split}
                    1&= \sum_{y\in \R(x_i)}w(y)\norm{F^\cq_{y}}_w^2\abs[\Big]{\Bigl\langle \frac{F^\cq_y}{\norm{F^\cq_{y}}_w},\frac{F^\cq_{x_i}}{\norm{F^\cq_{x_i}}_w}\Bigr\rangle_w}^2\\
                    &= \sum_{y\in \R(x_i)}\mu_y^\cq(\delta)\abs[\Big]{\Bigl\langle \frac{(-1)^{\sigma_i(y)}F^\cq_y}{\norm{(-1)^{\sigma_i(y)}F^\cq_{y}}_w},\frac{F^\cq_{x_i}}{\norm{F^\cq_{x_i}}_w}\Bigr\rangle_w}^2\\
                    &\geqslant \sum_{y\in\R(x_i)}\mu^\cq_y(\delta)
                    \Bigl(1-\frac{2\norm{F^\cq_{x_i}-(-1)^{\sigma_i(y)}F^\cq_y}^2_w}{\norm{F^\cq_{x_i}}^2_w}\Bigr)^2\\
                    &\geqslant \mu^\cq_{\R(x_i)}(\delta)\Bigl(1-\frac{2}{16}\Bigr)^2=\frac{49}{64}\mu^\cq_{\R(x_i)}(\delta),
                \end{split}\end{displaymath}
                where we use \Cref{aux} to get the first inequality.

                It follows that
                    \begin{displaymath}
                        \mu^\cq_{\R(x_i)}(\delta)\leqslant \frac{64}{49}< \frac{4}{3}.
                    \end{displaymath}
            \par
            (1) We know $\mu_{S_0}^\cq(\delta)=\mu_V^\cq(\delta)=n\mu^\cq(\delta)$. In addition, by assertion (0), the total spectral measure of removed vertices in each iteration of the for loop in \Cref{alg:setselection} is at most $\frac{4}{3}$. We have, for each $i\leqslant m$,
                \begin{displaymath}
                    \mu_{S_{i-1}}^\cq(\delta)\geqslant n\mu^\cq(\delta)-\frac{4}{3}(m-1)\geqslant\frac{n\mu^\cq(\delta)}{3},
                \end{displaymath}
                where the last inequality holds by the definition of $m$. This implies $S_{i-1}\neq\varnothing$ and $x_i$ could be chosen in \Cref{alg:setselection} for $i=1,2,\ldots, m$. In other words, the algorithm is well designed and is not stopping before $i=m$.
                \par
                Furthermore, since $x_i$ has the largest vertex spectral measure in $S_{i-1}$ for $i=1,2,\ldots, m$, we have
                \begin{displaymath}
                    \mu_{x_i}^\cq(\delta)\geqslant \frac{\mu_{S_{i-1}}^\cq(\delta)}{n}\geqslant\frac{\mu_{S_{m-1}}^\cq(\delta)}{n}
                    \geqslant \frac{\mu^\cq(\delta)}{3}.
                \end{displaymath}
            \par
            (2) Suppose that some vertex $y$ lies in $V\bigl(T(x_i)\bigr)\cap V\bigl(T(x_j)\bigr)$. We deal with assertion (2) in the following two cases. For ease of notation, we write $\tau_k\coloneqq \dist(x_k,y;T(x_k))$ for $k=i,j$. By the construction of $T(x_i)$ and $T(x_j)$,
                \begin{displaymath}
                    \norm[\big]{F^\cq_{x_k}-(-1)^{\tau_k}F_y^\cq}_w\leqslant \frac{\norm{F_{x_k}^\cq}_w}{9},\qquad k=i,j.
                \end{displaymath}
                \par\underline{Case (a)}: $\norm{F^\cq_{x_j}}_w>\frac{5}{4}\norm{F^\cq_{x_i}}_w$. In this case, we have
                \begin{displaymath}\begin{split}
                    \norm[\big]{F^\cq_y}_w
                    &=
                    \norm[\big]{(-1)^{\tau_j}F_y^\cq}_w \\
                    &\geqslant
                    \frac{8}{9}\norm[\big]{F^\cq_{x_j}}_w
                    >
                    \frac{8}{9}\cdot\frac{5}{4}\norm[\big]{F^\cq_{x_i}}_w \\
                    &=
                    \frac{10}{9}\norm[\big]{F^\cq_{x_i}}_w.
                \end{split}\end{displaymath}
                Thus, we have
                    \begin{displaymath}
                        \norm[\big]{(-1)^{\tau_i}F_y^\cq}_w
                        =\norm[\big]{F^\cq_y}_w
                        >\frac{10}{9}\norm[\big]{F^\cq_{x_i}}_w.
                    \end{displaymath}
                This contradicts our assumption that $y\in V\bigl(T(x_i)\bigr)$.
                \par\underline{Case (b)}: $\norm{F^\cq_{x_j}}_w \leqslant \frac{5}{4}\norm{F^\cq_{x_i}}_w$.
                In this case, we have
                \begin{displaymath}\begin{split}
                    \norm[\big]{F^\cq_{x_i}-(-1)^{\tau_i+\tau_j}F_{x_j}^\cq}_w
                    &= \norm[\big]{F^\cq_{x_i}-(-1)^{\tau_i}F_y^\cq+(-1)^{\tau_i}F_y^\cq-(-1)^{\tau_i+\tau_j}F_{x_j}^\cq}_w \\
                    &\leqslant \norm[\big]{F^\cq_{x_i}-(-1)^{\tau_i}F_y^\cq}_w+\norm[\big]{(-1)^{\tau_i}F_y^\cq-(-1)^{\tau_i+\tau_j}F_{x_j}^\cq}_w \\
                    &= \norm[\big]{F^\cq_{x_i}-(-1)^{\tau_i}F_y^\cq}_w+\norm[\big]{F_{x_i}^\cq-(-1)^{\tau_j}F_{y}^\cq}_w \\
                    &\leqslant \frac{\norm{F^\cq_{x_i}}_w}{9}+\frac{\norm{F^\cq_{x_j}}_w}{9}
                    \leqslant \Bigl(\frac{1}{9}+\frac{1}{9}\cdot\frac{5}{4}\Bigr)\norm{F^\cq_{x_i}}_w \\
                    &=\frac{\norm{F^\cq_{x_i}}_w}{4}.
                \end{split}\end{displaymath}
                Thus, $x_j\in\R(x_i)$. But this is impossible.

                Therefore, we may conclude that $V(T(x_i))\cap V(T(x_j))=\varnothing$.
        \end{proof}
        For a subset $E'\subseteq E(G)$ and a mapping $F\colon V\to \cH$ to a Hilbert space $\cH$, we define energy
            \begin{displaymath}
                \widetilde\cE(E')\coloneqq \sum_{(x,y)\in E'}\norm{F_x+F_y}_\cH^2.
            \end{displaymath}
        For $S\subseteq V$, let $E(S)$ be the collection of edges that are incident with the vertices in $S$. We define the energy of $S$ as $\widetilde\cE(S)\coloneqq  \widetilde\cE\bigl(E(S)\bigr)$.
        \begin{lemma}\label{lm52}
            Assume that the graph $G$ is a non-bipartite, finite, simple, connected, unweighted graph. When $\delta\geqslant \lambda^\cq_\mathrm{min}$, for\/ $1\leqslant i\leqslant m$,
                \begin{displaymath}
                    \widetilde{\cE} \bigl(V(T(x_i))\bigr)
                    >\frac{\mu^\cq(\delta)}{250\abs[\big]{V(T(x_i))}^2}.
                \end{displaymath}
        \end{lemma}
        \Cref{lm52} plays a similar role to that of Lemma 5.3 in \autocite{Lyons-2018}.
        \begin{proof}
            a) We first consider the case $w(x_i)>\abs[\big]{V(T(x_i))}$. Recall that $w(x_i)$ equals the degree of $x_i$, since we are considering unweighted graphs. In this case,
                \begin{displaymath}\begin{split}
                    \widetilde{\cE} \bigl(V(T(x_i))\bigr)
                    &\geqslant \sum_{y\in N(x_i)\setminus V(T(x_i))}\norm{F^\cq_{x_i}+F^\cq_y}_w^2\\
                    &> \abs[\big]{N(x_i)\setminus V(T(x_i))}\cdot\frac{\norm{F^\cq_{x_i}}_w^2}{81}\\
                    &=\abs[\big]{N(x_i)\setminus V(T(x_i))}\cdot\frac{\mu_{x_i}^\cq(\delta)}{81 w(x_i)}\\
                    &>\bigl(w(x_i)-\abs[\big]{V(T(x_i))}+1\bigr)\frac{\mu^\cq(\delta)}{250w(x_i)}\\
                    &\geqslant \frac{\mu^\cq(\delta)}{250\abs[\big]{V(T(x_i))}^2},
                \end{split}\end{displaymath}
                where the second inequality holds because $V(T(x_i))\supseteq\widetilde{N}(x_i)$, and the third inequality holds thanks to \Cref{lm51}(1).
            \par
            b) We now turn to the case $w(x_i)\leqslant \abs[\big]{V(T(x_i))}$. In this case, we claim that there must be a path $z_0z_1\cdots z_\ell$, such that $\ell\leqslant \abs[\big]{V(T(x_i))}$, $z_0=x_i$, and
            \begin{displaymath}
                    \norm[\big]{F^\cq_{x_i}-(-1)^{\ell}F_{z_\ell}^\cq}_w> \frac{\norm{F_{x_i}^\cq}_w}{9}.
            \end{displaymath}
            In fact, if $V\bigl(T(x_i)\bigr)$ is a proper subset of $V$, there is a vertex $y\in V\setminus V(T(x_i))$ with $\dist\bigl(y,T(x_i);G\bigr)=1$. Let $z_0z_1\cdots z_\ell$ be a path joining $z_0\coloneqq  x_i$ and $z_\ell\coloneqq  y$, with $z_0z_1\cdots z_{\ell-1}$ being a path in $T(x_i)$. Because $T(x_i)$ is assumed maximal and $y\notin V\bigl(T(x_i)\bigr)$, we must have
                \begin{displaymath}
                    \norm[\big]{F^\cq_{x_i}-(-1)^{\ell}F_{z_\ell}^\cq}_w> \frac{\norm{F_{x_i}^\cq}_w}{9}.
                \end{displaymath}
            If $V\bigl(T(x_i)\bigr) = V$, since $T(x_i)$ is assumed to be a maximal connected bipartite graph and $G$ is a connected non-bipartite graph, there must be an edge $e=(u,v)$, whose addition to $T(x_i)$ results in an odd cycle. To be specific, we have
                \begin{displaymath}
                    \dist\bigl(x_i,u;T(x_i)\bigr)\equiv \dist\bigl(x_i,v;T(x_i)\bigr) \mod 2.
                \end{displaymath}
            Write $\tau\coloneqq \dist\bigl(x_i,u;T(x_i)\bigr)$. Then
                \begin{displaymath}
                    \norm[\big]{F^\cq_{x_i}-(-1)^\tau F_u^\cq}_w \leqslant  \frac{\norm{F_{x_i}^\cq}_w}{9}.
                \end{displaymath}
            This implies
                \begin{equation}\label{1step}
                    \norm[\big]{F^\cq_{x_i}-(-1)^{\tau+1} F_u^\cq}_w >  \frac{\norm{F_{x_i}^\cq}_w}{9}.
                \end{equation}
            In fact, if otherwise $\norm[\big]{F^\cq_{x_i}-(-1)^{\tau+1} F_u^\cq}_w \leqslant   \frac{\norm{F_{x_i}^\cq}_w}{9}$, then
                \begin{displaymath}
                    \norm[\big]{2F^\cq_{x_i}}_w
                    =\norm[\big]{F^\cq_{x_i}-(-1)^\tau F_u^\cq+F^\cq_{x_i}-(-1)^{\tau+1} F_u^\cq}_w
                    <\Bigl(\frac{1}{9}+\frac{1}{9}\Bigr)\norm[\big]{F^\cq_{x_i}}_w.
                \end{displaymath}
            Therefore, $\norm[\big]{F^\cq_{x_i}}_w=0$. But $\norm[\big]{F^\cq_{x_i}}_w \neq0$ for $\delta\geqslant \lambda^\cq_\mathrm{min}$.

            Now pick a shortest path $z_0z_1\cdots z_{\ell-1}$ joining $x_i$ and $v$ in $T(x_i)$, where $z_0=x_i$ and $z_{\ell-1}=v$. Let $z_\ell=u$. Then $z_0z_1\cdots z_{\ell-1}z_\ell$ is a path joining $x_i$ and $u$. Hence, by our assumption,
                \begin{displaymath}
                    \ell-1\equiv \tau\mod 2.
                \end{displaymath}
            Therefore, \Cref{1step} can be written as
                \begin{displaymath}
                    \norm[\big]{F^\cq_{x_i}-(-1)^{\ell} F_{z_\ell}^\cq}_w >  \frac{\norm{F_{x_i}^\cq}_w}{9}.
                \end{displaymath}
            Our claim is proved. By the claim, we have
                \begin{displaymath}\begin{split}
                    \widetilde{\cE} \bigl(V(T(x_i))\bigr)
                    &\geqslant
                    \sum_{k=0}^{\ell-1} \norm{F^\cq_{z_k}+F^\cq_{z_{k+1}}}_w^2
                    =
                    \sum_{k=0}^{\ell-1} \norm{(-1)^k F^\cq_{z_k}-(-1)^{k+1}F^\cq_{z_{k+1}}}_w^2\\
                    &\geqslant
                    \frac{1}{\ell}\Bigl(\sum_{k=0}^{\ell-1} \norm{(-1)^k F^\cq_{z_k}-(-1)^{k+1}F^\cq_{z_{k+1}}}_w\Bigr)^2
                    \geqslant
                    \frac{1}{\ell}\norm{ F^\cq_{z_0}-(-1)^\ell F^\cq_{z_{\ell}}}_w^2\\
                    &>
                    \frac{1}{\ell}\cdot\frac{1}{81}\norm{F^\cq_{x_i}}^2_w
                    \geqslant
                    \frac{1}{\abs[\big]{V(T(x_i))}}\cdot\frac{\norm{F^\cq_{x_i}}^2_w}{81}\\
                    &=
                    \frac{1}{\abs[\big]{V(T(x_i))}}\cdot\frac{\mu^\cq_{x_i}(\delta)}{81 w(x_i)}
                    >
                    \frac{1}{\abs[\big]{V(T(x_i))}}\cdot\frac{\mu^\cq(\delta)}{250 w(x_i)}\\
                    &\geqslant
                    \frac{\mu^\cq(\delta)}{250\abs[\big]{V(T(x_i))}^2},
                \end{split}\end{displaymath}
            where the sixth inequality follows from \Cref{lm51}(1), and the last inequality holds thanks to our assumption that $w(x_i)\leqslant \abs[\big]{V(T(x_i))}$. The proof is complete.
        \end{proof}
        \begin{proof}[Proof of \Cref{spec_avg}]
            To begin, we claim that
                \begin{displaymath}
                    \delta
                    \geqslant \frac{\sum_{(y,z)\in E(G)}\norm{F_y^\cq+F_z^\cq}_w^2}{\sum_{x\in V}\norm{F_x^\cq}_w^2w(x)}.
                \end{displaymath}
                This can be proved exactly in the same way as \autocite[Lemma~3.14]{Lyons-2018}. In fact, for all $x, y \in V$,
                \begin{displaymath}
                    F^\cq_x(y)
                    =
                    \big\langle I_\cq(\delta) \ind_x/w(x), \ind_y \big\rangle
                    =
                    \big\langle I_\cq(\delta) \ind_x/w(x), \ind_y/w(y)
                    \big\rangle_w
                    =
                    F^\cq_y(x).
                \end{displaymath}
            Therefore, using \Cref{q_form}, we have
            \begin{align*}
                \sum_{(y,z)\in E(G)} w(y,z) \norm{F^\cq_y+F^\cq_z}_w^2
                &=\sum_{(y,z)\in E(G)} w(y,z)\sum_{x\in V} w(x) \abs[\big]{\bigl(F^\cq_y+F^\cq_z\big)(x)}^2\\
                &=\sum_{(y,z)\in E(G)} w(y,z)\sum_{x\in V} w(x) \abs[\big]{F^\cq_x(y)+F^\cq_x(z)}^2.
            \end{align*}
            Proceeding further, we get that
            \begin{align*}
                \sum_{(y,z)\in E(G)} w(y,z) \norm{F^\cq_y+F^\cq_z}_w^2
                &=\sum_{x\in V} w(x) \sum_{(y,z)\in E(G)} w(y,z)\abs[\big]{F^\cq_x(y)+F^\cq_x(z)}^2\\
                &= {\sum_{x\in V} w(x) \langle F^\cq_x, \cq F^\cq_x\rangle_w}\\
                &\leqslant {\sum_{x\in V} \delta\,w(x) \norm{F^\cq_x}_w^2}.
            \end{align*}
            The claim is thus proved.
            \par
            Now we assume $\delta\geqslant \lambda^\cq_\mathrm{min}$ without loss of generality. Since for different $i$ and $j$,
            \begin{displaymath}
                V(T(x_i))\cap V(T(x_j))=\varnothing,
            \end{displaymath}
            we have $\sum_{i=1}^{m}\abs[\big]{V(T(x_i))}\leqslant n$. Therefore,
            \begin{displaymath}\begin{split}
                \delta
                &\geqslant \frac{\sum_{(x,y)\in E(G)}\norm{F_x^\cq+F_y^\cq}_w^2}{\sum_{y\in V}\norm{F_y^\cq}_w^2w(y)}
                \geqslant \frac{1}{n\mu^\cq(\delta)}\cdot\frac{1}{2}\sum^{m}_{i=1}\widetilde{\cE} \bigl(V\bigl(T(x_i)\bigr)\bigr)\\
                &>\frac{1}{2n\mu^\cq(\delta)}\sum^{m}_{i=1}\frac{\mu^\cq(\delta)}{250\abs[\big]{V(T(x_i))}^2}
                \geqslant \frac{m^3}{500n^3}
                \geqslant \frac{\mu^\cq(\delta)^3}{4000},
            \end{split}\end{displaymath}
            where the second inequality follows by \Cref{spec_embed_norm} and the fact
            that each edge is counted in at most two sets $T(x_i)$ for energy, the fourth inequality
            follows by convexity of the function ${s \mapsto 1/s^2}$,
            and the last inequality holds because $m\geqslant  n\mu^\cq(\delta)/2$. The proof of
            \Cref{spec_avg} is complete.
        \end{proof}
    \subsection{Average Return Probability}
        Using \Cref{spec_avg}, we can bound average return probabilities. \Cref{avg_even,avg_odd} together are comparable to \autocite[Corollary 5.4]{Lyons-2018}.
        \begin{theorem}\label{avg_even}
          Let $G$ be a non-bipartite, finite, simple, connected, unweighted graph. Then for\/ $t\equiv 0\bmod 2$, we have
            \begin{displaymath}
                0\leqslant
                \frac{\sum_{x\in V}p_t(x,x)-1}{n}
                \leqslant
                \frac{30}{t^{1/3}}.
            \end{displaymath}
        \end{theorem}
    \begin{proof}
        Set
        \begin{numcases}
            {\Phi(\lambda)\coloneqq }
            n\sqrt[3]{4000\lambda}
            \quad \nonumber &if\/ $\lambda\geqslant0\text{ and }n\sqrt[3]{4000\lambda}\leqslant n\mu^*(1)$,\\
            n\mu^*(1) \quad \nonumber &\text{for intermediate values of $\lambda$},\\
            n-1-n\sqrt[3]{4000(2-\lambda)}
            \quad \nonumber &if\/ $\lambda\leqslant2\text{ and }n-1-n\sqrt[3]{4000(2-\lambda)}\geqslant n\mu^*(1)$.
        \end{numcases}
        Then as in the proof of \Cref{reg_return}, by \Cref{spec_avg} and \autocite[Theorem~5.1]{Lyons-2018},
        \begin{displaymath}\begin{split}
            &\Phi(0)=0,\qquad \Phi(2)=n-1,\\
            &n\mu^*(\lambda)\leqslant \Phi(\lambda) \qquad \text{for\quad} \lambda\in[0,1],\\
            &n\mu^*(\lambda)\geqslant \Phi(\lambda) \qquad \text{for\quad} \lambda\in[1,2].
        \end{split}\end{displaymath}
        Therefore, by our calculation in part a) of the proof of \Cref{spec_bound}, for $t\equiv 0\bmod 2$,
            \begin{displaymath}\begin{split}
                \sum_{x\in V}p_t(x,x)
                &=\sum_{x\in V}\Bigl(\pi(x)+\bigl(1-\pi(x)\bigr)\Bigr)
                +t\int_{0}^{2}\Bigl(\sum_{x\in V}\mu^*_x(\lambda)\Bigr)(1-\lambda)^{t-1}\df{\lambda}\\
                &=n+t\int_{0}^{2}n\mu^*(\lambda)(1-\lambda)^{t-1}\df{\lambda}
                \leqslant
                n+t\int_{0}^{2}\Phi(\lambda)(1-\lambda)^{t-1}\df{\lambda}\\
                &=n-(n-1)+\int_{0}^{2}(1-\lambda)^t\Phi'(\lambda)\df{\lambda}
                =1+\int_{0}^{2}(1-\lambda)^t\Phi'(\lambda)\df{\lambda}\\
                &\leqslant 1+2\int_{0}^{1}(1-\lambda)^t\frac{n\sqrt[3]{4000}}{3}\lambda^{-2/3}\df{\lambda}\\
                &\leqslant 1+\frac{30n}{t^{1/3}},
            \end{split}\end{displaymath}
        where we are using \Cref{ibp} to get the third equality, and the last inequality holds by \Cref{avg_clc} in the appendix. Therefore,
            \begin{displaymath}
                \frac{\sum_{x\in V}p_t(x,x)-1}{n}
                \leqslant
                \frac{30}{t^{1/3}}.
            \end{displaymath}
        On the other hand, for all $x\in V$ and even $t$, $p_t(x,x)\geqslant \pi(x)$.
        Hence, we have
            \begin{displaymath}
               \frac{\sum_{x\in V}p_t(x,x)-1}{n}
               \geqslant 0.\qedhere
            \end{displaymath}
    \end{proof}
        \begin{theorem}\label{avg_odd}
          Let $G$ be a non-bipartite, finite, simple, connected, unweighted graph. Then for $t\equiv 1\bmod 2$,
            \begin{displaymath}
                \frac{\abs[\big]{\sum_{x\in V}p_t(x,x)-1}}{n}
                \leqslant
                \frac{15}{t^{1/3}}.
            \end{displaymath}
        \end{theorem}
        \begin{proof}
            a) Let $t\equiv1\bmod 2$. By our calculation in part a) of the proof of \Cref{spec_bound},
                \begin{displaymath}\begin{split}
                \sum_{x\in V}p_t(x,x)
                        &=
                        \sum_{x\in V}\Bigl(\pi(x)-\bigl(1-\pi(x)\bigr)\Bigr)
                        +t\int_{0}^{2}\Bigl(\sum_{x\in V}\mu^*_x(\lambda)\Bigr)(1-\lambda)^{t-1}\df{\lambda}\\
                        &=
                        2-n+t\int_{0}^{2}n\mu^*(\lambda)(1-\lambda)^{t-1}\df{\lambda}.
                 \end{split}\end{displaymath}
            \par
            b) Set
                \begin{displaymath}
                    \Psi_1(\lambda)\coloneqq \bigl(n-1-n\sqrt[3]{4000(2-\lambda)}\,\bigr)\vee0.
                \end{displaymath}
                Then by \Cref{spec_avg},
                \begin{displaymath}\begin{split}
                    &\Psi_1(0)=0,\qquad \Psi_1(2)=n-1,\\
                    &\sum_{x\in V}\mu_x^*(\lambda)\geqslant \Psi_1(\lambda) \qquad \text{for } \lambda\in[0,2].
                \end{split}\end{displaymath}
                By our calculation in part a),
                \begin{displaymath}\begin{split}
                    \sum_{x\in V}p_t(x,x)
                    &\geqslant
                    2-n+t\int_{0}^{2}\Psi_1(\lambda)(1-\lambda)^{t-1}\df{\lambda}\\
                    &=
                    2-n+(n-1)+\int_{0}^{2}(1-\lambda)^t\Psi'_1(\lambda)\df{\lambda}
                    =
                    1+\int_{0}^{2}(1-\lambda)^t\Psi'_1(\lambda)\df{\lambda}\\
                    &\geqslant
                    1-\int_{0}^{1}(1-\lambda)^t\frac{n\sqrt[3]{4000}}{3}\lambda^{-2/3}\df{\lambda}
                    \geqslant
                    1-\frac{15n}{t^{1/3}},
                \end{split}\end{displaymath}
                where the first equality holds thanks to \Cref{ibp}, and the last inequality follows from \Cref{avg_clc} in the appendix.
            \par
            c) Set
                \begin{displaymath}
                    \Psi_2(\lambda)\coloneqq \bigl(n\sqrt[3]{4000\lambda}\,\bigr)\wedge(n-1).
                \end{displaymath}
                Then by \autocite[Theorem~5.1]{Lyons-2018},
                \begin{displaymath}\begin{split}
                    &\Psi_2(0)=0,\qquad \Psi_2(2)=n-1,\\
                    &\sum_{x\in V}\mu_x^*(\lambda)\leqslant \Psi_2(\lambda) \qquad \text{for } \lambda\in[0,2].
                \end{split}\end{displaymath}
                Using a similar argument as in b), we get
                \begin{displaymath}
                    \sum_{x\in V}p_t(x,x)
                    \leqslant
                    1+\frac{15n}{t^{1/3}}.\qedhere
                \end{displaymath}
        \end{proof}
    \subsection{Sum of Eigenvalue Powers in Absolute Value}
        Since similar matrices have the same trace, we have
            \begin{displaymath}
                \frac{\sum_{x\in V}p_t(x,x)-1}{n}
                =
                \frac{\sum_{i=1}^n(\lambda^P_i)^t-1}{n}=\frac{1}{n}\sum_{i=1}^{n-1}(\lambda^P_i)^t.
            \end{displaymath}
        Therefore, when $t\geqslant 2$ is even, \Cref{avg_even} gives
            \begin{equation}\label{abs_bdd_even}
                \frac{1}{n}\sum_{i=1}^{n-1}\abs{\lambda^P_i}^t
                =\frac{\sum_{x\in V}p_t(x,x)-1}{n}
                \leqslant \frac{30}{t^{1/3}}.
            \end{equation}
        For odd $t$, we have
        \begin{displaymath}
            \abs[\Big]{\sum_{x\in V}p_t(x,x)-1}
            = \abs[\Big]{\sum^{n-1}_{i=1}(\lambda_i^P)^t}
            \leqslant
            \sum^{n-1}_{i=1}\abs{\lambda_i^P}^t.
        \end{displaymath}
        \Cref{avg_odd} will not give a bound on $\sum^{n-1}_{i=1}\abs{\lambda_i^P}^t$ directly for odd $t$. But we can still make a detour and bound $\sum^{n-1}_{i=1}\abs{\lambda_i^P}^t$ by \Cref{avg_even} as follows.
        \begin{proposition}\label{abs_bdd}
            Let $G$ be a non-bipartite, finite, simple, connected, unweighted graph. We have
            \begin{displaymath}\begin{split}
                \frac{1}{n}\sum^{n-1}_{i=1}\abs{\lambda_i^P}^t&\leqslant \frac{30}{t^{1/3}} \qquad \text{for\quad} t\equiv 0\bmod 2 \text{ and } t\geqslant 2,\\
                \frac{1}{n}\sum^{n-1}_{i=1}\abs{\lambda_i^P}^t&\leqslant \frac{30}{\sqrt[6]{t^2-1}} \qquad \text{for\quad} t\equiv 1\bmod 2\text{ and } t\geqslant 3.
            \end{split}\end{displaymath}
        \end{proposition}
        The bound in \Cref{avg_odd} is $\frac{15}{t^{1/3}}$; we get the bound $\frac{30}{(t^2-1)^{1/6}}$ for odd $t$ in \Cref{abs_bdd}. These two bounds are not far apart: they are of the same order in $t$.
        \begin{proof}[Proof of \Cref{abs_bdd}]
            The assertion for even $t$ is nothing but \Cref{abs_bdd_even}. We now consider odd $t$. When $t\geqslant 3$ is odd, by \Cref{abs_bdd_even},
                \begin{displaymath}\begin{split}
                    \sum_{i=1}^{n-1}\abs{\lambda^P_i}^t
                    &=
                    \sum_{i=1}^{n-1}\bigl(\abs{\lambda^P_i}^{\frac{t-1}{2}}\abs{\lambda^P_i}^{\frac{t+1}{2}}\bigr)^{1/2}
                    \leqslant
                    \Bigl(\sum_{i=1}^{n-1}\abs{\lambda^P_i}^{t-1}\Bigr)^{1/2}
                        \Bigl(\sum_{i=1}^{n-1}\abs{\lambda^P_i}^{t+1}\Bigr)^{1/2}\\
                    &\leqslant
                    \Bigl(\frac{30n}{\sqrt[3]{t-1}}\Bigr)^{1/2}\Bigl(\frac{30n}{\sqrt[3]{t+1}}\Bigr)^{1/2}\\
                    &=
                    \frac{30n}{\sqrt[6]{t^2-1}},
                \end{split}\end{displaymath}
            where the first inequality holds thanks to the Cauchy--Schwarz inequality.
        \end{proof}


    \section{Bipartite Case}\label{sec_bip}
        Now we deal with bipartite graphs. On bipartite graphs, simple random walk has period two. We have the following result on return probabilities, whose proof uses only the bound on the vertex spectral measure of $\cl$ from \autocite[Theorem~4.9]{Lyons-2018}.
        \begin{theorem}\label{reg_return_bi}
            Consider a regular, bipartite, simple, connected, unweighted graph $G$. For each $x\in V$, simple random walk on $G$ satisfies
                \begin{displaymath}
                    0\leqslant  p_t(x,x)-2\pi(x)\leqslant \frac{18}{\sqrt{t}} \qquad \text{for\quad} t\equiv 0\bmod 2.
                \end{displaymath}
        \end{theorem}
        To prove \Cref{reg_return_bi}, we don't need to get a bound on the vertex spectral measure of $\cq$ as we did in \Cref{sec_reg}, because the following lemma gives us a relation between the vertex spectral measures of $\cl$ and $\cq$.
        \begin{lemma}[\cite{Mohar-1989}, Theorem~4.8]\label{mohar}
            If $G$ is bipartite, the spectrum of $P$ is symmetric with respect to zero. For each $x\in V$,
            the vertex spectral measure of $P$, $\bigl\langle I_P(\ud\delta)\be_x,\be_x\bigr\rangle_w$, is symmetric with respect to\/ $0$. As a consequence, if $G$ is bipartite, for each $x\in V$ and $\delta\in[0,2]$, $\mu_x(\delta)=\mu^\cq_x(\delta)$.\qed
        \end{lemma}
        \begin{proof}[Proof of \Cref{reg_return_bi}]
            a) Set
                    \begin{numcases}
                        {\mu_x^\#(\lambda)\coloneqq}
                            \mu_x^*(\lambda)
                            \quad \nonumber &if\/ $0\leqslant \lambda<2$,\\
                            1-2\pi(x)
                            \quad \nonumber &if\/ $\lambda=2$,
                    \end{numcases}
                and
                    \begin{numcases}
                        {\varphi^\#(\lambda)\coloneqq }
                            10\sqrt{\lambda}
                            \quad \nonumber &if\/ $\lambda\geqslant0\text{ and }10\sqrt{\lambda}\leqslant \mu_x^\#(1)$,\\
                            \mu_x^\#(1) \quad \nonumber &\text{for intermediate values of $\lambda$},\\
                            \mu_x^\#(2)-10\sqrt{2-\lambda}
                            \quad \nonumber &if\/ $\lambda\leqslant2\text{ and }\mu_x^\#(2)-10\sqrt{2-\lambda}\geqslant \mu_x^\#(1)$.
                    \end{numcases}
                Then $\varphi^\#(0)=\mu_x^\#(0)=0$, $\varphi^\#(2)=\mu_x^\#(2)=1-2\pi(x)$. In addition, by \Cref{mohar} and \autocite[Theorem~4.9]{Lyons-2018}, we have
                    \begin{displaymath}\begin{split}
                        &\mu_x^\#(\lambda)\leqslant \varphi^\#(\lambda) \qquad \text{for\quad} \lambda\in[0,1],\\
                        &\mu_x^\#(\lambda)\geqslant \varphi^\#(\lambda) \qquad \text{for\quad} \lambda\in[1,2].
                    \end{split}\end{displaymath}
            \par
            b) Now let $t$ be a positive even number. We have
                    \begin{equation}\begin{split}
                        p_t(x,x)
                        &=\int_{[0,2]}(1-\lambda)^t\,\mu_x(\ud\lambda)\\
                        &=2\pi(x)+\int_{[0,2]}(1-\lambda)^t\,\mu^\#_x(\ud\lambda).
                    \end{split}\end{equation}
                Hence, we see that
                    \begin{displaymath}
                        p_t(x,x)-2\pi(x)=\int_{[0,2]}(1-\lambda)^t\,\mu^\#_x(\ud\lambda)\geqslant 0.
                    \end{displaymath}
            \par
            c) On the other hand,
                    \begin{displaymath}\begin{split}
                        p_t(x,x)
                        &=
                        2\pi(x)+\int_{[0,2]}(1-\lambda)^t\,\mu^\#_x(\ud\lambda)\\
                        &=
                        2\pi(x)+\left.(1-\lambda)^t\mu^\#_x(\lambda)\right|^2_0-\int_{0}^{2}\mu^\#_x(\lambda)\df{(1-\lambda)^t}\\
                        &=
                        2\pi(x)+\bigl(1-2\pi(x)\bigr)+t\int_{0}^{2}\mu^\#_x(\lambda)(1-\lambda)^{t-1}\df{\lambda}\\
                        &\leqslant
                        1+t\int_{0}^{2}\varphi^\#(\lambda)(1-\lambda)^{t-1}\df{\lambda}
                        =
                        1-\varphi^\#(2)+\int_{0}^{2}(1-\lambda)^t(\varphi^\#)'(\lambda)\df{\lambda}\\
                        &=
                        2\pi(x)+\int_{0}^{2}(1-\lambda)^t(\varphi^\#)'(\lambda)\df{\lambda}
                        =
                        2\pi(x)+2\int_{0}^{1}(1-\lambda)^t(\varphi^\#)'(\lambda)\df{\lambda}\\
                        &\leqslant
                        2\pi(x)+2\int_{0}^{1}\frac{5(1-\lambda)^t}{\sqrt{\lambda}}\df{\lambda}
                        =
                        2\pi(x)+10\int_{0}^{1}\frac{(1-\lambda)^t}{\sqrt{\lambda}}\df{\lambda}\\
                        &\leqslant
                        2\pi(x)+\frac{18}{\sqrt{t}},
                    \end{split}\end{displaymath}
                    where the fourth equality follows from \Cref{ibp}, and we use \Cref{calc_aux} to get the last inequality.
            \par
            d) Summing up the above discussion, we get, for $x\in V$ and $t\equiv 0\bmod 2$,
                \begin{displaymath}
                    0\leqslant  p_t(x,x)-2\pi(x)\leqslant \frac{18}{\sqrt{t}}.\qedhere
                \end{displaymath}
        \end{proof}
        \Cref{reg_return_bi} is an example of treating simple random walk on bipartite graphs. Using the same method, one may get parallel results to what we had
        in previous sections. For instance, the following \Cref{trans_return_bi} is parallel to \Cref{trans_return}.
        \begin{corollary}\label{trans_return_bi}
            Let $G$ be a finite or infinite, vertex-transitive, $d$-regular, bipartite, simple, connected, unweighted  graph with at least polynomial growth rate $N^\#(r)\geqslant  Cr^D$, where $C > 0$ and $D \geqslant  1$ are constants and\/ $0 \leqslant  r \leqslant  \D(G)$.
            Then for each $x\in V$ and $t\equiv 0\bmod 2$,
            \begin{displaymath}
                0\leqslant  p_t(x,x)-2\pi(x)
                \leqslant 2\widetilde{C}t^{-D/2},
            \end{displaymath}
        where $\widetilde{C}=\frac{(D+4)^{D/2+2}d^{D/2}}{32CD^{D/2-1}}\Gamma\bigl(\frac{D}{2}\bigr)$.\qed
        \end{corollary}
        Note that the rate $t^{-D/2}$ here is the correct decay rate for the simple random walk on $\ZN^D$, $D\in \NN$. To prove \Cref{trans_return_bi}, one need only follow our argument in the proof of \Cref{reg_return_bi} and use the bound on the vertex spectral measure of $\cl$ obtained in \autocite[Theorem~6.1]{Lyons-2018}. Details are omitted.


\section{Combinatorial Signless Laplacian}\label{sec_comb}
    We used spectral embedding to deal with random walk on graphs in previous sections. In fact, this tool is also powerful in analyzing the spectrum of graph adjacency matrices.

    Assume that $G$ is a weighted finite graph. Let $\ell^2(V)$ be the Hilbert space of functions $f \colon V \to \RNS \text{ or } \CN$ with inner product
        \begin{displaymath}
            \langle f,g \rangle \coloneqq  \sum_{x \in V}f(x)\overline{g(x)}
        \end{displaymath}
    and squared norm $\norm{f}^2 \coloneqq  \langle f,f\rangle$. Let $W$ be the diagonal weight matrix of the graph $G$: $W\coloneqq \mathsf{diag}(w(x)\st x\in V)$.
    Then it is easy to see that the combinatorial signless Laplacian $\Theta\coloneqq  W+A$ is a bounded self-adjoint operator on $\ell^2(V)$. We denote the resolution of identity for $\Theta$ as $I_\Theta$ and define the vertex spectral measure of $\Theta$ at $x\in V$ as
            \begin{displaymath}
                \mu_x^\Theta(\delta)\coloneqq \langle I_\Theta\bigl([0,\delta]\bigr)\ind_x,\ind_x\rangle,\qquad \delta\geqslant 0.
            \end{displaymath}
    For ease of notation, we also write $I_\Theta(\delta)\coloneqq I_\Theta\bigl([0,\delta]\bigr)$ for $\delta\geqslant 0$.
    \begin{lemma}\label{Theta_form}
        For $f\in \ell^2(V)$, we have
            \begin{displaymath}
                \langle\Theta f,f\rangle=\sum_{(v,u)\in E}w(v,u)\abs[\big]{f(v)+f(u)}^2.
            \end{displaymath}
        Therefore, the spectrum of $\Theta$ is non-negative. Moreover, if $f\in\img\bigl(I_\Theta(\delta)\bigr)$ for some $\delta\geqslant 0$, then $\langle\Theta f,f\rangle\leqslant \delta \norm{f}^2$.\qed
    \end{lemma}
    See the appendix for a proof.

        For $\delta\geqslant 0$, we define the spectral embedding based on $\Theta$ as
            \begin{displaymath}\begin{split}
                F^\Theta\colon V&\rightarrow \ell^2(V)\\
                x&\mapsto F^\Theta_x\coloneqq I_\Theta(\delta)\ind_x.
            \end{split}\end{displaymath}
        It is clear that $F^\Theta_x$ is a real-valued function on $V$ for each $x\in V$.
        \begin{lemma}\label{embed_norm_com}
            For each finite graph $G$ and $x\in V$,
                \begin{displaymath}
                    \norm[\big]{F^\Theta_x}^2=F^\Theta_x(x)=\mu_x^\Theta(\delta).
                \end{displaymath}
        \end{lemma}
        \begin{proof}
            Since $I_\Theta(\delta)$ is a self-adjoint projection operator on $\ell^2(V)$, we see that
            \begin{displaymath}
            \norm[\big]{F^\Theta_x}^2
            =
            \langle F^\Theta_x,F^\Theta_x\rangle
            =
            \langle I_\Theta(\delta)\ind_x,I_\Theta(\delta)\ind_x\rangle
            =
            \langle I_\Theta(\delta)\ind_x,\ind_x\rangle.
            \end{displaymath}
            By the definition of the vertex spectral measure of $\Theta$, we see that
            $
            \langle I_\Theta(\delta)\ind_x,\ind_x\rangle
            =
            \mu_x^\Theta(\delta)$.
            Moreover, since $I_\Theta(\delta)\ind_x=F^\Theta_x$, we have $\langle I_\Theta(\delta)\ind_x,\ind_x\rangle=F^\Theta_x(x)$.
        \end{proof}
        \begin{lemma}\label{def_f_com}
            If $\mu_x^\Theta(\delta)>0$, define $f\colon V\rightarrow\CN$ as $f\coloneqq \frac{F^\Theta_x}{\norm{F^\Theta_x}}$. We have
            \begin{enumerate}[\upshape(1)]
              \item $\norm{f}=1$;
              \item $f(x)=\sqrt{\mu_x^\Theta(\delta)}$;
              \item $f\in\img\bigl(I_\Theta(\delta)\bigr)$.
            \end{enumerate}
        \end{lemma}
        \begin{proof}
            The first and third assertions are obvious. As for the second assertion, by \Cref{embed_norm_com},
            \begin{displaymath}
                f(x)
                =
                \frac{F^\Theta_x(x)}{\norm{F^\Theta_x}}
                =
                \frac{\mu_x^\Theta(\delta)}{\sqrt{\mu_x^\Theta(\delta)}}
                =
                \sqrt{\mu_x^\Theta(\delta)}.\qedhere
            \end{displaymath}
        \end{proof}

        We are now in position to present bounds on vertex spectral measures.
        Denote the eigenvalues of the combinatorial signless Laplacian $\Theta$ on $G$ as
            \begin{displaymath}
                0 \leqslant \lambda^\Theta_{\mathrm{min}} = \lambda^\Theta_1
                \leqslant \lambda^\Theta_2 \leqslant \lambda^\Theta_3 \leqslant \dotsb \leqslant \lambda^\Theta_{n-1}
                < \lambda^\Theta_n = \lambda^\Theta_{\mathrm{max}}.
            \end{displaymath}
        Recall that $\kg$ is defined in \Cref{def_kudos}; \Cref{est} shows that $\kg\geqslant \frac{1}{\D(G)+1}$ when $G$ is non-bipartite.
    \begin{proposition}\label{comb_spec}
      Let $G$ be a non-bipartite, finite, simple, connected, weighted graph with weight at least\/ $1$ for each edge. Then for each $\delta\in [0,\lambda^\Theta_{\mathrm{max}})$ and $x\in V$, we have
            \begin{displaymath}
                \mu_x^\Theta(\delta)
                \leqslant
                \frac{\delta}{\kg}
                \leqslant
                \bigl(\D(G)+1\bigr)\delta.
            \end{displaymath}
    \end{proposition}

        \begin{proof}
            This proof is a mimic of the proof of \Cref{mix_spec_prep}. Fixing a vertex $x\in V$, we define $f$ as in \Cref{def_f_com}.
            \par
            When $0\leqslant \delta<\lambda_{\mathrm{min}}^\Theta$, $\mu_x^\Theta(\delta)=0$ by definition. So the inequality holds automatically.
            \par
            When $\lambda^\Theta_{\mathrm{min}}\leqslant \delta<\lambda^\Theta_{\mathrm{max}}$, we know that $f$ is orthogonal to the eigenspace of $\Theta$ corresponding to $\lambda^\Theta_{\mathrm{max}}$, which is spanned by a positive vector according to the Perron--Frobenius theorem. By \Cref{def_kudos}, we have
            \begin{displaymath}
                \delta \geqslant \langle f,\Theta f\rangle
                =
                \sum_{(v,u)\in E}w(v,u)\abs{f(v)+f(u)}^2
                \geqslant f(x)^2\kg
                =\mu_x^\Theta(\delta)\kg.\qedhere
            \end{displaymath}
        \end{proof}
    In order to get a lower bound on eigenvalues of $\Theta$ from \Cref{comb_spec}, we need the following \Cref{spec_loc_com}.
        \begin{lemma}\label{spec_loc_com}
            Let $G$ be a finite, connected, weighted graph. We have
            \begin{displaymath}
                \sum_{x\in V}\mu_x^\Theta(\delta)=\abs[\big]{\{j\st \lambda_j^\Theta\leqslant \delta\}}.
            \end{displaymath}
        \end{lemma}
        \begin{proof}
            Note that
            \begin{displaymath}
                \sum_{x\in V}\mu_x^\Theta(\delta)
                =
                \sum_{x\in V}\bigl\langle I_\Theta(\delta)\ind_x,\ind_x\bigr\rangle.
            \end{displaymath}
            This is the trace of $I_\Theta(\delta)$, which equals the dimension of its image. Therefore, we have
            \begin{displaymath}
                \sum_{x\in V}\mu_x^\Theta(\delta)=\abs[\big]{\{j\st \lambda_j^\Theta\leqslant \delta\}}.\qedhere
            \end{displaymath}
        \end{proof}

    \begin{corollary}\label{com_s_lap}
        Let $G$ be a non-bipartite, finite, simple, connected, weighted graph with weight at least\/ $1$ for each edge. For\/ $1\leqslant k\leqslant n$, we have
        \begin{displaymath}
            \lambda^\Theta_k
            \geqslant \frac{k\kg}{n}
            \geqslant\frac{k}{\bigl(\D(G)+1\bigr)n}.
        \end{displaymath}
    \end{corollary}
    \begin{proof}
        By \Cref{comb_spec},
            \begin{displaymath}
                \sum_{x\in V}\mu_x^\Theta(\delta)
                \leqslant
                \sum_{x\in V}\frac{\delta}{\kg}=\frac{n\delta}{\kg}.
            \end{displaymath}
        Hence, \Cref{spec_loc_com} gives
        \begin{displaymath}
            \abs[\big]{\{j\st \lambda_j^\Theta\leqslant \delta\}}
            \leqslant
            \frac{n\delta}{\kg}.
        \end{displaymath}
        Therefore, we must have
        $\lambda^\Theta_k\geqslant\frac{k\kg}{n}$.
    \end{proof}
        For the adjacency matrix $A$ of $G$, we denote its eigenvalues as
            \begin{displaymath}
                -w_{\mathrm{max}} \leqslant \lambda^A_{\mathrm{min}} = \lambda^A_1
                    \leqslant  \lambda^A_2 \leqslant \lambda^A_3 \leqslant \dotsb \leqslant \lambda^A_{n-1} < \lambda^A_n= \lambda^A_{\mathrm{max}} \leqslant  w_{\mathrm{max}},
            \end{displaymath}
            where $w_{\mathrm{max}}\coloneqq \max_{x\in V}w(x)$. The following \Cref{Alon_ext} improves \autocite[Theorem~1.1]{Alon-2000}, which obtained that $\dmx+\lambda_1^A\geqslant\frac{1}{(\D(G)+1)n}$ for unweighted graphs.
    \begin{corollary}\label{Alon_ext}
        Let $G$ be a non-bipartite, finite, simple, connected, weighted graph with weight at least\/ $1$ for each edge. For\/ $1\leqslant k\leqslant n$, we have
            \begin{displaymath}
                w_{\mathrm{max}}+\lambda_k^A
                \geqslant
                \frac{k\kg}{n}
                \geqslant
                \frac{k}{\bigl(\D(G)+1\bigr)n}.
            \end{displaymath}
    \end{corollary}
    \begin{proof}
        Let $X_1$ be the linear subspace of $\ell^2(V)$ spanned by the eigenvectors of $A$ corresponding to $\lambda^A_1$, $\lambda^A_2$, $\ldots$$, \lambda^A_k$, and $X_2$ be the linear subspace of $\ell^2(V)$ spanned by the eigenvectors of $\Theta$ corresponding to $\lambda^\Theta_k$, $\lambda^\Theta_{k+1}$, $\ldots$, $\lambda^\Theta_n$. Then
        \begin{displaymath}
            \dim X_1\geqslant k,\quad \dim X_2\geqslant n-k+1.
        \end{displaymath}
        It follows that
        \begin{displaymath}
            \dim{X_1}+\dim{X_2}>n=\dim{\ell^2(V)}.
        \end{displaymath}
        Therefore, the intersection of $X_1$ and $X_2$ contains a non-zero vector $h$ of unit norm. Hence,
        \begin{displaymath}
            \frac{k\kg}{n}\leqslant \lambda^\Theta_k
                    \leqslant (\Theta h, h)=(Wh,h)+(Ah,h)\leqslant w_{\mathrm{max}}+\lambda_k^A.\qedhere
        \end{displaymath}
    \end{proof}
    \begin{rmk}
        The combinatorial signless Laplacian is also related to line graphs. Let $\rho_1\geqslant \rho_2\geqslant \cdots\geqslant \rho_r$  be the positive combinatorial signless Laplacian eigenvalues of $G$. Then by \autocite[Proposition~1.4.1]{Brouwer-2011}, the eigenvalues of the line graph of $G$ are
        \begin{displaymath}\begin{split}
            \theta_i&=\rho_i-2,\qquad i=1,2,\ldots,r,\\
            \theta_i&=-2,\qquad i=r+1,r+2,\ldots,\abs{E(G)}.
        \end{split}\end{displaymath}
        Using this relation and \Cref{com_s_lap}, one may do some quick analysis on the spectrum of the line graph of $G$.
    \end{rmk}

\section*{Acknowledgments}
\addcontentsline{toc}{section}{{Acknowledgments}}
        The author is greatly indebted to Prof.\@ Russell Lyons for many useful discussions and for his careful guidance. Prof.\@ Shayan Oveis Gharan deserves sincere appreciation from the author for helpful comments on an earlier version of this manuscript. The author thanks Prof.\@ Lu-Jing Huang, Dr.\@ Rongjuan Fang, and Mr.\@ Tao Wang for their helpful suggestions. This work was supported by the National Science Foundation grant DMS-1954086.
\setcounter{equation}{0}
\renewcommand\thesection{A}
\renewcommand\theequation{A.\arabic{equation}}
\renewcommand\thetheorem{A.\arabic{theorem}}
    \section{Appendix}\label{sec_Oli}
    \subsection{Miscellaneous Lemmas and Proofs}
        In the proof of \Cref{reg_return}, the following elementary calculation is needed.
        \begin{lemma}\label{calc_aux}
            For $t>0$, we have
                \begin{displaymath}
                    \int_{0}^{1}(1-\lambda)^t\frac{1}{\sqrt{\lambda}}\df{\lambda}\leqslant \frac{9}{5\sqrt{t}}.
                \end{displaymath}
        \end{lemma}
        \begin{proof}
            We have
            \begin{displaymath}
                \int_{0}^{1}(1-\lambda)^t\frac{1}{\sqrt{\lambda}}\df{\lambda}
                \leqslant \int_{0}^{1}\exp\{-\lambda t\}\lambda^{-1/2}\df{\lambda}
                \leqslant \int_{0}^{\infty}\exp\{-\lambda t\}\lambda^{-1/2}\df{\lambda}.
            \end{displaymath}
            Taking a change of variable $\lambda t=s$, we get that
                    \begin{displaymath}\begin{split}
                        \int_{0}^{\infty}\ue^{-s}(s/t)^{-1/2}\df{(s/t)}
                        &\leqslant \frac{1}{\sqrt{t}}\int_{0}^{\infty}\ue^{-s}s^{-1/2}\df{s}\\
                        &=\frac{\Gamma(1/2)}{\sqrt{t}}\leqslant \frac{9}{5\sqrt{t}}.
                    \end{split}\end{displaymath}
            Hence, the inequality follows immediately.
        \end{proof}

        \Cref{avg_clc} is useful in the proofs of \Cref{avg_even,avg_odd}.
        \begin{lemma}\label{avg_clc}
            For $t>0$, we have
                \begin{displaymath}
                    \frac{\sqrt[3]{4000}}{3}\int_{0}^{1}(1-\lambda)^t\lambda^{-2/3}\df{\lambda}
                    \leqslant \frac{15}{t^{1/3}}.
                \end{displaymath}
        \end{lemma}
        \begin{proof}
            We have
            \begin{displaymath}
                \int_{0}^{1}(1-\lambda)^t\lambda^{-2/3}\df{\lambda}
                \leqslant \int_{0}^{\infty}\exp\{-\lambda t\}\lambda^{-2/3}\df{\lambda}.
            \end{displaymath}
            Taking $\lambda t=s$, we get
            \begin{displaymath}\begin{split}
                \int_{0}^{\infty}\exp\{-\lambda t\}\lambda^{-2/3}\df{\lambda}
                &=\int_{0}^{\infty}\ue^{-s}(s/t)^{-2/3}\df{(s/t)}\\
                &=\frac{1}{t^{1/3}}\int_{0}^{\infty}\ue^{-s}s^{-2/3}\df{s}=\frac{\Gamma(1/3)}{t^{1/3}}.
            \end{split}\end{displaymath}
            Therefore,
                \begin{displaymath}
                    \frac{\sqrt[3]{4000}}{3}\int_{0}^{1}(1-\lambda)^t\lambda^{-2/3}\df{\lambda}
                    \leqslant \frac{\sqrt[3]{4000}\,\Gamma(1/3)}{3 t^{1/3}}
                    \leqslant \frac{15}{t^{1/3}}.\qedhere
                \end{displaymath}
        \end{proof}
        \begin{proof}[Proof of \Cref{q_form}]
            We need prove the first assertion only. First, we claim that
            \begin{displaymath}
                \sum_{v,u\in V}w(v,u)\abs{f(v)}\abs{f(v)+f(u)}<\infty.
            \end{displaymath}
            In fact, we have
            \begin{displaymath}\begin{split}
                \sum_{v,u\in V}w(v,u)\abs{f(v)}\abs{f(v)+f(u)}
                &\leqslant \sum_{v,u\in V}w(v,u)\bigl(\abs{f(v)}^2+\abs{f(v)}\abs{f(u)}\bigr)\\
                &=\sum_{v,u\in V}w(v,u)\abs{f(v)}^2 + \sum_{v,u\in V}w(v,u)\abs{f(v)}\abs{f(u)}.
            \end{split}\end{displaymath}
            By the Cauchy--Schwarz inequality, we may proceed and get
                \begin{displaymath}\begin{split}
                \sum_{v,u\in V}w(v,u)\abs{f(v)}\abs{f(v)+f(u)}&\leqslant
                \sum_{v,u\in V}w(v,u)\abs{f(v)}^2\\
                 &{\qquad}+ \Bigl(\sum_{v,u\in V}w(v,u)\abs{f(v)}^2\Bigr)^{1/2}\Bigl(\sum_{v,u\in V}w(v,u)\abs{f(u)}^2\Bigr)^{1/2}\\
                &= \sum_{v,u\in V}w(v,u)\abs{f(v)}^2 + \sum_{v,u\in V}w(v,u)\abs{f(v)}^2\\
                &= 2\sum_{v,u\in V}w(v,u)\abs{f(v)}^2
                = 2\sum_{v\in V}\abs{f(v)}^2\sum_{u\in V}w(v,u)\\
                &= 2\sum_{v\in V}\abs{f(v)}^2w(v)=2\norm{f}_w<\infty.
            \end{split}\end{displaymath}
            The claim is proved. By this claim, using Fubini's theorem, we have
            \begin{displaymath}\begin{split}
                \sum_{v,u\in V}w(v,u)f(v)\overline{\bigl(f(v)+f(u)\bigr)}
                &=\sum_{v\in V}f(v)\overline{\sum_{u\in V}w(v,u)\bigl(f(v)+f(u)\bigr)}\\
                &=\sum_{v\in V}w(v)f(v)\overline{\bigl(f(v)+\sum_{u\in V}p(v,u)f(u)\bigr)}\\
                &=\sum_{v\in V}w(v)f(v)\overline{\bigr((I+P)f\bigl)(v)}=\langle f,(I+P)f\rangle_w \\
                &=\langle f,\cq f\rangle_w.
            \end{split}\end{displaymath}
            By interchanging $u$ and $v$, we have
            \begin{displaymath}
                \sum_{v,u\in V}w(v,u)f(u)\overline{\bigl(f(v)+f(u)\bigr)}
                =\langle f,\cq f\rangle_w.
            \end{displaymath}
            Therefore,
            \begin{displaymath}\begin{split}
                \sum_{(v,u)\in E}&w(v,u)\abs[\big]{f(v)+f(u)}^2
                =
                \frac{1}{2}\sum_{v,u\in V}w(v,u)\abs[\big]{f(v)+f(u)}^2\\
                &=\frac{1}{2}\Bigl(
                \sum_{v,u\in V}w(v,u)f(v)\overline{\bigl(f(v)+f(u)\bigr)}
                +\sum_{v,u\in V}w(v,u)f(u)\overline{\bigl(f(v)+f(u)\bigr)}
                \Bigr)\\
                &= \tfrac{1}{2}\bigl(\langle f,\cq f\rangle_w+\langle f,\cq f\rangle_w\bigr)\\
                &=\langle f,\cq f\rangle_w.\qedhere
            \end{split}\end{displaymath}
        \end{proof}
        \begin{proof}[Proof of \Cref{Theta_form}]
        We need to prove the first assertion only. Notice that
            \begin{displaymath}\begin{split}
                \sum_{v,u\in V}w(v,u)f(v)\overline{\bigl(f(v)+f(u)\bigr)}
                &=
                \sum_{v\in V}f(v)\overline{\sum_{u\in V}w(v,u)f(v)}
                +\sum_{v\in V}f(v)\overline{\sum_{u\in V}w(v,u)f(u)}\\
                &=
                \sum_{v\in V}f(v)\overline{w(v)f(v)}
                +\sum_{v\in V}f(v)\overline{(Af)(v)}\\
                &=
                \langle f,Wf\rangle +\langle f,Af\rangle\\
                &=
                \langle f,\Theta f\rangle.
            \end{split}\end{displaymath}
            By interchanging $u$ and $v$, we have
            \begin{displaymath}
                \sum_{v,u\in V}w(v,u)f(u)\overline{\bigl(f(v)+f(u)\bigr)}
                =
                \langle f,\Theta f\rangle.
            \end{displaymath}
            Therefore,
            \begin{displaymath}\begin{split}
                \sum_{(v,u)\in E}&w(v,u)\abs[\big]{f(v)+f(u)}^2\\
                &=\frac{1}{2}\Bigl(
                \sum_{v,u\in V}w(v,u)f(v)\overline{\bigl(f(v)+f(u)\bigr)}
                +\sum_{v,u\in V}w(v,u)f(u)\overline{\bigl(f(v)+f(u)\bigr)}
                \Bigr)\\
                &=\tfrac{1}{2}\bigl(\langle f,\Theta f\rangle+\langle f,\Theta f\rangle\bigr)\\
                &=\langle f,\Theta f\rangle.\qedhere
            \end{split}\end{displaymath}
        \end{proof}
    \subsection{Return Probability Bound Involving Relaxation Time}
        In this part, we consider finite graphs only. We use essentially a similar method to \autocite{Oliveira-2019}; yet the negative spectrum of $P$ is also considered. 
        Let $G$ be a non-bipartite, finite, simple, connected, unweighted graph.
        Then $\lambda^P_{\mathrm{min}} = \lambda^P_1>-1$. Set $\Lambda\coloneqq \abs{\lambda^P_1}\vee\abs{\lambda^P_{n-1}}$, $\trel\coloneqq (1-\Lambda)^{-1}$, $t'\coloneqq 2\lceil\trel/2\rceil-2$.
        \begin{theorem}\label{reg_return_imp}
             Let $G$ be a non-bipartite, finite, simple, connected, unweighted graph. For $t\geqslant 0$, simple random walk on $G$ satisfies
                \begin{displaymath}
                    \abs[\big]{p_{t}(x,x)-\pi(x)}< \frac{20d(x)\sqrt{\trel+1}}{(t+1)\dmi}.
                \end{displaymath}
        \end{theorem}
        \Cref{reg_return_imp} is analogous to \autocite[Theorem 1.2]{Oliveira-2019}.
        To prove \Cref{reg_return_imp}, we need some preparation.  Recall that the hitting time of $A\subseteq V$ is $\tau_A\coloneqq \inf\{t\geqslant 0\st X_t\in A\}$ and Green's function is
            \begin{displaymath}
                g_t(x,y)\coloneqq \sum_{s=0}^{t}p_t(x,y),\qquad t\geqslant 0,\ x,y\in V.
            \end{displaymath}
        \begin{lemma}\label{lm31}
            For $t\equiv 0\bmod 2$ and $x\in V$, we have
            \begin{displaymath}
                0\leqslant p_{t}(x,x)-\pi(x)\leqslant \frac{1}{1-\ue^{-1}}\cdot\frac{g_{t'}(x,x)}{\tfrac{t}{2}+1}.
            \end{displaymath}
        \end{lemma}
        \begin{proof}
            By \Cref{rep}, for $t\equiv 0 \bmod 2$, $p_{t}(x,x)-\pi(x)$ is nonnegative and decreasing in $t$. Therefore, we have
                \begin{equation}\label{green1}
                    0\leqslant  p_{t}(x,x)-\pi(x)\leqslant \frac{1}{\tfrac{t}{2}+1}\Bigl(\sum_{s=0}^{t/2}p_{2s}(x,x)-(\tfrac{t}{2}+1)\pi(x)\Bigr).
                \end{equation}
            In addition,
            \begin{equation}\label{green2}\begin{split}
                \sum_{s=0}^{t/2}p_{2s}(x,x)-(\tfrac{t}{2}+1)\pi(x)&=\sum_{s=0}^{t/2}\int_{(-1,1)}\lambda^{2s}\,\norm{I_P(\ud\lambda)\be_x}_w^2
                =\int_{(-1,1)}\Bigl(\sum_{s=0}^{t/2}\lambda^{2s}\Bigr)\,\norm{I_P(\ud\lambda)\be_x}_w^2\\
                &=\int_{(-1,1)}\frac{1-\lambda^{t+2}}{1-\lambda^2}\,\norm{I_P(\ud\lambda)\be_x}_w^2\\
                &\leqslant \int_{(-1,1)}\frac{1}{1-\lambda^2}\,\norm{I_P(\ud\lambda)\be_x}_w^2,
            \end{split}\end{equation}
            where $I_P$ is the resolution of identity for $P$.

            On the other hand, since for $1\leqslant i \leqslant n$,
                \begin{displaymath}
                    (\lambda_i^P)^{t'+2}\leqslant \Lambda^{t'+2}\leqslant \Lambda^{1/(1-\Lambda)}\leqslant \ue^{-1},
                \end{displaymath}
            we have
                \begin{equation}\begin{split}\label{green3}
                    \int_{(-1,1)}\frac{1-\ue^{-1}}{1-\lambda^2}\,\norm{I_P(\ud\lambda)\be_x}_w^2
                    &\leqslant \int_{(-1,1)}\frac{1-\lambda^{t'+2}}{1-\lambda^2}\,\norm{I_P(\ud\lambda)\be_x}_w^2\\
                    &=
                    \sum_{s=0}^{t'/2}p_{2s}(x,x)-(\tfrac{t'}{2}+1)\pi(x)
                    \leqslant \sum_{s=0}^{t'/2}p_{2s}(x,x)\\
                    &\leqslant  g_{t'}(x,x).
                \end{split}\end{equation}
            By \Cref{green1,green2,green3},
            \begin{displaymath}\begin{split}
                0&\leqslant p_{t}(x,x)-\pi(x)
                \leqslant
                \frac{1}{\tfrac{t}{2}+1}\Bigl(\sum_{s=0}^{t/2}p_{2s}(x,x)-(\tfrac{t}{2}+1)\pi(x)\Bigr)\\
                &\leqslant
                \frac{1}{\tfrac{t}{2}+1}\int_{(-1,1)}\frac{1}{1-\lambda^2}\,\norm{I_P(\ud\lambda)\be_x}_w^2\\
                &=
                \frac{1}{(\tfrac{t}{2}+1)(1-\ue^{-1})}\int_{(-1,1)}\frac{1-\ue^{-1}}{1-\lambda^2}\,\norm{I_P(\ud\lambda)\be_x}_w^2\\
                &\leqslant
                \frac{1}{(\tfrac{t}{2}+1)(1-\ue^{-1})}g_{t'}(x,x).
            \end{split}\end{displaymath}
            \Cref{lm31} is proved.
        \end{proof}

        \begin{proposition}\label{trel_bound}
            We have
            \begin{displaymath}
                \trel
                <\frac{24n^2\dav}{\dmi}.
            \end{displaymath}
        \end{proposition}
        \begin{proof}
            a) Suppose $G=(V,E)$. We construct an auxiliary graph $\widetilde{G}=(\widetilde{V},\widetilde{E})$ as follows:
            \begin{enumerate}[\hspace{3em}(1)]
                \item Let $V'=\{x'\st x\in V\}$ be a copy of $V$. The vertex set of $\widetilde{G}$ is  $\widetilde{V}\coloneqq  V\cup V'$;
                \item If $(x,y)\in E$, we introduce two edges $(x,y')$ and $(x',y)$ in  $\widetilde{E}$.
            \end{enumerate}
            Obviously,  $\widetilde{G}$ is a bipartite graph. So the spectrum of the transition matrix $\widetilde{P}$ on $\widetilde{G}$ is symmetric about $0$. Denote the eigenvalues of $\widetilde{P}$ as
            \begin{displaymath}
                -1 =\widetilde{\lambda}_1
                    <  \widetilde{\lambda}_2 \leqslant \widetilde{\lambda}_3 \leqslant \dotsb \leqslant \widetilde{\lambda}_{2n-1}
                    <\widetilde{\lambda}_{2n} = 1.
            \end{displaymath}
            Set $\widetilde{t}_{\mathrm{rel}}\coloneqq \frac{1}{1-\widetilde{\lambda}_{2n-1}}$. Because the eigenvalues of $P$ are also eigenvalues of $\widetilde{P}$, we have
            \begin{displaymath}
                {t}_{\mathrm{rel}}\leqslant \widetilde{t}_{\mathrm{rel}}.
            \end{displaymath}
            \par
            b) It is easy to show that $\D(\widetilde{G})\leqslant 4\D(G)+1$. Using a similar argument as in \autocite[Proposition 3.1]{Oliveira-2019}, one may get that
            \begin{displaymath}\begin{split}
                \widetilde{t}_{\mathrm{rel}}
                &\leqslant \D(\widetilde{G})\vol(\widetilde{V})
                \leqslant 2\bigl(4\D(G)+1\bigr)\vol(V)\\
                &\leqslant 2\biggl(4\Bigl(\frac{3n}{\dmi}-1\Bigr)+1\biggr)\vol(V)
                <\frac{24n}{\dmi}\vol(V)
                =\frac{24n^2\dav}{\dmi}.\qedhere
            \end{split}\end{displaymath}
        \end{proof}
        For non-empty $A\subsetneq V$ and $x\in V\setminus A$, define $ G(x,x;A)\coloneqq \Esubbig{x}{\sum_{s=0}^{\tau_A-1}\ind_{\{X_s\notin A\}}}$.
        \begin{proposition}\label{prp32}
            We have
            \begin{displaymath}
                \frac{G(x,x;A)}{\pi(x)}\leqslant \frac{9}{2}\Bigl(\frac{\dav n}{\dmi}\Bigr)^2\bigl(1-\pi(A)\bigr).
            \end{displaymath}
        \end{proposition}
        \begin{proof}
            By \autocite[Eq.~(2.5)]{Lyons-2017} and network reduction, we have
                \begin{displaymath}
                    \frac{G(x,x;A)}{\pi(x)}
                    \leqslant \vol(V){\mathcal R}_{\text{eff}}(x\leftrightarrow A)
                    =\dav n {\mathcal R}_{\text{eff}}(x\leftrightarrow A),
                \end{displaymath}
            where ${\mathcal R}_{\text{eff}}(x\leftrightarrow A)$ is the effective resistance between $x$ and $A$. Then we need only follow the proof of \autocite[Proposition~3.2]{Oliveira-2019} to get
            \begin{displaymath}
            {\mathcal R}_{\text{eff}}(x\leftrightarrow A)
            \leqslant
            \frac{9\dav n}{2\dmi^2}\bigl(1-\pi(A)\bigr).\qedhere
            \end{displaymath}
        \end{proof}
        Now fix $x\in V$. For $\alpha>1$, let
            \begin{displaymath}
                A_\alpha\coloneqq \bigl\{y\in V\st g_{t'}(y,x)\leqslant \alpha\pi(x)(t'+1)\bigr\}.
            \end{displaymath}
        We claim that $A_\alpha\neq \varnothing$ for $\alpha>1$. In fact,
            \begin{equation}\label{alpha}\begin{split}
                1-\pi(A_\alpha)&=\sum_{y\notin A_\alpha}\pi(y)<\sum_{y\notin A_\alpha}\pi(y)\frac{g_{t'}(y,x)}{\alpha\pi(x)(t'+1)}\\
                &=\sum_{y\notin A_\alpha}\frac{\pi(x)g_{t'}(x,y)}{\alpha\pi(x)(t'+1)}
                \leqslant \frac{1}{\alpha}\sum_{y\in V}\frac{g_{t'}(x,y)}{t'+1}\\
                &=\frac{1}{\alpha}<1.
            \end{split}\end{equation}
        So for $\alpha>1$, $A_\alpha$ is non-empty.
        \begin{lemma}\label{lm32}
            For $x\in V$,
                \begin{displaymath}
                    \frac{g_{t'}(x,x)}{\pi(x)}
                    \leqslant \frac{6\dav n}{\dmi}\sqrt{t'+1}.
                \end{displaymath}
        \end{lemma}
        \begin{proof}
            a) Set $\alpha_0\coloneqq \frac{5\dav n}{\dmi}\cdot\frac{1}{\sqrt{t'+1}}$. We claim $\alpha_0>1$. In fact, we have 
                    \begin{displaymath}\begin{split}
                        \alpha_0
                        &=\frac{5\dav n}{\dmi}\cdot\frac{1}{\sqrt{t'+1}}
                        =
                        \frac{5\dav n}{\dmi}\cdot\frac{1}{\sqrt{2\lceil\trel/2\rceil-1}}\\
                        &\geqslant
                        \frac{5\dav n}{\dmi}\cdot\frac{1}{\sqrt{\trel+2-1}}
                        >
                        \frac{5\dav n}{\dmi}\cdot\frac{1}{\sqrt{\frac{24\dav n^2}{\dmi}+2-1}}\\
                        &\geqslant
                        \frac{5\dav n}{\dmi}\cdot\frac{1}{\sqrt{\frac{25\dav n^2}{\dmi}}}
                        =\sqrt{\frac{\dav}{\dmi}}\geqslant 1.
                    \end{split}\end{displaymath}
                Therefore, $\alpha_0>1$. As a consequence, $A_{\alpha_0}$ is non-empty.
            \par
            b)  If $x\in A_{\alpha_0}$, by the definition of $A_{\alpha_0}$,
                \begin{displaymath}
                    \frac{g_{t'}(x,x)}{\pi(x)}\leqslant \alpha_0(t'+1)=\frac{5\dav n}{\dmi}\sqrt{t'+1}.
                \end{displaymath}
            \par
            c) If $x\notin A_{\alpha_0}$, by the strong Markov property, \Cref{prp32}, and the definition of $A_{\alpha_0}$,
                \begin{displaymath}\begin{split}
                    \frac{g_{t'}(x,x)}{\pi(x)}
                    &\leqslant \frac{G(x,x;A_{\alpha_0})}{\pi(x)} + \Esubbigg{x}{\frac{g_{t'}(X_{\tau_{A_{\alpha_0}}},x)}{\pi(x)}}\\
                    &\leqslant \frac{9}{2}\cdot\Bigl(\frac{\dav n}{\dmi}\Bigr)^2\bigl(1-\pi(A_{\alpha_0})\bigr)+\alpha_0(t'+1)\\
                    &\leqslant \frac{9}{2}\cdot\Bigl(\frac{\dav n}{\dmi}\Bigr)^2\frac{1}{\alpha_0}+\alpha_0(t'+1)\\
                    &= (\tfrac{9}{10}+5)\sqrt{t'+1}\frac{\dav n}{\dmi}\\
                    &<\frac{6\dav n}{\dmi}\sqrt{t'+1},
                \end{split}\end{displaymath}
                where the third inequality is by \Cref{alpha}.
        \end{proof}
        We are now in position to prove \Cref{reg_return_imp}.
        \begin{proof}[Proof of \Cref{reg_return_imp}]
            a) Recalling that $t'=2\lceil\trel/2\rceil-2$, we have $t'\leqslant \trel$. 
            By \Cref{lm31,lm32}, for $t\equiv 0 \bmod 2$ and $x\in V$,
                \begin{displaymath}\begin{split}
                    0&\leqslant p_{t}(x,x)-\pi(x)\leqslant \frac{1}{1-\ue^{-1}}\cdot\frac{g_{t'}(x,x)}{\tfrac{t}{2}+1}\\
                    &\leqslant
                    \frac{6\dav n\pi(x)}{(1-\ue^{-1})(\tfrac{t}{2}+1)\dmi}\sqrt{t'+1}
                    = \frac{6d(x)}{(1-\ue^{-1})(\tfrac{t}{2}+1)\dmi}\sqrt{t'+1}\\
                    &\leqslant \frac{6d(x)}{(1-\ue^{-1})(\tfrac{t}{2}+1)\dmi}\sqrt{\trel+1}
                    \leqslant \frac{10d(x)}{\dmi}\frac{\sqrt{\trel+1}}{\tfrac{t}{2}+1}\\
                    &=\frac{20d(x)\sqrt{\trel+1}}{(t+2)\dmi}.
                \end{split}\end{displaymath}
            Therefore, for $t\equiv 0 \bmod 2$ and $x\in V$, we have
                \begin{displaymath}
                    0\leqslant p_{t}(x,x)-\pi(x)\leqslant\frac{20d(x)\sqrt{\trel+1}}{(t+1)\dmi}.
                \end{displaymath}
            \par
            b) Our calculation in part a) implies that for $t\equiv 0 \bmod 2$,
                \begin{displaymath}
                    p_{t}(x,x)-\pi(x)=\int_{(-1,1)}\lambda^{t}\,\norm{I_P(\ud\lambda)\be_x}_w^2
                    \leqslant \frac{20d(x)\sqrt{\trel+1}}{(t+2)\dmi}.
                \end{displaymath}
                Therefore, for $t\equiv 1 \bmod 2$, we have
                \begin{displaymath}\begin{split}
                    \abs[\big]{p_{t}(x,x)-\pi(x)}
                    &=\abs[\Big]{\int_{(-1,1)}\lambda^{t}\,\norm{I_P(\ud\lambda)\be_x}_w^2}
                    \leqslant \int_{(-1,1)}\abs{\lambda}^{t}\,\norm{I_P(\ud\lambda)\be_x}_w^2\\
                    &= \int_{(-1,1)}\abs{\lambda}^{(t-1)/2}\abs{\lambda}^{(t+1)/2}\,\norm{I_P(\ud\lambda)\be_x}_w^2.
                \end{split}\end{displaymath}
                Hence, the Cauchy--Schwarz inequality gives
                \begin{displaymath}\begin{split}
                    \abs[\big]{p_{t}(x,x)-\pi(x)}
                    &\leqslant \sqrt{\int_{(-1,1)}\abs{\lambda}^{t-1}\,\norm{I_P(\ud\lambda)\be_x}_w^2}
                        \sqrt{\int_{(-1,1)}\abs{\lambda}^{t+1}\,\norm{I_P(\ud\lambda)\be_x}_w^2}\\
                    &\leqslant \frac{20d(x)\sqrt{\trel+1}}{\sqrt{(t-1+2)(t+1+2)}\,\dmi}
                    <\frac{20d(x)\sqrt{\trel+1}}{(t+1)\dmi}.\qedhere
                \end{split}\end{displaymath}
        \end{proof} 

\printshorthands\addcontentsline{toc}{section}{{Abbreviations}}
\printbibliography[heading = bibintoc]

@article{Moon-1965,
   author = {Moon, J. W.},
   title = {On the diameter of a graph},
   journal = {Michigan Mathematical Journal},
   volume = {12},
   number = {3},
   pages = {349-351},
   doi = {10.1307/mmj/1028999370},
   year = {1965},
   type = {Journal Article}
}

@article{Liu-2015,
   author = {Liu, Shiping},
   title = {Multi-way dual Cheeger constants and spectral bounds of graphs},
   journal = {Advances in Mathematics},
   volume = {268},
   number = {2},
   pages = {306-338},
   ISSN = {00018708},
   DOI = {10.1016/j.aim.2014.09.023},
   year = {2015},
   type = {Journal Article}
}

@article{Trevisan-2012,
   author = {Trevisan, Luca},
   title = {Max Cut and the Smallest Eigenvalue},
   journal = {SIAM Journal on Computing},
   volume = {41},
   number = {6},
   pages = {1769-1786},
   ISSN = {0097-5397
1095-7111},
   DOI = {10.1137/090773714},
   year = {2012},
   type = {Journal Article}
}

@article{Alon-2000,
   author = {Alon, Noga and Sudakov, Benny},
   title = {Bipartite Subgraphs And The Smallest Eigenvalue},
   journal = {Combinatorics, Probability and Computing},
   volume = {9},
   number = {1},
   pages = {1--12},
   ISSN = {09635483},
   DOI = {10.1017/S0963548399004071},
   url = {https://www.cambridge.org/core/journals/combinatorics-probability-and-computing/article/abs/bipartite-subgraphs-and-the-smallest-eigenvalue/B8556CD8A8819800423D0D9A8B6EC6EB},
   year = {2000},
   type = {Journal Article}
}

@book{Brouwer-2011,
   author = {Brouwer, Andries E.  and Haemers, Willem H. },
   title = {Spectra of Graphs},
   publisher = {Springer-Verlag New York},
   edition = {1},
   series = {Universitext},
   ISBN = {978-1-4614-1938-9
978-1-4614-1939-6},
   DOI = {10.1007/978-1-4614-1939-6},
   year = {2012},
   type = {Book}
}

@article{Cvetkovic-2009,
   author = {Cvetković, Dragoš and Simić, Slobodan K.},
   title = {Towards A Spectral Theory Of Graphs Based On The Signless {L}aplacian, {I}},
   journal = {Publications de l'Institut Mathematique},
   volume = {85},
   number = {99},
   pages = {19--33},
   ISSN = {0350-1302
1820-7405},
   DOI = {10.2298/pim0999019c},
   url = {http://www.doiserbia.nb.rs/img/doi/0350-1302/2009/0350-13020999019C.pdf},
   year = {2009},
   type = {Journal Article}
}

@article{Cvetkovic-2010a,
   author = {Cvetković, Dragoš and Simić, Slobodan K.},
   title = {Towards A Spectral Theory Of Graphs Based On The Signless {L}aplacian, {II}},
   journal = {Linear Algebra and its Applications},
   volume = {432},
   number = {9},
   pages = {2257--2272},
   ISSN = {00243795},
   DOI = {10.1016/j.laa.2009.05.020},
   url = {https://www.sciencedirect.com/science/article/pii/S0024379509002808},
   year = {2010},
   type = {Journal Article}
}

@article{Cvetkovic-2010b,
   author = {Cvetković, Dragoš and Simić, Slobodan K.},
   title = {Towards A Spectral Theory Of Graphs Based On The Signless {L}aplacian, {III}},
   journal = {Applicable Analysis and Discrete Mathematics},
   volume = {4},
   number = {1},
   pages = {156--166},
   ISSN = {14528630, 2406100X},
   DOI = {10.2298/AADM1000001C},
   url = {http://www.jstor.org/stable/43671298},
   year = {2010},
   type = {Journal Article}
}

@article{Davis-1995,
   author = {Davis, Burgess and McDonald, David},
   title = {An Elementary Proof of The Local Central Limit Theorem},
   journal = {Journal of Theoretical Probability},
   volume = {8},
   number = {3},
   pages = {693--701},
   ISSN = {1572-9230},
   DOI = {10.1007/BF02218051},
   url = {https://doi.org/10.1007/BF02218051
https://link.springer.com/article/10.1007/BF02218051},
   year = {1995},
   type = {Journal Article}
}

@article{Desai-1994,
   author = {Desai, Madhav and Rao, Vasant},
   title = {A Characterization of the Smallest Eigenvalue of A Graph},
   journal = {Journal of Graph Theory},
   volume = {18},
   number = {2},
   pages = {181--194},
   ISSN = {03649024},
   DOI = {10.1002/jgt.3190180210},
   url = {https://onlinelibrary.wiley.com/doi/abs/10.1002/jgt.3190180210},
   year = {1994},
   type = {Journal Article}
}

@article{Desai-1993,
   author = {Desai, Madhav and Rao, Vasant},
   title = {On the Convergence of Reversible {M}arkov Chains},
   journal = {SIAM Journal on Matrix Analysis and Applications},
   volume = {14},
   number = {4},
   pages = {950--966},
   ISSN = {0895-4798
1095-7162},
   DOI = {10.1137/0614063},
   year = {1993},
   type = {Journal Article}
}

@article{Diaconis-1991,
   author = {Diaconis, Persi and Stroock, Daniel},
   title = {Geometric Bounds For Eigenvalues Of {M}arkov Chains},
   journal = {The Annals of Applied Probability},
   volume = {1},
   number = {1},
   pages = {36--61},
   DOI = {10.1214/aoap/1177005980},
   url = {https://doi.org/10.1214/aoap/1177005980},
   year = {1991},
   type = {Journal Article}
}

@article{Landau-1981,
   author = {Landau, Henry J. and Odlyzko, Andrew M.},
   title = {Bounds for Eigenvalues of Certain Stochastic Matrices},
   journal = {Linear Algebra and its Applications},
   volume = {38},
   pages = {5--15},
   ISSN = {00243795},
   DOI = {10.1016/0024-3795(81)90003-3},
   year = {1981},
   type = {Journal Article}
}

@book{Levin-2017,
   author = {Levin, David A and Peres, Yuval},
   title = {{M}arkov Chains and Mixing Times},
   publisher = {American Mathematical Society},
   series = {AMS Non-Series Monographs},
   volume = {107},
   ISBN = {1470429624},
   DOI = {10.1090/mbk/107},
   year = {2017},
   type = {Book}
}

@article{Lyons-2018,
   author = {Lyons, Russell and Oveis Gharan, Shayan},
   title = {Sharp Bounds On Random Walk Eigenvalues Via Spectral Embedding},
   journal = {International Mathematics Research Notices},
   volume = {2018},
   number = {24},
   pages = {7555--7605},
   ISSN = {1073-7928
1687-0247},
   DOI = {10.1093/imrn/rnx082},
   url = {https://doi.org/10.1093/imrn/rnx082
https://watermark.silverchair.com/rnx082.pdf?token=AQECAHi208BE49Ooan9kkhW_Ercy7Dm3ZL_9Cf3qfKAc485ysgAAAqIwggKeBgkqhkiG9w0BBwagggKPMIICiwIBADCCAoQGCSqGSIb3DQEHATAeBglghkgBZQMEAS4wEQQMhQRrmRLL7Ch-IWqgAgEQgIICVVBkabO29YTtlbfXnBznTpzktIa8EOlnOCDVjjfOAAwxPjm9os7k_UiZAa9RiqY7gSMTkNQfyIPwhwFzUMRcV9icWDnBom-jnxOYOqRBa4_dT8DRNHsPX8nh59yslE6RNiFX51in32G6DNSVD6-igyrQqfGhiLYonRuA-lm_dn3mDD4sjP6ZAyC0sls_dqzUmx-7J2NU62k9NYiZt2FmM0Zuw1reqbhiL4tDyEcZflzA7_xOfG8azeIg45-hZwWbuy7na_ZN2aOo6bu5n7eosjYTiJrEmp6i2pw7sfX9RjLb9fN854ID1EDI-XYYjZVsnYUqvcUQuwU7lOeDwY3FzLHHJ55UKG0xlvgH4o8u4bArkOEJCbCB3UE6gm8k_5Z1AogTL-mc1wKPG20Kthc0URW5sfrkLBabG-YpDuWaIgtZMevn2dfJuFr4mRW5sTXI8_EIEGG852FbYE_yRHNOJcNR59IobGN0__yOYRIkdd0mukGHd-SMO8p40H0YF25MSm-6_3znEP2rd-iJWl67BI7eYGlKbD41DeQWO07Zpt-H-MNenD7AFcGU8RYdDbLdOPxvYjq5SQLqg8Wk8-mt-jLu9-MYx-BkdEzssL4snAbDaT3BVeaqlTGNbg_FPWWxxzeA8nCRC8BQhHHCC7AuoP58aSGIzZkyhljmrS3L5J_opLtlbAbbDhf1LcFMWH_TqTZp7VGaBtkPW8AoIYs-70M9mMRE_mItSS24_1xthq8bpPqhyAP61RFLGVUPB1xlnGGqb3Hj8KHCfoKyr9NZBoc0N6uZ2g
https://watermark.silverchair.com/rnx082.pdf?token=AQECAHi208BE49Ooan9kkhW_Ercy7Dm3ZL_9Cf3qfKAc485ysgAAAqMwggKfBgkqhkiG9w0BBwagggKQMIICjAIBADCCAoUGCSqGSIb3DQEHATAeBglghkgBZQMEAS4wEQQMpTy2k5dkp22jsiDEAgEQgIICVvOizAs8Wjga_dE8dOauC04x1V7oB1Y7k1XQpn-r8Dofz_2ujgHm2jHaeFwb7m9KZpUS2uWj7pXBgZCzwCp5WXwlDPlS3OGyJcwY3kxWVbPNspZophwj0JD9AyyLr6Y5_oF6lk4UTYJo4znp6UvFG8Hhu98tE_EnY_nGXRK1CI_KCNr5SFTXs14LKrpCKFbFZXnVbLMM-qwrYFRTOEs58I4M66y2nCDZqLFzAuF4rCGHxihWh0tBftZLuUEYa7CpTLfJuU1Nh5JyCBnJBigwhoiAwF7-IkOyVPZgKabDOY9ZcB269cb0bYPACdWTygYCXFod8kTanwHeHW576_GPMvauJhAod3jnS9Kupr_p9adgLI6QPF0622O60YsvWzIx9KjO-bCAGBxqnEhVKZFfHVIV---dNXFNhYx82Bf10xXilKg8Xjaht01DyRrFkqTK9XilujJH6bRmQzaWkhekgbARMt6u-BKwUFHNgnYx8T0YHBPA1Snyc5iYoqkftzZFyX3IfBywKJkigYjDOEyuvRu0DoCnIYLjfFIP8Tis9x3F9szpHOi4BOdH5cBxW0XC9CQLTCbhwyElaT1CKbuu0wxnsi7R1gxR2zNHaQe6AVK4LlMa6ybJ_4Z55jCKW-QWEmVjnPPASJDDmZkPaFOsn4dbmwefF9X1TXPZX3jO5GX1mtR9zB9J_OIEjkhCIdERWG5DxXYkvfzurmtpMybgbIhXvBpztvmKZ78SAzLOSzUEHWeC7-GA-mwgQVgW-OsBOB8-eutF7W7kbCWoVe1D8x5U_i9HZX8},
   year = {2017},
   type = {Journal Article},
   shorthand = {LOG17}
}

@misc{Lyons-2021,
   author = {Lyons, Russell and Oveis Gharan, Shayan},
   title = {Errata to ``Sharp Bounds on Random Walk Eigenvalues via Spectral Embedding''},
   note = {Available at \url{https://rdlyons.pages.iu.edu}},
   year = {2021},
   type = {Manuscript}
}

@book{Lyons-2017,
   author = {Lyons, Russell and Peres, Yuval},
   title = {Probability on Trees and Networks},
   publisher = {Cambridge University Press, New York},
   volume = {42},
   series = {Cambridge Series in Statistical and Probabilistic Mathematics},
   %note = {Available at \url{https://rdlyons.pages.iu.edu/}},
   ISBN = {9781316672815},
   DOI = {10.1017/9781316672815},
   url = {http://dx.doi.org/10.1017/9781316672815},
   year = {2017},
   type = {Book}
}

@article{Mao-2013,
   author = {Mao, Yong-Hua and Song, Yan-Hong},
   title = {Spectral gap and convergence rate for discrete-time {M}arkov chains},
   journal = {Acta Mathematica Sinica, English Series},
   volume = {29},
   number = {10},
   pages = {1949--1962},
   ISSN = {1439-8516
1439-7617},
   DOI = {10.1007/s10114-013-2594-1},
   year = {2013},
   type = {Journal Article}
}

@article{Mohar-1989,
   author = {Mohar, Bojan and Woess, Wolfgang},
   title = {A Survey on Spectra of Infinite Graphs},
   journal = {Bulletin of the London Mathematical Society},
   volume = {21},
   number = {3},
   pages = {209--234},
   ISSN = {00246093},
   DOI = {10.1112/blms/21.3.209},
   url = {https://londmathsoc.onlinelibrary.wiley.com/doi/abs/10.1112/blms/21.3.209},
   year = {1989},
   type = {Journal Article}
}

@inproceedings{Oliveira-2019,
   author = {Oliveira, Roberto and Peres, Yuval },
   title = {Random walks on graphs: new bounds on hitting, meeting, coalescing and
returning},
   booktitle = {2019 Proceedings of the Sixteenth Workshop on Analytic Algorithmics and Combinatorics (ANALCO)},
   publisher = {SIAM},
   pages = {119--126},
   DOI = {10.1137/1.9781611975505.13},
   year = {2019},
   type = {Conference Proceedings}
}

@article{Tetali-2005,
   author = {Tetali, Prasad and Montenegro, Ravi},
   title = {Mathematical Aspects of Mixing Times in {M}arkov Chains},
   journal = {Foundations and Trends in Theoretical Computer Science},
   volume = {1},
   number = {3},
   pages = {237--354},
   ISSN = {1551-305X
1551-3068},
   DOI = {10.1561/0400000003},
   year = {2005},
   type = {Journal Article}
}

@article{Luxburg-2007,
   author = {\lowercase{von} Luxburg, Ulrike},
   title = {A Tutorial On Spectral Clustering},
   journal = {Statistics and Computing},
   volume = {17},
   number = {4},
   pages = {395--416},
   ISSN = {1573-1375},
   DOI = {10.1007/s11222-007-9033-z},
   url = {https://doi.org/10.1007/s11222-007-9033-z
https://link.springer.com/article/10.1007/s11222-007-9033-z
https://link.springer.com/article/10.1007%2Fs11222-007-9033-z},
   year = {2007},
   type = {Journal Article}
}
\end{document}